\documentclass[epsfig,amsfonts]{amsart}
\usepackage{color}
\usepackage{epsfig}
\usepackage{amsmath}
\usepackage{amssymb}
\usepackage{amscd}
\usepackage{graphicx}

\newcommand{\EE}{\ensuremath{\mathbb{E}}}

\newcommand{\NN}{\ensuremath{\mathbb{N}}}

\newcommand{\PP}{\ensuremath{\mathbb{P}}}

\newcommand{\RR}{\ensuremath{\mathbb{R}}}

\newcommand{\eps}{\ensuremath{\epsilon}}

\newcommand{\bB}{\ensuremath{\mathcal{B}}}

\newcommand{\dD}{\ensuremath{\mathcal{D}}}

\newcommand{\oO}{\ensuremath{\mathcal{O}}}

\newcommand{\tT}{\ensuremath{\mathcal{T}}}

   \newtheorem{lemma}{Lemma}[section]
   \newtheorem{theorem}[lemma]{Theorem}
   \newtheorem{remark}[lemma]{Remark}

   \newtheorem{coro}[lemma]{Corollary}
   \newtheorem{definition}[lemma]{Definition}

\numberwithin{equation}{section}

\renewcommand{\phi}{\varphi}

\begin{document}

\author{Mar\'{\i}a J. Garrido-Atienza}
\address[Mar\'{\i}a J. Garrido-Atienza]{Dpto. Ecuaciones Diferenciales y An\'alisis Num\'erico\\
Universidad de Sevilla, Apdo. de Correos 1160, 41080-Sevilla,
Spain} \email[Mar\'{\i}a J. Garrido-Atienza]{mgarrido@us.es}

\author{Kening Lu}
\address[Kening Lu]{346 TMCB\\
Brigham Young University, Provo, UT 84602, USA} \email[Kening
Lu]{klu@math.byu.edu}

\author{Bj{\"o}rn Schmalfu{\ss }}
\address[Bj{\"o}rn Schmalfu{\ss }]{Institut f\"{u}r Stochastik\\
Friedrich Schiller Universit{\"a}t Jena, Ernst Abbe Platz 2, 77043\\
Jena,
Germany\\
 }
\email[Bj{\"o}rn Schmalfu{\ss }]{bjoern.schmalfuss@uni-jena.de}

\keywords{Stochastic PDEs, fractional Brownian motion, pathwise solutions, rough path theory. \\
This work was partially supported by MTM2011-22411, FEDER founding (M.J.
Garrido-Atienza and B. Schmalfu{\ss}), and by NSF0909400 (K. Lu). }

\title[Pathwise solutions to stochastic partial differential equations]
{Pathwise solutions to stochastic partial differential equations driven by fractional Brownian motions with Hurst parameters in $(1/3,1/2]$}

\begin{abstract}
Combining fractional calculus and the Rough Path Theory we study the existence and uniqueness of mild solutions to evolutions equations driven by a H\"older continuous function with H\"older exponent in $(1/3,1/2)$. Our stochastic integral is in some sense a generalization of the well-known Young integral and can be defined independently of the initial condition. Similar to the Rough Path Theory we establish a second variable which is given, roughly speaking, by a tensor product. It is then necessary to formulate a second equation for this new variable, and we do in a mild sense. The crucial point in order to get this new equation is to construct a tensor depending on the noise path but also on the semigroup. We then prove the existence of a unique H{\"o}lder continuous solution of the system of equations, consisting of the path and the area components, if the nonlinear term and the initial condition are sufficiently smooth. Once the abstract theory is developed, we can present a pathwise nonlinear SPDE driven by a fractional Brownian motion with Hurst parameter in $(1/3,1/2]$.
\end{abstract}
\maketitle


\section{Introduction}

In this article, we are interested in developing a pathwise theory for the following type of stochastic evolution equations
\begin{align}\label{eqn1}
\left\{
\begin{array}{ll}
du(t) &= Au(t) dt  + G(u(t))d\omega(t),\\
u(0) &= u_0 \in V_\delta,
\end{array}\right.
\end{align}
in a Hilbert space $V$, where $\omega$ is a H\"older continuous function with
H\"older exponent $\beta\in(1/3, 1/2)$, $A$ is the infinitesimal
generator of an analytic semigroup $S(\cdot)$ on $V$,  $G$ is a nonlinear term
satisfying certain assumptions which will be described in the
next sections,  and $V_\delta=D((-A)^\delta)$ for adequate $\delta>0$.
As a particular case of driving noises we  can consider a
fractional Brownian motion (fBm) $B^H$ with Hurst parameter $H\in (1/3,1/2]$. In general, an fBm $B^H$ with Hurst parameter $H\in (0,1)$ is
a stochastic process which differs significantly from the standard
Brownian motion and, in particular, when $H\not =1/2$ is not a martingale, so the Ito integrals cannot be used to describe integration when having this type of integrators.
Our interpretation of pathwise is that
we obtain a solution of these stochastic equations which does not produce exceptional sets depending on the initial conditions. In the classical theory of stochastic evolution equations ($H=1/2)$, see for instance the monograph
by Da Prato and Zabczyk \cite{DaPrato}, stochastic integrals are constructed to be a limit in probability of particular random variables defined only almost surely, where the exceptional sets may depend on the initial conditions. Pathwise results for that classical theory are only available for the white noise case ($G={\rm id}$) and a few special cases when $u\mapsto G(u)$ is  linear.

During the last 15 years it can be found in the literature several attempts to develop a theory for stochastic integration for integrators which are not given  by a Wiener process, and in particular, for the fractional Brownian motion $B^H$. One of these attempts is given by the {\em Rough Path Theory}, and we refer to the monographs by Lyons and Qian \cite{Lyons} and Friz and Victoir \cite{FV10} for a comprehensive presentation of this theory.

A different technique called {\em Fractional Calculus} was developed by Z{\"a}hle  \cite{Zah98}, who considered for a fractional Brownian motion with $H>1/2$ the well-known Young integral.
In contrast to the Ito-or Stratonovich integral, that integral can be  defined in a {\em pathwise sense}. In particular, that integral is given by  fractional derivatives, which allow a pathwise estimate of the integrals in terms of integrand and integrator using special norms, see also Nualart and R{\u{a}}{\c{s}}canu \cite{NuaRas02}. In this article
the authors are able to show the existence and uniqueness of the solution of a finite-dimensional stochastic differential equation driven by a fractional  Brownian motion for $H>1/2$. These results were extended by Maslowski and Nualart \cite{MasNua03} to show the existence of mild solutions for stochastic  evolution equations. In particular, the mild solution exists for {\em any}  H{\"o}lder continuous noise path with H{\"o}lder exponent  larger than 1/2 if the coefficients are sufficiently smooth and the linear part of the equation generates an analytic semigroup.

In addition to the mentioned monographs \cite{FV10} and \cite{Lyons},  other references using the rough path theory to solve stochastic differential equations are Anh and Grecksch \cite{AnhGre99}, Coutin {\it et al.} \cite{coutin} and Lejay \cite{Lejay}, to name only a few of them. However, to the best of our knowledge, the literature concerning the study of stochastic partial differential equations by using the rough path theory is not so extensive. In this group interesting articles are Caruana and Friz \cite{caruana}, Caruana {\it et al.} \cite{caruana2}, Friz and Oberhauser \cite{FO11}, Gubinelli {\it et al.} \cite{GuLeTin}, Hinz and Z\"ahle \cite{HinZah09}, and the recent papers by Deya {\it et al.} \cite{DGT}  and by Gubinelli and Tindel \cite{GuTin}. In particular, in this last paper the authors proved local mild solutions of stochastic PDEs driven by rough paths  for $\beta$-H{\"o}lder continuous paths ($\beta\in (1/3,1/2]$) with a special quadratic nonlinearity. The moving frame ansatz has been used
by Teichmann \cite{Tmann} to consider Rough-Path-differential equations. \\

Our idea to face the study of mild solutions for \eqref{eqn1} is to combine the rough path theory with classical tools, inspired by the work by Hu and Nualart \cite{HuNu09}.  In order to do that, they use a special fractional derivative, the so-called {\em compensated  }fractional derivative. Thanks to that, they are able to formulate an existence and  uniqueness result for finite-dimensional stochastic differential  equations having coefficients which are sufficiently smooth. We want to stress that, to formulate an operator equation solving this problem, they need a second equation for the so-called {\em area} in the space of tensors.

We adapt the techniques in \cite{HuNu09} to obtain mild solutions for our infinite-dimensional stochastic evolution equation (\ref{eqn1}), assuming that the linear part generates an analytic  semigroup $S$ on a separable Hilbert space. However,
there are significant differences between our setting and the one in  \cite{HuNu09}.
We consider mild solutions for the trajectories part of the evolution equations.  As we will see, we have to build the right fixed point argument in order to solve our equation: we will be able to obtain solutions in small intervals that later on can be concatenated to turn out the mild solution over any interval.
And due to the appearance of the term $S(t)u_0$, which does not appear in the finite dimensional situation, we have to adapt any of these intervals to the size of the initial condition.
Moreover, one has to face with the correct area equation in this infinite-dimensional setting. In a first part of this article we will consider that our evolution equation is driven by a regular noise path, which makes its analysis easier. For this purpose we construct an area object $\omega \otimes_S \omega$ depending on the noise path $\omega$ as well as on the semigroup $S$. This term is studied using tools from the fractional calculus theory. According to these tools we are able to state some useful properties for this area, for instance the Chen equality. We would like to stress that in \cite{DGT} the authors refer to the twisted iterated integral, which would be the counterpart of our tensor $\omega \otimes_ S\omega$, but some problems defining this area are mentioned in Remark 4.3 of this article. \\

The problem of showing that solutions of stochastic evolution equations generate random dynamical systems is unsolved even for the stochastic partial differential equations  driven by the standard Brownian motion. The main difficulties are (i) the stochastic integral is only defined almost surely where the exceptional set may depend on the initial state; and (ii) Kolmogorov's theorem, as cited in Kunita \cite{Kunita90} Theorem 1.4.1,  is only applicable for finite dimensional random fields. However, there are some partial results for additive as well as multiplicative noise (see for example, \cite{FL}, \cite{DLS} \cite{DLS2}, \cite{CarKloSchm03} and \cite{MZZ}). In the recent work \cite{GLS09}, under appropriate conditions on $A$, $F$ and $G$, it has been shown that the stochastic partial differential equation \eqref{eqn1} above driven by a fBm with Hurst parameter  $H \in (1/2, 1)$  generates a random dynamical system.

Thanks to the pathwise results that we are going to establish in this article, in a forthcoming article we will be able to go one step further with respect to this unsolved problem, since we will prove that stochastic evolution equations like (\ref{eqn1}), driven by a fBm $B^H$ with $H\in (1/3,1/2]$, generate random dynamical systems.\\

The article is organized as follows. In Section 2 we collect the main tools and give the main assumptions.
In particular, we mention important properties of the fractional  derivatives. In Section 3 we consider the evolution equation
for a {\em smooth} noise path. That analysis
provides us a meaningful definition of solutions of \eqref{eqn1}.
We rewrite the equation using  fractional derivatives and we then show that we need to present
a second equation for the area part. To do this, as we have announced previously, we  introduce a tensor
defined by the semigroup generated by $A$, being the crucial step for our  considerations, and which constitutes the main difference with respect to the finite-dimensional case developed in \cite{HuNu09}. In Section 4 we give the definition of a solution for our system consisting of a path- and an area variables,  and an existence and uniqueness theorem is formulated. Since we consider regular non-linear operators $G$, this allows to prove a {\em global} existence result.  We also present two examples to show nonlinearities $G$ that fit the abstract theory. Some technical proofs are shifted to the Appendix Section.\\

We would like to point out that we could also add a nonlinear diffusion term $F$ on the right-hand side of the equation in (\ref{eqn1}). Nevertheless, to simplify the whole presentation we have not considered it since the $dt$-nonlinearity is not the interesting problem in the paper.\\

Finally, we refer to \cite{GLS12note} for a short and recent announcement of our results.

\section{Preliminaries}\label{preli}

Let in general $V, \, \tilde V,\, \hat V$ be separable Hilbert spaces. We denote by $L(V,\tilde V)$ the Banach space of linear operators from $V$ to $\tilde V$ and by $L_2(V,\tilde V)\subset L(V,\tilde V)$ the space of Hilbert--Schmidt operators, which is a separable Hilbert space. For $T\in L(\tilde V, \hat V),\, G\in L_2(V,\tilde V)$ (and vice versa for $T\in L_2(\tilde V, \hat V),\, G\in L(V,\tilde V)$) we have that

\begin{equation*}
    \|TG\|_{L_2(V,\hat V)}\le \|T\|_{L(\tilde V,\hat V)}\|G\|_{L_2(V,\tilde V)},\quad
     \|TG\|_{L_2(V,\hat V)}\le \|T\|_{L_2(\tilde V,\hat V)}\|G\|_{L(V,\tilde V)}.
\end{equation*}

Consider now the separable Hilbert space  $(V, |\cdot|,(\cdot,\cdot))$ and assume that $S(\cdot)$ is an analytic semigroup on $V$ generated by an operator $A$. We also assume that $A$ is a strictly negative operator  generating a complete orthonormal basis $(e_i)_{i\in\mathbb{N}}$ in $V$.  Let $D((-A)^\delta)$, $\delta \in \RR$ denote the domain of the fractional power $(-A)^{\delta}$ equipped with the graph norm $|x|_{D((-A)^{\delta})}:=|(-A)^{\delta}x|$, see Pazy \cite{Pazy} Section 2.6.

\vskip.4cm
For any $t>0$ the
following inequalities hold
\begin{align}
  &\|S(t)\|_{L(D((-A)^{\delta}), D((-A)^{\gamma}))}=  \|(-A)^\gamma S(t)\|_{L(D((-A)^{\delta}),V)}\le ct^{\delta-\gamma},\quad\text{for }
  \gamma>\delta\ge0,
  \label{eq1} \\
 &\|S(t)-{\rm
id}\|_{L(D((-A)^{\sigma}),D((-A)^{\theta}))} \le c
t^{\sigma-\theta}, \quad \text{for } \sigma-\theta\in (0,1], \label{eq2}
\end{align}
since $S(\cdot)$ is an analytic semigroup, see \cite{Chu02}, Page 84 (it should be also taken into account that $(-A)^\mu$ is an isomorphism from $D((-A)^{\delta+\mu})$ to $D((-A)^{\delta})$.
\vskip.4cm

In addition, the following crucial properties, which proofs are immediate consequences of the previous inequalities, are satisfied for
any analytic semigroup $S(\cdot)$:
\begin{lemma}\label{l0}
For any $0<\nu \le 1$, $0<\beta\leq 1$, $0\le\delta <\gamma+\nu\in [0,1)$, $0\le \rho$ there exists a constant $c>0$ such that for $0<q<r< s<t$ we have that
\begin{align*}
\|S(&t-r)-S(t-q)\|_{L(D((-A)^{\delta}),D((-A)^{\gamma}))}\le c(r-q)^\nu(t-r)^{-\nu-\gamma+\delta},\\
\|S(&t-r)- S(s-r)-S(t-q)+S(s-q)\|_{L(D((-A)^{\rho}),D((-A)^{\rho}))}\leq
c(t-s)^{\beta}(r-q)^{\nu}(s-r)^{-(\nu+\beta)}.
\end{align*}
\end{lemma}

In what follows, let us abbreviate $V:=V_0,\, V_\delta:=D((-A)^{\delta})$ with norm $|\cdot|_{V_\delta}$. If $(\lambda_i)_{i\in \mathbb N}$ denotes the spectrum of $A$, since $(e_i)_{i\in\mathbb{N}}$ is an orthonormal basis in $V$, it follows that $(e_i/\lambda_i^\delta)_{i\in\mathbb{N}}$ is an orthonormal basis in $V_\delta$.

We also need the following estimates concerning the drift $G$ in system \eqref{eqn1}.
\begin{lemma}\label{l2}
Suppose that $G: V\to L_2(V, \tilde V)$ is bounded and twice continuously
Fr\'echet--differentiable with bounded first and second
derivatives $DG(u)$ and $D^2G(u)$, for $u\in V$. Let us denote, respectively, by $c_G$, $c_{DG},\, c_{D^2G}$
the bounds for $G$, $DG$ and $D^2G$.
Then, for $u_1,\,u_2,\,v_1,\,v_2\in V$, we have
\begin{align*}
&\|G(u_1)\|_{L_2(V, \tilde V)}\le c_G,\\
&\|G(u_1)-G(v_1)\|_{L_2(V, \tilde V)}\le c_{DG}|u_1-v_1|,\\
&\|DG(u_1)-DG(v_1)\|_{L_2{(V\times V,\tilde V)}}\le c_{D^2G}|u_1-v_1|,\\
&\|G(u_1)-G(u_2)-DG(u_2)(u_1-u_2)\|_{L_2(V, \tilde V)}\le c_{D^2G}|u_1-u_2|^2,\\
&\|G(u_1)-G(v_1)-(G(u_2)-G(v_2))\|_{L_2(V, \tilde V)}\\
& \quad \le c_{DG}|u_1-v_1-(u_2-v_2)|+c_{D^2G} |u_1-u_2|(|u_1-v_1|+|u_2-v_2|).
\end{align*}
\end{lemma}
These estimates follows by the mean value theorem. For the last
estimate we refer to Nualart and R{\u{a}}{\c{s}}canu \cite{NuaRas02} Lemma
7.1.

Notice that, in particular, $DG: V\to L_2(V,L_2(V, \tilde V))$ (or equivalently,  $DG: V\to L_2(V\times V,\tilde V)$, where this last space is defined below) is a bilinear map
and similarly $D^2G(u)$ is a trilinear map.

\begin{lemma}\label{l3}
In addition to the assumptions of Lemma \ref{l2}, suppose that $G$
is three times continuously Fr\'echet--differentiable where the
third derivative is uniformly bounded in $L_2(V,\tilde V)$. This bound is denoted by
$c_{D^3G}$. Then, for $u_1,\,u_2,\,v_1,\,v_2\in V$, we have
\begin{align*}
    &\|DG(u_1)-DG(v_1)-(DG(u_2)-DG(v_2))\|_{L_2(V\times V,\tilde V)}\\
& \quad \le c_{D^2G}|u_1-v_1-(u_2-v_2)|+c_{D^3G} |u_1-u_2|(|u_1-v_1|+|u_2-v_2|),\\
&\|G(u_1)-G(u_2)-DG(u_2)(u_1-u_2)-(G(v_1)-G(v_2)-DG(v_2)(v_1-v_2))\|_{L_2(V, \tilde V)}\\
& \quad \le c_{D^2G}
    (|u_1-u_2|+|v_1-v_2|)|u_1-v_1-(u_2-v_2)|\\
& \quad +c_{D^3G}|v_1-v_2||u_2-v_2|(|u_1-u_2|+|u_1-v_1-(u_2-v_2)|).
\end{align*}
\end{lemma}
The proof of the last inequality can be found in Hu and Nualart
\cite{HuNu09} Proposition 6.4.

\bigskip

Let us denote by $L_2(V\times V,\tilde V)$ the space of bilinear continuous mappings $B$ from $V\times V$ which satisfy the Hilbert--Schmidt property
\begin{equation*}
    \|B\|_{L_2(V\times V,\tilde V)}^2:=\sum_{i,j=1}^\infty|B(e_i,e_j)|_{\tilde V}^2<\infty
\end{equation*}
for the complete orthonormal basis $(e_i)_{i\in\NN}$ of $V$.
The topological tensor product of
the Hilbert space $V$  is denoted by $V
\otimes V$ with norm $\|\cdot\|$.
In
particular,  $V \otimes V$ is a separable Hilbert space. The elements of
$V\otimes V$ of the type
$v_1\otimes_V v_2$, with $v_1,\,v_2\in V$, are called {\it elementary tensors} and for them  $\|v_1\otimes_V v_2\| = |v_1||v_2|$. Moreover, $(e_i\otimes_V
e_j)_{i,j\in\NN}$ is a complete orthonormal system of $V \otimes V$.

We note that an operator $B\in L_2(V\times V,\tilde V)$ can be {\em extended} to an
operator $\hat B\in L_2(V\otimes V,\tilde V)$. For weaker conditions  we refer to \cite{KadRin97} Chapter 2.6.
More precisely,  there exists a weak Hilbert-Schmidt mapping $p:V\times V\to V\otimes V$ where $p(e_i,e_j)=e_i\otimes_V e_j$ for $i,\,j\in \NN$. Then $\hat B$ on $V\otimes V$ is given by factorization such that $B=\hat Bp$.
In addition, it is easily seen that for the norm of $\hat B\in L_2(V\otimes V,\tilde V)$ it follows that

\begin{equation*}
   \|\hat B\|_{L_2(V\otimes V,\tilde V)}^2:= \sum_{i,j}|\hat B(e_i\otimes_V e_j)|_{\tilde V}^2=\sum_{i,j}|B(e_i,e_j)|_{\tilde V}^2= \|B\|_{L_2(V\times V,\tilde V)}^2.
\end{equation*}
We will write for  $\hat B$ also the symbol $B$.\\

Let $0\leq T_1 <T_2$. For $\beta\in (0,1)$ we introduce the space
of $\beta$--H{\"o}lder continuous functions on $[T_1,T_2]$ with values
in $V$ denoted by $C_{\beta}([T_1,T_2];V)$ with
the  seminorm

\begin{equation*}
\|u\|_{\beta}=\sup_{s<t\in [T_1,T_2]}\frac{|u(t)-u(s)|}{|t-s|^\beta}.
\end{equation*}

If we consider the nonlinear subspace of {\em all} functions from this linear space having the same value at say $T_1$, then $d(u_1,u_2)=\|u_1-u_2\|_{\beta}$ creates a complete metric space  which will be used later. If we add $|u(T_1)|$ to this seminorm we obtain a Banach space. In particular, this norm is equivalent to the standard norm of H{\"o}lder functions on $[T_1,T_2]$.

Let $\Delta_{T_1,T_2}$ be the triangle $\{(s,t):T_1\le s<t\le T_2\}$. We
now introduce the space $C_{\beta+\beta^\prime}(\Delta_{T_1,T_2};V\otimes V)$ of functions $v$ with finite norm given by
\begin{equation*}
   \|v\|_{\beta+\beta^\prime}=\sup_{s<t\in [T_1,T_2]} \frac{\|v(s,t)\|}{|t-s|^{\beta+\beta^\prime}}, \quad
  \beta+\beta^\prime<1.
\end{equation*}
Notice that we prefer not to stress the interval $[T_1,T_2]$ in the notation of the previous norms, even though this interval can be different through the text.\\

Now we aim at
introducing the so called fractional derivatives and later at giving
the pathwise interpretation of the stochastic integral, following
the definition in \cite{Zah98}.

Let $0< \alpha<1$ and consider $g\in I_{T_1+}^\alpha (L^p((T_1,T_2);\mathbb R)),\, \zeta\in
I_{T_2-}^{\alpha} (L^q((T_1,T_2); \mathbb R))$ with $p$, $q\ge 1$ (for the definition of the spaces
$I_{T_1+}^\alpha(L^p((T_1,T_2);\mathbb R))$ and $I_{T_2-}^\alpha(L^q((T_1,T_2);\mathbb R))$ we
refer, for instance, to Samko {\it et al.} \cite{Samko}). The
fractional derivatives in the Weyl sense are defined by
\begin{align}\label{fractder}
    D_{{T_1}+}^\alpha g[r]&=\frac{1}{\Gamma(1-\alpha)}\bigg(\frac{g(r)}{(r-T_1)^\alpha}+\alpha\int_{T_1}^r\frac{g(r)-g(q)}{(r-q)^{1+\alpha}}dq\bigg)\\
    D_{{T_2}-}^\alpha \zeta_{{T_2}-}[r]&=\frac{(-1)^\alpha}{\Gamma(1-\alpha)}
    \bigg(\frac{\zeta(r)-\zeta(T_2)}{(T_2-r)^\alpha}+\alpha\int_r^{T_2}\frac{\zeta(r)-\zeta(q)}{(q-r)^{1+\alpha}}dq\bigg),\nonumber
\end{align}
where $T_1\leq r\leq T_2$, and $ \zeta_{T_2-}(r)= \zeta(r)- \zeta(T_2)$. For such functions, if in addition $1/p+1/q\le 1+\alpha$, the following formula holds (see \cite{Zah98}):
\begin{equation}\label{eq10b}
(-1)^\alpha \int_{T_1}^{T_2} D_{T_1+}^\alpha g[r]\zeta(r)dr=\int_{T_1}^{T_2} g(r)
D_{T_2-}^\alpha \zeta[r]dr.
\end{equation}
If now we assume that $g(T_1+),\,\zeta(T_1+),\,\zeta(T_2-)$ exist, being respectively the right side limit of $g$ at $T_1$ and the right and left side limits of $\zeta$ at $T_1,\,T_2$, and that $g_{T_1+} \in I_{T_1+}^\alpha (L^ p((T_1,T_2);\mathbb R)),\, \zeta_{T_2-} \in
I_{T_2-}^{\alpha} (L^{q}((T_1,T_2); \mathbb R))$ with $1/p+1/q\le 1$, then
\begin{align}\label{eq10bi}
\begin{split}
    \int_{T_1}^{T_2} gd\zeta&=(-1)^\alpha\int_{T_1}^{T_2} D_{T_1+}^\alpha g_{T_1+}[r]D_{T_2-}^{1-\alpha}\zeta_{T_2-}[r]dr+g(T_1+)(\zeta(T_2-)-\zeta(T_1+)).
\end{split}
\end{align}
Here
$g_{T_1+}(\cdot)=g(\cdot)-g(T_1+)$ and $\zeta_{T_2-}(\cdot)=\zeta(\cdot)-\zeta(T_2-)$. In addition, when $\alpha p<1$ and $g(T_1+)$ exists and $g\in I_{T_1+}^\alpha (L^p((T_1,T_2);\mathbb R))$, (\ref{eq10bi}) can be rewritten as
\begin{align}\label{eq29}
    \int_{T_1}^{T_2} gd\zeta&=(-1)^\alpha\int_{T_1}^{T_2} D_{T_1+}^\alpha g[r]D_{T_2-}^{1-\alpha}\zeta_{T_2-}[r]dr.
\end{align}
Notice that in the case that $\zeta$ is not Lipschitz continuous we cannot define the integral on the left-hand side of \eqref{eq10bi} in the classical way. Nevertheless, when $g$ and $\zeta$ are H\"older continuous with exponents $\beta,\,\beta^\prime$ resp., and $\alpha<\beta,\,1-\alpha<\beta^\prime$, we can define the integral. In particular, assume that $g\in C_\beta([T_1,T_2];L_2(V,\tilde V))$, $\zeta\in
C_{\beta^{\prime}} ([T_1,T_2]; V)$ for $0<\alpha<\beta,\,1-\alpha<\beta^\prime$ (note that for these Hilbert-valued functions the fractional derivatives can be defined similar than (\ref{fractder})).
Let us define
\begin{equation}\label{eq10bis}
    \int_{T_1}^{T_2} g(r)d\zeta(r)=(-1)^\alpha\int_{T_1}^{T_2} D_{T_1+}^\alpha g[r]D_{T_2-}^{1-\alpha}\zeta_{T_2-}[r]dr.
    \end{equation}
This expression can be also interpreted as a fractional integration by parts formula. By the separability of $\tilde V$, Pettis' theorem  and by
 \begin{align}\label{INT}
 \begin{split}
\bigg|\int_{T_1}^{T_2}&D_{T_1+}^{\alpha}g[r]D_{T_2-}^{1-\alpha}\zeta_{T_2-}[r]dr\bigg|_{\tilde V}\le
    \int_{T_1}^{T_2}|D_{T_1+}^{\alpha}g[r]D_{T_2-}^{1-\alpha}\zeta_{T_2-}[r]|_{\tilde V}dr\\&\le \int_{T_1}^{T_2}\|D_{T_1+}^{\alpha}g[r]\|_{L_2(V,\tilde V)}|D_{T_2-}^{1-\alpha}\zeta_{T_2-}[r]|dr\\
&\le c \|\zeta\|_{\beta^\prime} \int_{T_1}^{T_2} ( \|g(T_1)\|_{L_2(V,\tilde V)} (r-T_1)^{-\alpha}(T_2-r)^{\alpha+{\beta^\prime}-1}
+ \|g\|_{\beta}  (r-T_1)^{\beta-\alpha}(T_2-r)^{\alpha+{\beta^\prime}-1} )dr \\
   & \le c  \|\zeta\|_{\beta^\prime} \bigg(\|g(T_1)\|_{L_2(V,\tilde V)}(T_2-T_1)^{\beta^\prime} +\|g\|_{\beta}   (T_2-T_1)^{{\beta^\prime}+\beta}\bigg)
    \end{split}
\end{align}
this integral is well-defined. Indeed,

\begin{equation}\label{eq28}
\|D_{T_1+}^{\alpha}g[r]\|_{L_2(V,\tilde V)}\le c\bigg(\frac{\|g(T_1)\|_{L_2(V,\tilde V)}}{(r-T_1)^\alpha}+\frac{\|g\|_{\beta}}{(r-T_1)^{\alpha-\beta}}\bigg),
\end{equation}
and, if $1-\alpha<{\beta^\prime}$, the expression $D_{T_2-}^{1-\alpha}\zeta_{T_2-}[r]$ is well-defined, actually, it is simple to obtain that $|D_{T_2-}^{1-\alpha}\zeta_{T_2-}[r]| \leq \|\zeta\|_\beta (T_2-r)^{{\beta^\prime}+\alpha-1}$.

Let $\{\tilde e_i\}_{i\in \mathbb N}$ be a complete orthonormal basis of $\tilde V$.
Denote by $\pi_m$ and $\tilde \pi_m$ the orthogonal projections on $\{e_1,\cdots,e_m\}$ and $\{\tilde e_1,\cdots,\tilde e_m\}$, respectively, and define $g_j=(\tilde \pi_j-\tilde \pi_{j-1})g$, $\zeta_i=(\pi_i- \pi_{i-1})\zeta$ and $g_{ji}=(\tilde \pi_j-\tilde \pi_{j-1})g( \pi_i- \pi_{i-1})$.
The above estimates allow to exchange the sum and the integral such that we have

\begin{equation}\label{eq10bis1}
    \int_{T_1}^{T_2}g(r)d\zeta(r)=\sum_{j}\bigg(\sum_i\int_{T_1}^{T_2}
    D_{T_1+}^{\alpha}g_{ji}[r]D_{T_2-}^{1-\alpha}\zeta_{iT_2-}[r]dr \bigg)\tilde e_j
\end{equation}
and

\begin{equation*}
    \bigg|\int_{T_1}^{T_2}g(r)d\zeta(r)\bigg|_{\tilde V}
    =\bigg(\sum_j\bigg|\sum_{i}\int_{T_1}^{T_2}g(r)_{ji}d\zeta_i(r)\bigg|^2\bigg)^\frac12<\infty
\end{equation*}
where we have used that

\begin{align*}
   \|D_{T_1+}^{\alpha}g[r]\|_{L_2(V,\tilde V)}&=\bigg( \sum_{i,j}| D_{T_1+}^\alpha g(\cdot)_{ji}[r]|^2\bigg)^\frac12 \le \bigg(\sum_{i,j}\bigg(\frac{1}{\Gamma(1-\alpha)}\bigg(\frac{g_{ji}(r)}{(r-T_1)^\alpha}
   +\alpha\int_{T_1}^r\frac{g_{ji}(r)-g_{ji}(q)}{(r-q)^{1+\alpha}}dq\bigg)\bigg)^2\bigg)^\frac12\\
   &\le \sqrt{2} c\bigg(\frac{(\sum_{i,j}g_{ji}(r)^2)^\frac12}{(r-T_1)^\alpha}
   +\bigg(\sum_{i,j}\bigg(\int_{T_1}^r\frac{g_{ji}(r)-g_{ji}(q)}{(r-q)^{1+\alpha}}dq\bigg)^2\bigg)^\frac12\bigg)\\
   &\le \sqrt{2}c\bigg(\frac{\|g(r)\|_{L_2(V,\tilde V)}}{(r-T_1)^\alpha}+\int_{T_1}^r\frac{\|g(r)-g(q)\|_{L_2(V,\tilde V)}}{(r-q)^{1+\alpha}}dq\bigg)\\
   & \le c (r-T_1)^{-\alpha} (\|g(T_1)\|_{L_2(V,\tilde V)}+\|g\|_\beta (r-T_1)^\beta).
\end{align*}
In a similar manner we can also define integrals with values in the separable Hilbert space $V\otimes V$
when $g(r)\in L_2(V,\RR)\cong V$ by

\begin{equation*}
    \int_{T_1}^{T_2}g(r)\otimes_V d\zeta(r).
\end{equation*}

\begin{remark}\label{sing}
Suppose that $g_{ji}$ satisfies the assumptions for (\ref{eq29}) but is not H{\"o}lder continuous in general. In addition
suppose that
$r\mapsto \|D_{T_1+}^\alpha g[r]\|_{L_2(V,\tilde V)}|D_{T_2-}^{1-\alpha}\zeta_{T_2-}[r]|$ is integrable. Then we can define the integrals (\ref{eq10bis}) and (\ref{eq10bis1}).
This will be used later on for  $g(r)=S(T_2-r)f(r)$ where $f$ is an appropriate function. Note that $r\mapsto S(T_2-r)x,\,x\in V$ is not H{\"o}lder continuous but its finite dimensional approximations are Lipschitz. \end{remark}

Consider now $u\in C_{\beta}([T_1,T_2];V)$, $\zeta\in C_{\beta^\prime}([T_1,T_2];V)$ and $v\in C_{\beta+\beta^\prime}(\Delta_{T_1,T_2};V\otimes
V))$  such that the so called {\em Chen equality} holds, that is, for
$T_1\le s\le r\le t\le T_2$,
\begin{equation}\label{chen}
   v(s,r)+v(r,t)+(u(r)-u(s))\otimes_V(\zeta(t)-\zeta(r))=
   v(s,t).
\end{equation}

\begin{remark}\label{r1}
If $\zeta$ is continuously differentiable, an example for $v$ is given by $(u\otimes\zeta)$ where
\begin{align}\label{eq11}
(u\otimes\zeta)(s,t)=\int_s^t(u(\tau)-u(s))\otimes_V
\zeta^\prime(\tau)d\tau=v(s,t).
\end{align}
This expression is clearly well-defined and belongs (at least) to the space $C_{\beta+\beta^\prime}(\Delta_{T_1,T_2};V\otimes V)$. Moreover, the Chen equality
easily follows.

\end{remark}
In addition, for $v\in
C_{\beta+\beta^\prime}(\Delta_{T_1,T_2};V\otimes V)$ and for $r \in [T_1,T_2]$ we
introduce the following fractional derivative
\begin{align}\label{vfd}
\mathcal{D} _{{T_2}-}^{1-\alpha}v[r] &=\frac{(-1)^{1-\alpha}}{%
\Gamma(\alpha)} \bigg( \frac{v(r,T_2)}{(T_2-r)^{1-\alpha}}%
+(1-\alpha) \int_r^{T_2} \frac{v(r,\tau)}{(\tau-r)^{2-\alpha}}d\tau %
\bigg).
\end{align}

Suppose now that $g(r)=G(u(r))$ where $u\in C_\beta([T_1,T_2];V)$ such that $\alpha<\beta,\,\alpha+\beta^\prime>1$, with
$G$ having a bounded Fr\'echet derivative $DG$. Then
\begin{equation}\label{intnr}
   \int_{T_1}^{T_2} G(u)d\zeta=(-1)^\alpha \int_{T_1}^{T_2}D_{T_1+}^\alpha G(u(\cdot))[r]D_{T_2-}^{1-\alpha}\zeta_{T_2-}[r]dr
\end{equation}
is well-defined because $G(u(\cdot))$ is $\beta$-H{\"o}lder continuous.
 Assuming in addition that $G$ has a second bounded derivative we can rewrite the integral in \eqref{intnr} as follows
\begin{align}\label{rel3}
\begin{split}
\int_{T_1}^{T_2} G(u)d\zeta = & (-1)^\alpha\int_{T_1}^{T_2} D_{T_1+}^\alpha (G(u(\cdot))-DG(u(\cdot))(u-u(T_1),\cdot))[r]D_{T_2-}^{1-\alpha}\zeta_{T_2-}[r]dr\\+&(-1)^\alpha\int_{T_1}^{T_2}D^\alpha_{T_1+}DG(u(\cdot))(u-u(T_1),\cdot)[r]D_{T_2-}^{1-\alpha}\zeta_{T_2-}[r]dr.
\end{split}
\end{align}
Suppose now that the above condition $\beta>\alpha$ is not satisfied. Then $D_{T_1+}^\alpha G(u)$ is not well defined, in general.
In this case it has sense to rewrite \eqref{intnr} by using
the so called {\em compensated fractional derivative}
\begin{align}\label{compensated}
\begin{split}
\hat D^{\alpha}_{{T_1}+}G(u(\cdot))[r] &=\frac{1}{\Gamma(1-\alpha)} \bigg( \frac{G(u(r))%
}{(r-{T_1})^\alpha}\\
&+\alpha \int_{T_1}^r \frac{G(u(r))-G(u(q))-DG(u(%
q))(u(r)-u(q),\cdot)}{(r-q)^{\alpha+1}}dq \bigg)\in L_2(V,\tilde V)
\end{split}
\end{align}
if $2\beta>\alpha$. By making some computations, it is not difficult to see that there is the following relation between the fractional derivative and the compensated one:
\begin{align}\label{rel}
D^\alpha_{T_1+} (G(u(\cdot))-DG(u(\cdot))(u(\cdot)-u(T_1),\cdot))[r]=\hat
D_{T_1+}^\alpha G(u(\cdot))[r]-D^\alpha_{T_1+} DG(u(\cdot))[r]
(u(r)-u(T_1),\cdot).
\end{align}
In addition, due to the fact that $u,\zeta,v$ are coupled by the Chen equality \eqref{chen}, we  have
\begin{align}\label{rel2}
\begin{split}
D_{T_2-}^{1-\alpha} v(T_1,\cdot)_{T_2-}[r]&=\frac{(-1)^{1-\alpha}}{\Gamma(\alpha)}
    \bigg(\frac{v(T_1,r)-v(T_1,T_2)}{(T_2-r)^\alpha}+(1-\alpha)\int_r^{T_2}\frac{v(T_1,r)-v(T_1,q)}{(q-r)^{2-\alpha}}dq\bigg)\\
    &=\frac{(-1)^{1-\alpha}}{\Gamma(\alpha)}
    \bigg(\frac{-v(r,T_2)-(u(r)-u(T_1))\otimes_V (\omega(T_2)-\omega(r))}{(T_2-r)^\alpha}\\
    &+(1-\alpha)\int_r^{T_2}\frac{-v(r,q)-(u(r)-u(T_1))\otimes_V (\omega(q)-\omega(r))}{(q-r)^{2-\alpha}}dq\bigg)\\
   & =-\dD_{T_2-}^{1-\alpha} v[r]+(u(r)-u(T_1))\otimes_V D_{T_2-}^{1-\alpha}\zeta_{T_2-}[r].
    \end{split}
\end{align}
Therefore, the expressions (\ref{rel3}), (\ref{rel}) and (\ref{rel2}) enable us to express \eqref{intnr} under a weaker regularity condition (if $\beta>\alpha$ is not satisfied, but for instance $2\beta>\alpha$) in the way
\begin{align}\label{eq24}
\begin{split}
  \int_{T_1}^{T_2} G(u)d\zeta&=(-1)^\alpha\int_{T_1}^{T_2}\hat D^{\alpha}_{{T_1}+}G(u(\cdot))[r]D_{T_2-}^{1-\alpha}\zeta_{T_2-}[r]dr\\&-(-1)^{2\alpha-1}\int_{T_1}^{T_2}D_{T_1+}^{\alpha} DG(u(\cdot))[r]\mathcal{D} _{{T_2}-}^{1-\alpha}v[r] dr\\
&=(-1)^\alpha\int_{T_1}^{T_2}\hat D^{\alpha}_{{T_1}+}G(u(\cdot))[r]D_{T_2-}^{1-\alpha}\zeta_{T_2-}[r]dr\\
&-(-1)^{2\alpha-1}\int_{T_1}^{T_2}D_{T_1+}^{2\alpha-1} DG(u(\cdot))[r]D_{T_2-}^{1-\alpha}\mathcal{D} _{{T_2}-}^{1-\alpha}v[r] dr.
\end{split}
\end{align}
Note that the last equality is true due to the property (\ref{eq10b}). Moreover, this previous expression of the integral is similar to the one obtained in \cite{HuNu09}.

If $\hat D_{T_1^+}^\alpha G(u(\cdot))[r]$ has now the right regularity then we can define the first integral on the right-hand side of the
last formula similar to \eqref{eq10bis}.

Now let us consider an integral having the structure of the second integral, namely
\begin{equation}\label{eq25}
    \int_{T_1}^{T_2}D_{T_1+}^{2\alpha-1} g(\cdot)[r]D_{T_2-}^{1-\alpha}\mathcal{D} _{{T_2}-}^{1-\alpha}v[r] dr
\end{equation}
for some $v\in C_{\beta+\beta^\prime}(\Delta_{T_1,T_2};V\otimes V)$. Thanks to the Chen equality (\ref{chen}), it is not hard to prove that the following inequality holds
\begin{equation}\label{eq26}
   \|D_{T_2-}^{1-\alpha}\mathcal{D} _{{T_2}-}^{1-\alpha}v[r]\|\le c(\|v\|_{\beta+\beta^\prime}+\|u\|_\beta\|\omega\|_{\beta^\prime})(T_2-r)^{\beta+\beta^\prime+2\alpha-2}
\end{equation}
(see Lemma 6.3 in
\cite{HuNu09} for some details) and hence this integral can be defined in a similar manner than \eqref{eq10bis} provided that $g(\cdot)$ is $\beta$--H{\"o}lder continuous
with $0<\alpha<1,\, \beta+1>2\alpha,$ and $\beta^\prime>\beta >1-\alpha$. For details in finite dimension we refer to \cite{HuNu09}.

In the next sections we will give sense to the integrals appearing in the definition of mild solution to the infinite dimensional equation \eqref{eqn1}.

\section{Formulation of \eqref{eqn1} for smooth paths $\omega$}\label{regular}

Throughout this section, we assume that the driving path $\omega:[0,T]\to V$ in the system \eqref{eqn1} is  smooth in the sense that $\omega$ is continuous at any $t$ and continuously differentiable except at finitely many points.  Then we derive a system of equations which is needed to define a solution when the noise in only H{\"o}lder continuous, the case to be considered in the next section. When the path $\omega$ is only $V$-valued H\"older continuous  with H\"older exponent in $(1/3,1/2)$, we will consider a piecewise linear approximation $\omega^n$ of $\omega$, for which we can apply the results that we are going to establish throughout this section.

Throughout all the paper, we will use $c$ or $C$ to denote a generic positive constant which value
is not so important and that may change from line to line. That constant may depend on parameters, for instance, it may depend on  $\omega$.

In what follows we assume that $\tilde V=V_\delta$ with complete orthonormal base $({e_i}/{\lambda_i^\delta})_{i\in\mathbb{N}}$. Under such a choice, we consider $G:V\rightarrow L_2(V,V_\delta)$.

For the fixed regular $\omega$, we study the equation
\begin{equation}\label{eq0}
u(t)=S(t)u_0+\int_0^t S(t-s) G(u(s))d\omega(s).
\end{equation}

\begin{lemma}\label{l6} (\cite{Pazy})
Under the condition that $G$ is Lipschitz continuous, the above
equation has a unique global solution which depends continuously
on $u_0\in V_\delta$. Moreover, $u\in C_{\beta}([0,T];V)$ for $\beta\le \delta,\, \beta< 1$.
\end{lemma}
If in addition we assume that $G$ satisfies the conditions of Lemma \ref{l2},
similarly to the expression (\ref{eq24}) of last section,
we can rewrite, for $\alpha \in (0,1)$, the above equation as
    \begin{align}\label{sol}
\begin{split}
    u(t)&=S(t)u_0+(-1)^\alpha\int_0^t\hat D_{0+}^\alpha
    (S(t-\cdot)G(u(\cdot)))[r]D_{t-}^{1-\alpha}\omega_{t-}[r]dr\\
&-(-1)^{2\alpha-1}\int_0^t
    {D}_{0+}^{2\alpha-1}(S(t-\cdot)DG(u(\cdot)))[r]D_{t-}^{1-\alpha}\dD_{t-}^{1-\alpha}(u\otimes\omega)[r]dr,
 \end{split}
\end{align}
where the compensated fractional derivative is defined as in
\eqref{compensated} and the fractional derivative of
$(u\otimes \omega)$ is defined as in \eqref{vfd}.

 We want to point out that, when the noise is not
regular, as we will show in the following section, in order to
give a meaningful definition of the solution of \eqref{eqn1} we
need an equation to determine the corresponding counterpart of
$(u\otimes\omega)$. For this reason, now
we aim at getting such expression in the case that $\omega$ is
smooth. Firstly, choosing $\zeta=\omega$ in \eqref{eq11}, we can
express the tensor as follows
\begin{align}\label{ns}
\begin{split}
    (u\otimes\omega)(s,t)&=\int_s^t(S(\xi-s)-{\rm id})u(s)\otimes_V\omega^\prime(\xi)d\xi+\int_s^t\int_s^\xi S(\xi-r)G(u(r))\omega^\prime(r)dr\otimes_V\omega^\prime(\xi)d\xi.
    \end{split}
\end{align}

Fix $\delta\in [0,1]$ and let $\beta^\prime \in (1/3,1/2)$. For $\alpha\in (0,1)$, $0 \leq s\leq t \leq T$ and the semigroup
$S$ introduced in Section \ref{preli}, we consider
\begin{align*}
    \omega_S(s,t)\cdot=(-1)^{-\alpha}\int_s^t(S(\xi-s)\cdot)\otimes_V\omega^\prime(\xi)d\xi
\end{align*}
on $\Delta_{0,T}$ as a linear mapping from $V$ into $V\otimes V$. It can be also easily seen that there exists a $c\geq 0$ such that for any
$e\in V$
\begin{align*}
\|\omega_S(s,t)e-\omega_S(r,t)e\|\leq c
(r-s)^{\beta^\prime}|e|,\quad 0\leq s\leq r\leq t.
\end{align*}
We consider $(\omega \otimes_S \omega): \Delta_{0,T}\times
L_2(V,V_\delta) \rightarrow V\otimes V$ given by
\begin{align}\label{omegaSS}
\begin{split}
    E(\omega&\otimes_S\omega)(s,t)=\int_s^t\int_s^\xi S(\xi-r)E\omega^\prime(r)dr\otimes_V\omega^\prime(\xi)d\xi\\
&=\int_s^t\int_r^t
    S(\xi-r)E\omega^\prime(r)\otimes_V\omega^\prime(\xi)d\xi dr=(-1)^\alpha\int_s^t\omega_S(r,t)E(\omega^\prime(r))dr.
    \end{split}
\end{align}

Let us recall that $L_2(V,V_\delta)$ is a separable Hilbert space with an orthonormal basis $(E_{ij})_{i,j\in\NN}$ derived from the basis $(e_k)_{k\in \NN}$ of $V$ and $({e_k}/{\lambda_k^\delta})_{k\in \NN}$ of $V_\delta$ as follows:
\begin{equation} \label{etiqueta}   E_{ij}e_k=\left\{\begin{array}{lcl}
    0&:& j\not= k\\
    \frac{e_i}{\lambda_i^\delta}&:& j= k.
    \end{array}
    \right.
\end{equation}
Suppose that
\begin{equation}\label{finitesum}
    \sum_{i=1}^\infty \lambda_i^{-1-2\delta}<\infty.
\end{equation}
Then for the smooth path $\omega$ we have
\begin{align*}
    \|E_{ij}(\omega\otimes_S\omega)(s,t)\|^2&=\bigg\|\int_s^t\int_r^tS(\xi-r)E_{ij}\omega^\prime(r)\otimes_V\omega^\prime(\xi)d\xi dr\bigg\|^2\\
    &=\bigg\|\int_s^t\int_r^te^{-\lambda_i(\xi-r)}\frac{e_i}{\lambda_i^\delta}(\omega^\prime(r),e_j)\otimes_V\omega^\prime(\xi) d\xi dr\bigg\|^2,\\
    \sum_{i,j}\|E_{ij}(\omega\otimes_S\omega)(s,t)\|^2&\le \bigg(\int_s^t |\omega^\prime(\xi)|^2d\xi\bigg)^2\sum_i\frac{1}{2\lambda_i^{1+2\delta}}\int_s^t(1-e^{-2\lambda_i(t-r)})dr<\infty.
\end{align*}
In particular, since $\omega$ is smooth we can conclude that $(\omega\otimes_S\omega)\in C_{2\beta^\prime}(\Delta_{0,T};L_2(L_2(V,V_\delta),V\otimes V)).$ Indeed, from the above inequality, there exists $c$  such that (at least)
\begin{align*}
     \sum_{i,j}\|E_{ij}(\omega\otimes_S\omega)(s,t)\|^2&\le c (t-s)^{4\beta^\prime}.
     \end{align*}
In addition, the following equality is interpreted to be the Chen equality for $(\omega\otimes_S \omega)$:  for $0\leq s\leq
r\leq t \leq T$ it follows
\begin{align}\label{chenomegaS}
\begin{split}
E(\omega \otimes_S \omega) &(s,r) +E(\omega\otimes_S
\omega)(r,t) \\
+&\int_r^t S(\xi-r) \int_s^r S(r-q) E\omega^\prime(q) dq \otimes_V
\omega^\prime(\xi)d\xi =E(\omega\otimes_S \omega)(s,t).
\end{split}
\end{align}
If now we take into account \eqref{omegaSS}, the last integral
on the right hand side of \eqref{ns} can be written as
\begin{align}\label{q5}
\begin{split}
 - \int_s^t G(u(r))& D_1(\omega\otimes_S\omega)(r,t) dr\\
  &= -\int_s^t    (G(u(r))-DG(u(r))(u(r)-u(s),\cdot))  D_1(\omega\otimes_S\omega)(r,t)
dr\\
&-\int_s^t
    DG(u(r))(u(r)-u(s),\cdot) D_1(\omega\otimes_S\omega)(r,t) dr\\
    &=-(-1)^\alpha\int_s^t \hat D_{s+}^\alpha G(u(\cdot))[r] D_{t-}^{1-\alpha}(\omega\otimes_S\omega)(\cdot,t)_{t-}[r]dr \\
&+(-1)^\alpha\int_s^t D_{s+}^\alpha DG(u(\cdot))(u(r)-u(s),\cdot)[r] D_{t-}^{1-\alpha}(\omega\otimes_S\omega)(\cdot,t)_{t-}[r] dr\\
&-(-1)^\alpha\int_s^t D_{s+}^\alpha
    DG(u(\cdot))[r]D_{t-}^{1-\alpha}(u\otimes(\omega\otimes_S\omega)(\cdot,t))(s,\cdot)_{t-}[r]dr.
    \end{split}
\end{align}
Notice that in the previous expression we have used \eqref{eq10bis} and \eqref{rel}. Moreover, for $\tilde E\in L_2(V\times V,V_\delta)$, we set
\begin{align}\label{eq33}
\begin{split}
    \tilde
    E(u\otimes(\omega\otimes_S\omega))(\cdot,t)(s,r)&=\int_s^r
   \tilde E (u(q)-u(s),\cdot)D_1(\omega\otimes_S\omega)(q,t)  dq\\
&=-\int_s^r\int_q^tS(\xi-q)\tilde E (u(q)-u(s),\omega^\prime (q))\otimes_V\omega^\prime(\xi)d\xi
    dq\in V\otimes V
    \end{split}
\end{align}
such that
\begin{align*}
&\tilde E D_2(u\otimes(\omega\otimes_S\omega))(\cdot,t)(s,\cdot)[r]=\tilde E (u(r)-u(s),\cdot)D_1(\omega\otimes_S\omega)(r,t)
   ,
\end{align*}
which gives us the last integral on the right-hand side of
\eqref{q5}, where $D_1$ and $D_2$ denote, respectively, the derivative with respect to the first and second component of the corresponding tensor. From now on we write
$(u\otimes(\omega\otimes_S\omega))(t)$ instead of
$(u\otimes(\omega\otimes_S\omega))(\cdot,t)$.

Consider the separable Hilbert space $L_2(V\times V,V_\delta)$ equipped with the complete orthonormal basis
$(\tilde E_{ijk})_{i,j,k\in \NN}$
\begin{equation*}
     \tilde E_{ijk}e_l\otimes_Ve_{m}=\tilde E_{ijk}(e_l,e_{m})=\left\{\begin{array}{lcl}
    0&:& j\not= l\;\text{ or } k\not= m\\
    \frac{e_i}{\lambda_i^\delta} &:& j= l\;\text{ and }  k=m.
    \end{array}
    \right.
\end{equation*}
Under the assumption \eqref{finitesum}, similarly to the above estimate for $(\omega \otimes_S \omega)$, we have that

\begin{align*}
    \sum_{i,j,k}&\bigg\| \int_s^r\int_q^tS(\xi-q)\tilde E_{ijk} (u(q)-u(s),\omega^\prime (q))\otimes_V\omega^\prime(\xi)d\xi
    dq\bigg\|^2\\
&\le\sum_i\frac{1}{2\lambda_i^{1+2\delta}}\int_s^t(1-e^{-2\lambda_i(t-q)})dq \int_s^t |\omega^\prime(\xi)|^2d\xi\int_s^t|u(q)-u(s)|^2|\omega^\prime(q)|^2dq <\infty,
\end{align*}
which shows in particular that

\begin{equation*}
    (u\otimes(\omega\otimes_S\omega))\in C_{\beta+\beta^\prime}(\Delta_{0,T};L_2(L_2(V\times V,V_\delta),V\otimes V)).
\end{equation*}

\begin{lemma} Suppose that \eqref{finitesum} holds. For $0\leq s\leq t\leq T$, $(u \otimes \omega)$ satisfies the equation
\begin{align}\label{uomega}
\begin{split}
    (u\otimes\omega)(s,t)&=\int_s^t(S(\xi-s)-{\rm
    id})u(s)\otimes_V\omega^\prime(\xi)d\xi\\
    &-(-1)^\alpha\int_s^t \hat D_{s+}^\alpha
    G(u(\cdot))[r]D_{t-}^{1-\alpha}(\omega\otimes_S\omega)(\cdot,t)_{t-}[r] dr \\
&+(-1)^{2\alpha-1}\int_s^tD_{s+}^{2\alpha-1} DG(u(\cdot))[r] D_{t-}^{1-\alpha}\dD_{t-}^{1-\alpha}(u\otimes(\omega\otimes_S\omega)(t))[r]dr .
\end{split}
\end{align}
\end{lemma}

\begin{proof} Let us consider $\tilde E\in L_2(V\times V, V_\delta)$. Since, for $0\leq s\leq r\leq
q \leq t\leq T$,
\begin{align*}
-&\int_s^r\int_\tau^tS(\xi-\tau)\tilde E(u(\tau)-u(s),\omega^\prime(\tau))\otimes_V\omega^\prime(\xi)d\xi
    d\tau
   \\ -&
    \int_r^q\int_\tau^tS(\xi-\tau)\tilde E(u(\tau)-u(r),\omega^\prime(\tau))\otimes_V\omega^\prime(\xi)d\xi
    d\tau
    \\-&
    \int_r^q\int_\tau^qS(\xi-\tau)\tilde E(u(r)-u(s),\omega^\prime(\tau)) \otimes_V\omega^\prime(\xi)d\xi
    d\tau\\-&
    \int_r^q\int_q^tS(\xi-\tau)\tilde E(u(r)-u(s),\omega^\prime(\tau)) \otimes_V\omega^\prime(\xi)d\xi
    d\tau\\
    =
    -&\int_s^q\int_\tau^tS(\xi-\tau)\tilde E(u(\tau)-u(s),\omega^\prime(\tau)) \otimes_V\omega^\prime(\xi)d\xi
    d\tau
\end{align*}
we have
\begin{align}\label{3tensorChen}
\begin{split}
\tilde
E & (u\otimes (\omega\otimes_S\omega)(t)) (s,r)+\tilde E(u\otimes(\omega\otimes_S\omega)(t))(r,q)
    -\tilde E(u(r)-u(s),\cdot)(\omega\otimes_S\omega)(r,q) \\
    =&\tilde E(u\otimes (\omega\otimes_S\omega)(t))(s,q)+\int_r^q\int_q^tS(\xi-\tau)\tilde E(u(r)-u(s),\omega^\prime(\tau))\otimes_V\omega^\prime(\xi)d\xi
    d\tau.
    \end{split}
\end{align}
In particular, when $q=t$, we have
\begin{align}\label{chen3fold}
\begin{split}
\tilde
E(u\otimes&(\omega\otimes_S\omega)(t))(s,r)+\tilde E(u\otimes(\omega\otimes_S\omega)(t))(r,t)
    -\tilde E(u(r)-u(s),\cdot)(\omega\otimes_S\omega)(r,t) \\
    =&\tilde E(u\otimes (\omega\otimes_S\omega)(t))(s,t).
    \end{split}
\end{align}
Since $(\omega\otimes_S\omega)(t,t)=0$ the above expression is
exactly the Chen equality for $(u\otimes
(\omega\otimes_S\omega))(t)$:
\begin{align*}
 \tilde E(u & \otimes(\omega\otimes_S\omega)(t))(s,r)+\tilde E(u\otimes(\omega\otimes_S\omega)(t))(r,t)\\
&+ \tilde E(u(r)-u(s),\cdot)((\omega\otimes_S\omega)(t,t)-(\omega\otimes_S\omega)(r,t)) =\tilde E(u\otimes (\omega\otimes_S\omega)(t))(s,t).
\end{align*}
Thanks to \eqref{3tensorChen} we have

\begin{align*}
    \tilde E D_{t-}^{1-\alpha}&(u\otimes (\omega\otimes_S\omega)(t))(s,\cdot)_{t-}[r]\\
&=\frac{(-1)^{1-\alpha}}{\Gamma(\alpha)}\bigg(\frac{\tilde E(u\otimes (\omega\otimes_S\omega)(t))(s,r)-\tilde E (u\otimes (\omega\otimes_S\omega)(t))(s,t)}{(t-r)^{1-\alpha}}\\
&+(1-\alpha)\int_r^t\frac{\tilde E(u\otimes (\omega\otimes_S\omega)(t))(s,r)-\tilde E(u\otimes (\omega\otimes_S\omega)(t))(s,\theta)}{(\theta-r)^{2-\alpha}}d\theta\bigg)\\
&=\frac{(-1)^{1-\alpha}}{\Gamma(\alpha)}\bigg(\frac{-\tilde E(u\otimes (\omega\otimes_S\omega)(t))(r,t)+\tilde E (u(r)-u(s),\cdot)(\omega\otimes_S\omega)(r,t)}{(t-r)^{1-\alpha}}\\
&+(1-\alpha)\int_r^t\frac{-\tilde E(u\otimes (\omega\otimes_S\omega)(t))(r,\theta)+\tilde E (u(r)-u(s),\cdot)(\omega\otimes_S\omega)(r,\theta)}{(\theta-r)^{2-\alpha}}d\theta\\
&+(1-\alpha)\int_r^t\frac{\int_r^\theta \int_\theta^t S(\xi-\tau) \tilde E (u(r)-u(s),\omega^\prime(\tau))\otimes_V \omega^\prime(\xi)
     d\xi d\tau }{(\theta-r)^{2-\alpha}}d\theta \bigg)\\
&=-\tilde E \dD_{t-}^{1-\alpha}(u\otimes(\omega\otimes_S\omega)(t))[r]\\
&+\frac{(-1)^{1-\alpha}}{\Gamma(\alpha)}\bigg(\frac{\tilde E (u(r)-u(s),\cdot)(\omega\otimes_S \omega)(r,t)}{(t-r)^{1-\alpha}}
    +(1-\alpha)\int_r^t\frac{\tilde
    E (u(r)-u(s),\cdot)(\omega\otimes_S \omega)(r,\theta) }{(\theta-r)^{2-\alpha}}d\theta\\
&+(1-\alpha)\int_r^t\frac{\int_r^\theta \int_\theta^t S(\xi-\tau)  \tilde E (u(r)-u(s),\omega^\prime(\tau))\otimes_V \omega^\prime(\xi)
     d\xi d\tau  }{(\theta-r)^{2-\alpha}}d\theta \bigg).
\end{align*}
Furthermore, by \eqref{chenomegaS} we have
    \begin{align*}
 \tilde E D_{t-}^{1-\alpha}&
     (u(r)-u(s),\cdot)(\omega\otimes_S \omega)(\cdot,t)_{t-}[r]=
    \tilde E (u(r)-u(s),\cdot)D_{t-}^{1-\alpha}(\omega\otimes_S \omega)(\cdot,t)_{t-}[r] \\
    &=\frac{(-1)^{1-\alpha}}{\Gamma(\alpha)}\bigg(\frac{\tilde E (u(r)-u(s),\cdot)(\omega\otimes_S \omega)(r,t)}{(t-r)^{1-\alpha}}\\
&+(1-\alpha)\int_r^t\frac{\tilde E (u(r)-u(s),\cdot)(\omega\otimes_S \omega)(r,t)-\tilde E (u(r)-u(s),\cdot)(\omega\otimes_S \omega)(\theta,t)}{(\theta-r)^{2-\alpha}}d\theta\bigg)\\
    &=\frac{(-1)^{1-\alpha}}{\Gamma(\alpha)}\bigg(\frac{\tilde E (u(r)-u(s),\cdot)(\omega\otimes_S \omega)(r,t)}{(t-r)^{1-\alpha}}\\
&+(1-\alpha)\int_r^t\frac{\tilde E (u(r)-u(s),\cdot)(\omega\otimes_S
    \omega)(r,\theta)}{(\theta-r)^{2-\alpha}}d\theta\\
&+(1-\alpha)\int_r^t\frac{\int_r^\theta \int_\theta^t S(\xi-\tau) \tilde E (u(r)-u(s),\omega^\prime(\tau)) \otimes_V \omega^\prime(\xi)
     d\xi d\tau  }{(\theta-r)^{2-\alpha}}d\theta\bigg).
\end{align*}
Plugging the above expression into the previous expression of $\tilde E D_{t-}^{1-\alpha}(u\otimes
(\omega\otimes_S\omega)(t))(s,\cdot)_{t-}[r]$, we obtain
\begin{align*}
 \tilde E D_{t-}^{1-\alpha}&(u\otimes
(\omega\otimes_S\omega)(t))(s,\cdot)_{t-}[r] \\
    =&-\tilde E \dD_{t-}^{1-\alpha}(u\otimes(\omega\otimes_S
    \omega)(t))[r]+\tilde E (u(r)-u(s),\cdot)D_{t-}^{1-\alpha}
    (\omega\otimes_S \omega)(\cdot,t)_{t-}[r] .
\end{align*}
Note that the previous equality shows a similar connection between the fractional derivative and the compensated fractional derivative obtained previously in \eqref{rel2}.

Hence, using \eqref{q5} and the fractional integration by parts formula
\eqref{eq10b}, $(u\otimes\omega)$ satisfies the
equation
\begin{align*}
    (u\otimes\omega)(s,t)&=\int_s^t(S(\xi-s)-{\rm id})u(s)\otimes_V\omega^\prime(\xi)d\xi-(-1)^\alpha\int_s^t \hat D_{s+}^\alpha
    G(u(\cdot))[r] D_{t-}^{1-\alpha}(\omega\otimes_S\omega)(\cdot,t)_{t-}[r]dr\\
&+(-1)^\alpha\int_s^t D_{s+}^\alpha DG(u(\cdot))[r] \dD_{t-}^{1-\alpha}(u\otimes(\omega\otimes_S\omega)(t))[r]dr\\
    &=\int_s^t(S(\xi-s)-{\rm id})u(s)\otimes_V\omega^\prime(\xi)d\xi-(-1)^\alpha\int_s^t \hat D_{s+}^\alpha
    G(u(\cdot))[r] D_{t-}^{1-\alpha}(\omega\otimes_S\omega)(\cdot,t)_{t-}[r]dr\\
&+(-1)^{2\alpha-1}\int_s^t D_{s+}^{2\alpha-1} DG(u(\cdot))[r] D_{t-}^{1-\alpha}\dD_{t-}^{1-\alpha}(u\otimes(\omega\otimes_S\omega)(t))[r]dr
\end{align*}
which completes the proof.
\end{proof}

We note that the last two integrals in \eqref{uomega} are well-defined. In particular, we apply the operator $D_{t-}^{1-\alpha}(\omega\otimes_S\omega)(\cdot,t)_{t-}[r]$,  which is an element of the separable Hilbert space $L_2(L_2(V,V_\delta),V\otimes V)$, to $\hat D_{s+}^\alpha
    G(u(\cdot))[r]$, which is contained in the separable Hilbert space $L_2(V,V_\delta)$. Then we can use the definition of a Hilbert space valued integral of Section \ref{preli}. Similar we can argue for the last integral of \eqref{uomega}.
In the Appendix, we will prove that $(u\otimes \omega)$ given by \eqref{uomega} satisfies the
Chen equality (see Lemma \ref{l8}).\\

Let us now deal with the structure of $(u\otimes(\omega\otimes_S\omega))$.

\begin{lemma}\label{uomom}Suppose \eqref{finitesum} holds.
Let  $\tilde E\in L_2(V\otimes V,V_\delta)$. Then for $0\leq s \leq q\leq  t\leq T$ the expression $\tilde E(u\otimes (\omega\otimes_S\omega)(t))(s,q)$ satisfies the equation
\begin{align}\label{eq131}
\begin{split}
\tilde E(u\otimes (\omega\otimes_S\omega)(t))(s,q)&= -(-1)^\alpha \int_s^q\omega_S(r,t)\tilde E(u(r)-u(s),\omega^\prime(r))dr\\
&=-\int_s^q\hat D_{s+}^{\alpha}\omega_S(\cdot,t)\tilde E(u(\cdot)-u(s),\cdot)[r]  D_{q-}^{1-\alpha}\omega_{q-}[r]dr\\
&+(-1)^{\alpha-1}\int_s^q D_{s+}^{2\alpha-1}\tilde E(u(\cdot)-u(s),\cdot)[r]  D_{q-}^{1-\alpha}\dD_{q-}^{1-\alpha}(\omega_S(t)\otimes\omega)[r]
    dr\\&+(-1)^{\alpha-1}\int_s^q D_{s+}^{2\alpha-1}\omega_S(\cdot,t)[r] \tilde E D_{q-}^{1-\alpha}\dD_{q-}^{1-\alpha}(u\otimes\omega)(t)[r] dr,
\end{split}
\end{align}
where, for  $s\le\tau\le t$ and $E\in L_2(V,V_\delta)$, $(\omega_S(t)\otimes\omega)$ satisfies
\begin{align*}
E(\omega_S(t)&\otimes\omega)(s,\tau)=\int_s^\tau(\omega_S(r,t)-\omega_S(s,t))Ed\omega(r)\\
    =&\int_s^\tau(\omega_S(r,t)-\omega_S(r,\tau))Ed\omega(r)+\int_s^\tau\omega_S(r,\tau)Ed\omega(r)
    -\int_s^\tau\omega_S(s,t)Ed\omega(r)\\
    =&\omega_S(\tau,t)\int_s^\tau S(\tau-r)Ed\omega(r)+(-1)^{-\alpha}E(\omega\otimes_S\omega)(s,\tau)-\omega_S(s,t)E(\omega(\tau)-\omega(s)).
\end{align*}
\end{lemma}

Before proving this lemma, note that the expression of $E(\omega_S(t)\otimes\omega)(s,\tau)$ above is defined in the sense of \eqref{eq11} since $\omega$ is smooth.

\begin{proof}
We define $f_{\tilde E}:L_2(V_\delta,V\otimes V)\times V\to L_2(V,V \otimes
V)$ given by
\[
f_{\tilde E}(Q,u)=Q(\tilde E (u,\cdot) ).
\]
From \eqref{omegaSS} for smooth $\omega$ we have that
\begin{align*}
\tilde E(u\otimes(\omega\otimes_S\omega)(t))(s,q)=&
-(-1)^{\alpha} \int_s^q
\omega_S(r,t) \tilde
E(u(r)-u(s),\omega^\prime(r)) dr \\
=&-(-1)^{\alpha}\int_s^q f_{\tilde E}(\omega_S(r,t),u(r)-u(s))
\omega^\prime(r)dr.
\end{align*}
Following \eqref{rel} or Theorem 3.3 in \cite{HuNu09}, we have
\begin{align}\label{tensor3}
\begin{split}
\int_s^q f_{\tilde E}&(\omega_S(r,t),u(r)-u(s))
\omega^\prime(r)dr\\&=(-1)^\alpha\int_s^q \hat D_{s+}^\alpha
f_{\tilde
E}(\omega_S(\cdot,t),u(\cdot)-u(s))[r]D^{1-\alpha}_{q-}\omega_{q-}[r]dr\\
&-(-1)^\alpha \int_s^q (\omega_S(r,t)-\omega_S(s,t)) D^\alpha_{s+}
\tilde E(u(\cdot)-u(s),\cdot)[r] D^{1-\alpha}_{q-}\omega_{q-}[r] dr\\
&-(-1)^\alpha \int_s^q D^\alpha_{s+} \omega_S
(\cdot,t) [r]\tilde E(u(r)-u(s),\cdot)D^{1-\alpha}_{q-}\omega_{q-}[r] dr\\
& +\int_s^q Df_{\tilde
E}(\omega_S(r,t),u(r)-u(s))(\omega_S(r,t)-\omega_S(s,t),u(r)-u(s))
\omega^\prime(r)dr.
\end{split}
\end{align}
Now we calculate the derivative of $f_{\tilde E}$:
\begin{align}
\label{tensor31}
\begin{split}
Df_{\tilde
E}(\omega_S(r,t),u(r)-u(s))&(\omega_S(r,t)-\omega_S(s,t),u(r)-u(s))\omega^\prime(r)\\
    =&(\omega_S(r,t)-\omega_S(s,t))\tilde
    E(u(r)-u(s),\omega^\prime(r))+\omega_S(r,t)\tilde
    E(u(r)-u(s),\omega^\prime(r))\\
    =&\tilde
    E(u(r)-u(s),\cdot)D_2(\omega_S(t)\otimes\omega)(s,r)+\omega_S(r,t)\tilde E D_2(u\otimes\omega)(s,r).
\end{split}
\end{align}
Substituting the above expression in \eqref{tensor3}, after applying integration by parts to the last two terms, we have to calculate $ED^{1-\alpha}_{q-}(\omega_S(t)\otimes\omega)_{q-}(s,\cdot)[r]$ and $\tilde E D^{1-\alpha}_{q-} (u\otimes\omega)(t)(s,\cdot)_{q-}[r]$. First, we have

\begin{align*}
ED^{1-\alpha}_{q-}(\omega_S(t)&\otimes\omega)_{q-}(s,\cdot)[r]
=D^{1-\alpha}_{q-}E(\omega_S(t)\otimes\omega)_{q-}(s,\cdot)[r]\\
&=\frac{(-1)^{1-\alpha}}{\Gamma(\alpha)}
\bigg(\frac{E(\omega_S(t)\otimes \omega)(s,r)-E(\omega_S(t)\otimes\omega)(s,q)}{(q-r)^{1-\alpha}}\\
&\qquad + (1-\alpha)\int_r^q
\frac{E(\omega_S(t)\otimes\omega)(s,r)-E(\omega_S(t)\otimes\omega)(s,\theta)}{(\theta-r)^{2-\alpha}}
d\theta\bigg)\\
&=-\frac{(-1)^{1-\alpha}}{\Gamma(\alpha)}
\bigg(\frac{E(\omega_S(t)\otimes\omega)(r,q)+(\omega_S(r,t)-\omega_S(s,t))E(\omega(q)-\omega(r))}{(q-r)^{1-\alpha}}\\
&\qquad + (1-\alpha)\int_r^q
\frac{E(\omega_S(t)\otimes\omega)(r,\theta)+(\omega_S(r,t)-\omega_S(s,t))E(\omega(\theta)-\omega(r))}{(\theta-r)^{2-\alpha}}
d\theta\bigg)\\
&= -E{\mathcal D}^{1-\alpha}_{q-} (\omega_S(t)\otimes
\omega))[r]
+(\omega_S(r,t)-\omega_S(s,t))ED^{1-\alpha}_{q-}\omega_{q-}(r).
\end{align*}
Secondly, in a similar way, by \eqref{rel2}, we obtain that
\begin{align*}
\tilde E D^{1-\alpha}_{q-} (u\otimes\omega)(t)(s,\cdot)_{q-}[r]=-\tilde
E{\mathcal
D}^{1-\alpha}_{q-} (u\otimes\omega)(t)[r]-\tilde E(u(r)-u(s),D^{1-\alpha}_{q-}\omega_{q-}(r)),
\end{align*}
and substituting the last two expressions into \eqref{tensor3} we obtain the conclusion.
\end{proof}

\section{Global existence and uniqueness}

We now want to find an appropriate formulation for \eqref{eqn1}. As we said in the Introduction, $\omega$ is a H\"older continuous function of
order $\beta^\prime \in(1/3, 1/2)$, hence the integral with integrator $\omega$ is not well-defined in the classical sense. However, in what follows we will see that the two last terms in \eqref{sol} are well-defined when $\omega$ is $\beta^\prime$-H{\"o}lder continuous. A main difficulty in that point is that for a nonregular path $\omega$ we cannot expect that $(u\otimes \omega)$, which appears in the last term on \eqref{sol}, is well-defined. However, we are able to overcome these problems by formulating the term $(u\otimes \omega)$ by another operator equation. This will be possible thanks to the $2\beta^\prime$- H{\"o}lder continuity of $(\omega\otimes_S\omega)$, as we will explain below.

We now introduce for {\em every} $\beta^\prime$-H\"older continuous path $\omega$ the phase space in which we are looking for
solutions to the problem \eqref{eqn1}:
\begin{align*}
    & W(T_1,T_2)=C_{\beta}([T_1,T_2];V)\times C_{\beta+\beta^\prime}(\Delta_{T_1,T_2};V  \otimes V)
\end{align*}
for $0\leq T_1<T_2$, with seminorm
\begin{align*}
    &|||U|||=\|u\|_{\beta}+\|v\|_{\beta+\beta^\prime},\quad U=(u,v)\in
    W(T_1,T_2),
\end{align*}
and such that the Chen equality holds for $U$, which means that
for $0\leq T_1\le s\le r\le t\le T_2$,
\begin{equation}\label{chenbis}
   v(s,r)+v(r,t)+(u(r)-u(s))\otimes_V(\omega(t)-\omega(r))=
   v(s,t),
\end{equation}
where $\omega$ denotes a fixed $\beta^\prime$--H\"older path with
$\beta^\prime\in (1/3,1/2)$.
Note that, when we consider for $u$ the subset of functions with a fixed value say at $T_1$, the expression $|||U|||$ generates a complete metric, see Section \ref{preli}. For the metric $d_W(U_1,U_2)$ we will write $|||U_1-U_2|||$.

When $u_1(T_1)\not= u_2(T_1)$, $W(T_1,T_2)$ becomes a Banach space if we add $|u(T_1)|$ to  $|||U|||$. However, as we will see below, see Remark \ref{ic}, it suffices to work with the seminorm $|||\cdot|||$ as we have already defined.

Recall that the spaces appearing in
the definition of $W(T_1,T_2)$ as well as the corresponding norms were also introduced in Section \ref{preli}.\\

In order to establish the existence and uniqueness of solutions to  \eqref{eqn1} we set the following Hypothesis ${\bf H}$:

\begin{enumerate}

\item Assume that $S$ is an analytic semigroup with generator $A$ on  the separable Hilbert space $V$.
We assume that $A$ is a negative operator which generates a complete basis of eigenelements
$(e_i)_{i\in\mathbb{N}}$ of $V$. For the associated spectrum $(\lambda_i)_{i\in \mathbb N}$ of $A$ suppose that
\begin{equation}\label{finitesum2}
    \sum_i\lambda_i^{-2\delta}<\infty,
\end{equation}
where $\delta\in [0,1]$ is an appropriate parameter which rank value will be determined later.

Assume $G:V\mapsto L_2(V, V_\delta)$ is a three times Fr\'echet-differentiable mapping, bounded and with bounded derivatives as in Lemma \ref{l3}, such that $DG(\cdot)\in L_2(V\otimes V,V_\delta)$.

\item
Suppose that $1/3< H \le 1/2$ and $1/3<\beta<H$. Suppose that there is an $\alpha$ such that $1-\beta< \alpha <2 \beta$,
$\alpha< \frac{\beta+1}{2}$.

\item Let $\omega\in C_{\beta^\prime}([T_1,T_2];V)$ for any $\beta<\beta^\prime<H$.

\item
Let $(\omega^n)_{n\in\mathbb{N}}$ be a sequence  of piecewise smooth functions
with values in $V$ such that
$((\omega^n\otimes_S\omega^n))_{n\in\mathbb{N}}$ is defined by
\eqref{omegaSS}. Assume then that for any $\beta^\prime<H$ the sequence
$((\omega^n,(\omega^n\otimes_S\omega^n)))_{n\in\mathbb{N}}$ converges to
$(\omega,(\omega\otimes_S\omega))$ in
$C_{\beta^\prime}([T_1,T_2];V)\times C_{2\beta^\prime} (\Delta_{T_1,T_2}; L_2(L_2(V,V_\delta),V\otimes V))$, where $\omega\otimes_S\omega$ is {\em defined} to be this limit.

\end{enumerate}

\begin{remark}\label{rex1}
(i) In item (1) above note that the negativeness of the operator
$A$ is not a restriction, since otherwise we could consider
$A-c\, \rm{id}$ instead of $A$, provided that the $dt$-nonlinearity $c\, \rm{id}$ would be also added into the equation, being $c$ a positive constant.

\smallskip

(ii) As a consequence of Section \ref{regular} and item (4), in addition to $(\omega \otimes_S \omega) \in C_{2\beta^\prime} (\Delta_{T_1,T_2}; L_2(L_2(V,V_\delta),V\otimes V))$, the Chen equality holds in the following way: for $E\in L_2(V,V_\delta),\,T_1\le s\le r\le t\le T_2$ we have
\begin{align*}\begin{split}
&E(\omega\otimes_S \omega)(s,r) +E(\omega\otimes_S
\omega)(r,t) +(-1)^{-\alpha} \omega_S(r,t)\int_s^rS(r-q)Ed\omega(q)=E(\omega\otimes_S \omega)(s,t) .
\end{split}
\end{align*}
The integrals $\int_s^rS(r-q)Ed\omega(q)$ and
$\omega_S(r,t)\cdot =(-1)^{\alpha}\int_r^tS(\xi-r)\cdot \otimes_V
d\omega(\xi)$ are well defined in the Weyl sense if $S$ and $\omega$
satisfy the above properties, see
Lemma \ref{lex7} below.
\smallskip

(iii) We have  that
\begin{align}\label{q1}
    |D_{t-}^{1-\alpha}\omega_{t-}[r]|\le
    c\|\omega\|_{\beta^\prime}(t-r)^{\alpha+\beta^\prime-1},
\end{align}
which follows easily from the H{\"o}lder condition on $\omega$, item (3) above.

\smallskip

(iv) To interpret \eqref{eqn1} as a stochastic partial
differential equation we can assume that $\omega$ is given by a
fractional Brownian motion with Hurst parameter $H\in (1/3,1/2]$, see Section \ref{fBm} for a detailed description of this case.

\end{remark}

Now we can define what is understood as a solution to \eqref{eqn1}. For the sake of brevity, we consider solutions only on the interval $[0,1]$.

\begin{definition}\label{deff}
Under the Hypothesis {\bf H}, assuming that $u_0\in V$, a mild solution of \eqref{eqn1} is a pair $U=(u,v)\in W(0,1)$ satisfying
\begin{align}\label{equ1}
\begin{split}
    u(t)&=S(t)u_0+(-1)^\alpha\int_0^t\hat D_{0+}^\alpha
    (S(t-\cdot)G(u(\cdot)))[r]D_{t-}^{1-\alpha}\omega_{t-}[r]dr\\
    &-(-1)^{2\alpha-1}\int_0^t {D}_{0+}^{2\alpha-1}
    (S(t-\cdot)DG(u(\cdot)))[r]D_{t-}^{1-\alpha}\dD_{t-}^{1-\alpha}v[r] dr,
 \end{split}
\end{align}
\begin{align}\label{equ2}
\begin{split}
    v(s,t)&=\int_s^t(S(\xi-s)-{\rm id})u(s)\otimes_Vd\omega(\xi)\\
    &-(-1)^\alpha\int_s^t \hat D_{s+}^\alpha
    G(u(\cdot))[r]  D_{t-}^{1-\alpha}(\omega\otimes_S\omega)(\cdot,t)_{t-}[r] dr\\
    &+(-1)^{2\alpha-1}\int_s^t D_{s+}^{2\alpha-1}
    DG(u(\cdot))[r] D_{t-}^{1-\alpha}\dD_{t-}^{1-\alpha}(u\otimes(\omega\otimes_S\omega)(t))[r] dr,
\end{split}
\end{align}
for $0\leq s<t\leq 1$. The term $(u\otimes (\omega\otimes_S\omega)(t))(s,t)$
can be defined by the right-hand side of \eqref{eq131} where we  replace $(u\otimes \omega)$ by $v$, that is, for $\tilde E\in L_2(V\otimes V, V_\delta)$,
\begin{align}\label{eq131bis}
\begin{split}
\tilde E(u\otimes (\omega\otimes_S\omega)(t))(s,q)&=-\int_s^q\hat D_{s+}^{\alpha}\omega_S(\cdot,t)\tilde E(u(\cdot)-u(s),\cdot)[r]  D_{q-}^{1-\alpha}\omega_{q-}[r]dr\\
&+(-1)^{\alpha-1}\int_s^q D_{s+}^{2\alpha-1}\tilde E(u(\cdot)-u(s),\cdot)[r]  D_{q-}^{1-\alpha}\dD_{q-}^{1-\alpha}(\omega_S(t)\otimes\omega)[r]
    dr\\&+(-1)^{\alpha-1}\int_s^q D_{s+}^{2\alpha-1}\omega_S(\cdot,t)[r] \tilde E D_{q-}^{1-\alpha}\dD_{q-}^{1-\alpha}v[r] dr,
\end{split}
\end{align}
where $(\omega_S(t)\otimes\omega)$ is defined for  $s\le\tau\le t,\,E\in L_2(V,V_\delta)$ as
\begin{align}\label{xxxeqbis}
\begin{split}
E(\omega_S(t)\otimes\omega)(s,\tau)=&\omega_S(\tau,t)\int_s^\tau S(\tau-r)Ed\omega(r)+(-1)^{-\alpha}E(\omega\otimes_S\omega)(s,\tau)-\omega_S(s,t)E(\omega(\tau)-\omega(s))\\
 =&\omega_S(\tau,t)\int_s^\tau (S(\tau-r)-{\rm id})Ed\omega(r)+(-1)^{-\alpha}E(\omega\otimes_S\omega)(s,\tau)\\
 &-(\omega_S(s,t)-\omega_S(\tau,t))E(\omega(\tau)-\omega(s)).
\end{split}
\end{align}

\end{definition}

\begin{remark}
(i) The notation of the tensor given by \eqref{eq131bis} has been inherited by its analogous counterpart in the finite-dimensional case, see \cite{HuNu09}. However, we would like to point that, despite its name, the tensor $u\otimes (\omega\otimes_S\omega)$ depends on $v$ as can be seen in its definition.
\smallskip

(ii) We want to emphasize that $v$ from this section on is an own standing variable and not a function of $u$ and $\omega$, as it happens in Section \ref{regular}.
\smallskip

(iii) The above definition of the solution does not depend of the value of $\alpha$.

\end{remark}
We denote  the right-hand side of the system \eqref{equ1}-\eqref{equ2} by  ${\mathcal T}(U)=({\mathcal
T}_1(U),{\mathcal T}_2(U))$. For the sum of both integrals
in \eqref{equ1} over an interval $[s,t]$ we use the abbreviation
\begin{equation*}
    \int_s^tS(t-r)G(u(r))d\omega(r).
\end{equation*}
Observe that in the previous definition, the equations for the
second component $v$ as well as $(u\otimes
(\omega\otimes_S\omega))$ can be regarded as the generalizations
to the non-regular case of the equations we have in the regular
case for $(u\otimes \omega)$ and $(u\otimes
(\omega\otimes_S\omega))$ (see Section \ref{regular}). Due to
Lemma \ref{l8}  it is therefore natural to assume that $v$
given by \eqref{equ2} satisfies \eqref{chenbis}, and which is explicitly required when saying that the pair $U=(u,v) \in W(0,1)$.

\begin{remark}\label{exr2}
We have defined the solution with  respect to the time interval
$[0,1]$. Similarly we can define a solution  on $[T_1,T_2]$, $0<T_1<T_2$, when the
initial condition is given at $T_1$ and the space
for the solution is $W(T_1,T_2)$.
\end{remark}

At this point, it is crucial to stress that it is possible to give sense to all integrals appearing in our definition of solution, namely, to every integral in \eqref{equ1}-\eqref{equ2}. As it was shown in Section \ref{preli} and Section \ref{regular}, in order to do that it is necessary to make use of the $L_2$ spaces of Hilbert--Schmidt operators. Note that in \eqref{equ1}, the operators $\hat D_{0+}^\alpha
    (S(t-\cdot)G(u(\cdot)))[r]$ and ${D}_{0+}^{2\alpha-1}
    (S(t-\cdot)DG(u(\cdot)))[r]$ are applied, respectively, to an element in $V$ and $V\otimes V$. However, in \eqref{equ2} this is different. There the operators $(\omega\otimes_S \omega)$ and $(u\otimes(\omega\otimes _S\omega))$, considered to be contained in $L_2(L_2(V,V_\delta), V\otimes V)$ and $L_2(L_2(V\otimes V, V_\delta), V\otimes V)$, are applied to the elements $\hat D_{s+}^\alpha G(u(\cdot))[r]$ and ${D}_{s+}^{2\alpha-1} DG(u(\cdot))[r]$. We stress that this last regularity property of the tensor $(u\otimes (\omega \otimes_S \omega))$ is not part of the Hypothesis {\bf H} and will be analyzed in Lemma \ref{l7}.

Now we establish some properties of
$\omega_S$, which are consequences of \eqref{eq10bis} due to
the regularity of the semigroup.

\begin{lemma}\label{lex7}
Under the Hypothesis ${\bf H}$ the following statements hold:

(i) For $0\le s\le r\le t\le 1,\,e\in V$ and $1/3<\beta^\prime <\beta^{\prime\prime}<H$ we have that
\begin{align*}
\|\omega_S(r,t)e-\omega_S(s,t)e\|& \le c|r-s|^{\beta^\prime}(\|\omega\|_{\beta^\prime}+\|\omega\|_{\beta^{\prime\prime}})|e|,\\
\|\omega_S(s,t)e\|&\le c|t-s|^{\beta^\prime}\|\omega\|_{\beta^{\prime}} |e|,\\
\|\omega_S(s,t)(-A)^{\beta^\prime} e\|&\le c|e|.
\end{align*}

(ii) The mapping
\begin{equation*}
    E\in L_2(V,V_\delta) \mapsto I(E):=  \int_s^t S(t-r)Ed\omega(r)
\end{equation*}
is in $L_2(L_2(V,V_\delta),V)$, with norm bounded by $c\|\omega\|_{\beta^\prime}(t-s)^{\beta^\prime}\|E\|_{L_2(V,V_\delta)}$.

\end{lemma}

\begin{proof}
Let $\beta^\prime <\beta^{\prime\prime}<H$ such that for
$\alpha<\alpha^{\prime}<1$ we have
\begin{equation*}
    \beta^\prime+\alpha^{\prime}<1<\beta^{\prime\prime}+\alpha.
\end{equation*}
Then
\begin{equation}
\begin{split}\label{ex1m}
    \omega_S(r,t)e&-\omega_S(s,t)e\\=&(-1)^{\alpha}\int_r^t(S(\xi-r)e - S(\xi-s)e)\otimes_Vd\omega(\xi)-(-1)^{\alpha}\int_s^rS(\xi-s)e\otimes_Vd\omega(\xi).
    \end{split}
\end{equation}
Taking Lemma \ref{l0} into account we obtain
\begin{align*}
    |D^{\alpha}_{r+}(S(\cdot-r)e-S(\cdot-s)e)[\xi]|&\le c \bigg(\frac{(r-s)^{\beta^\prime}}{(\xi-r)^{\alpha +\beta^\prime}}
    +\alpha
    \int_r^\xi\frac{(r-s)^{\beta^\prime}(\xi-q)^{\alpha^{\prime}}}
    {(\xi-q)^{1+\alpha}(q-r)^{\alpha^{\prime}+\beta^\prime}}dq\bigg)|e|\\
    &\le c(r-s)^{\beta^\prime}(\xi-r)^{-\alpha -\beta^\prime}|e|.
\end{align*}
On the other hand, since $\omega$ is $\beta^\prime$-H\"older continuous it is also true that $\omega$ is $\beta^{\prime \prime}$-H\"older continuous. Moreover, we can write \eqref{q1} in the following way
\begin{align*}
    &|D^{1-\alpha}_{t-}\omega_{t-}[\xi]|\le c\|\omega\|_{\beta^{\prime\prime}}(t-\xi)^{\alpha+\beta^{\prime\prime}-1}.
\end{align*}
Hence, by applying Lemma \ref{l1-m} above, the first integral on the
right-hand side of \eqref{ex1m} is bounded in particular by
$c|r-s|^{\beta^\prime} \|\omega\|_{\beta^{\prime\prime}}|e|$.

Furthermore, for the last term in \eqref{ex1m} and  $1-\beta^\prime <\alpha<\alpha^{\prime}<1$, in a similar
way we obtain
\begin{align*}
    \int_ s^r
    \|D^{\alpha}_{s+}S(\cdot-s)e[\xi]&\otimes_VD^{1-\alpha}_{r-}\omega_{r-}[\xi]\|d\xi
    \le \int_ s^r |D^{\alpha}_{s+}S(\cdot-s)e[\xi]|
    |D^{1-\alpha}_{r-}\omega_{r-}[\xi]|d\xi \\
    &\leq c\|\omega\|_{\beta^{\prime}} |e| \int_
    s^r\bigg((\xi-s)^{-\alpha}+\int_s^\xi\frac{(\xi-q)^{\alpha^{\prime}}}{(q-s)^{\alpha^{\prime}}(\xi-q)^{\alpha+1}} dq\bigg)
    (r-\xi)^{\alpha+\beta^{\prime}-1}d\xi
    \\&\le c\|\omega\|_{\beta^{\prime}}|r-s|^{\beta^\prime}|e|.
\end{align*}
The second statement of (i) follows directly from the last inequality taking $r=t$.
For the last conclusion of (i), for the parameters chosen at the beginning of the proof,
  \begin{align*}
   |D_{s+}^{\alpha}(S(\cdot-s)(-A)^{\beta^\prime}e)[\xi]|\le c\bigg(\frac{|e|}{
(\xi-s)^{\alpha+\beta^\prime}}+\int_s^\xi\frac{(\xi-q)^{\alpha^{\prime}}|e|}{(\xi-q)^{1+\alpha}(q-s)^{\alpha^{\prime}+\beta^\prime}}dq\bigg)\leq c |e| (\xi-s)^{-\alpha-\beta^\prime},
\end{align*}
and therefore, since in particular $\beta^\prime+\alpha <1$,
\begin{align*}
    \int_ s^t
    \|D_{s+}^{\alpha}(S(\cdot-s)(-A)^{\beta^\prime}e)[\xi]&\otimes_VD^{1-\alpha}_{t-}\omega_{t-}[\xi]\|d\xi
   \leq c\|\omega\|_{\beta^{\prime \prime }} |e| \int_
    s^t  (\xi-s)^{-\alpha-\beta^\prime}    (t-\xi)^{\alpha+\beta^{\prime \prime}-1}d\xi
    \\&\le c\|\omega\|_{\beta^{\prime \prime}}(t-s)^{\beta^{\prime \prime} -\beta^\prime}|e| \leq  c |e|.
\end{align*}
Now we prove (ii). First of all, note that $I(\cdot)$ is well-defined since, thanks to  \eqref{eq10bis} and Remark \ref{sing}, it suffices the local H\"older continuity of $r\mapsto S(t-r)E$.
In addition the mapping
\begin{equation*}
    L_2(V, V_\delta) \ni E\mapsto ED_{t-}^{1-\alpha}\omega_{t-}[r]
\end{equation*}
is in $L_2(L_2(V,V_\delta),V)$, and the norm respect to this space is  $(\sum_i\lambda_i^{-2\delta})^{\frac12}|D_{t-}^{1-\alpha}\omega_{t-}[r]|$ which can be estimated by (\ref{finitesum2}) and \eqref{q1}. Furthermore, the integrand $D^\alpha_{s+} S(t-\cdot) [r]\cdot D^{1-\alpha}_{t-} \omega_{t-} [r]$ is weakly measurable  with respect to the separable Hilbert space $L_2(L_2(V,V_\delta),V)$ such that by Pettis' theorem
the integrand is measurable. Moreover, thanks to Lemma \ref{l1-m} above, and the fact that $D_{t-}^{1-\alpha}E \omega_{t-} [r]=ED_{t-}^{1-\alpha} \omega_{t-} [r]$, we get
\begin{align*}
   \|I(\cdot)&\|_{L_2(L_2(V,V_\delta),V)}=\bigg\| \int_s^t S(t-r)\cdot d\omega(r)\bigg\|_{L_2(L_2(V,V_\delta),V)}\\& \leq  \int_s^t \| D^\alpha_{s+} S(t-\cdot) [r]\|_{L(V)} \|\cdot D^{1-\alpha}_{t-} \omega_{t-} [r]\|_{L_2(L_2(V,V_\delta),V)}dr \le c\|\omega\|_{\beta^{\prime}} (t-s)^{\beta^\prime}.
   \end{align*}
\end{proof}
\begin{lemma}\label{l7} Under the Hypothesis ${\bf H}$, for $0 \leq s\le r \leq q\le  t\le 1$, $\tilde E\in L_2(V\otimes V,V_\delta)$  and $U=(u,v)\in
W(0,1)$ for the mapping
\begin{equation*}
    L_2(V\otimes V,V_\delta)\ni \tilde E\mapsto \tilde E(u\otimes (\omega \otimes_S \omega)(t))(s,q)\in L_2(L_2(V\otimes V,V_\delta),V\otimes V)
\end{equation*}
the inequality
\begin{align}\label{neu1}
&\|(u\otimes (\omega \otimes_S \omega)(t))(s,q)\|_{L_2(L_2(V\otimes V,V_\delta),V\otimes V)}\le c |||U||| (q-s)^{\beta+\beta^\prime}(t-s)^{\beta^\prime}
    \end{align}
holds.
Moreover, the expression $(u\otimes (\omega \otimes_S \omega))$ satisfies the Chen equality $\eqref{chen3fold}$ and its generalized form \eqref{3tensorChen}. In addition,
\begin{align}\label{neu2}
&\|D_{t-}^{1-\alpha}
\dD_{t-}^{1-\alpha}(u\otimes(\omega\otimes_S\omega)(t))[r]\|_{L_2(L_2(V\otimes V,V_\delta),V\otimes V)}  \leq c |||U||| (t-r)^{\beta+2\beta^\prime+2\alpha-2}.
\end{align}
\end{lemma}

\begin{proof} Let us consider separately the three terms of  $(u\otimes (\omega \otimes_S \omega)(t))(s,q)$ given in \eqref{eq131bis}.
We start with
\begin{align*}
I_1(\tilde E):=\int_s^q D_{s+}^{2\alpha-1}\omega_S(\cdot,t)[r]
D_{q-}^{1-\alpha}\dD_{q-}^{1-\alpha}\tilde E v[r]
dr.
\end{align*}
For a fixed $v\in V\otimes V$ the mapping
\begin{equation}\label{eq21}
    L_2(V\otimes V,V_ \delta)\ni \tilde E\mapsto  \tilde Ev
\end{equation}
is in $L_2(L_2(V\otimes V,V_ \delta), V)$. The norm with respect to this space is $(\sum_{i=1}^\infty\lambda_i^{-2\delta})^\frac12\|v\|$.
Then, since Lemma \ref{lex7} (i) in particular implies that $D_{s+}^{2\alpha-1}\omega_S(r,t)$ is in $L(V,V\otimes V)$, then $\tilde E\mapsto I_1(\tilde E)$ is a mapping in $L_2(L_2(V\otimes V,V_ \delta),V\otimes V)$.
By Pettis' Theorem it is not hard to see  that the integrand of $I_1$ is weakly measurable and hence measurable as a mapping from $[s,q]$ into the separable linear space $L_2(L_2(V\otimes V,V_ \delta),V\otimes V)$.
Moreover, we have
\begin{align*}
  \|I_1(\cdot)\|_{L_2(L_2(V\otimes V,V_\delta),V\otimes V)}& \leq \int_s^q \bigg\| D_{s+}^{2\alpha-1}\omega_S(\cdot,t)[r]D_{q-}^{1-\alpha}{\mathcal D}_{q-}^{1-\alpha}\cdot v[r]\bigg\|_{L_2(L_2(V\otimes V,V_\delta),V\otimes V)}dr\\
  &\leq \int_s^q \|D_{s+}^{2\alpha-1}\omega_S(\cdot,t)[r]\|_{L(V,V\otimes V)} \|D_{q-}^{1-\alpha}{\mathcal D}_{q-}^{1-\alpha}\cdot v[r]\|_{L_2(L_2(V\otimes V,V_ \delta), V)}dr.
  \end{align*}
In order to estimate the second factor in the integrand of $I_1$, we note that
\begin{align}\label{vfd2}
D_{q-}^{1-\alpha}{\mathcal D}_{q-}^{1-\alpha}\tilde E v[r]=\tilde E D_{q-}^{1-\alpha}{\mathcal D}_{q-}^{1-\alpha} v[r]
\end{align}
which follows easily simply exchanging the integrals in the definition of the fractional derivatives and $\tilde E$.
Under the Hypothesis ${\bf H}$, for $r\in (s,q)$ and $v\in C_{\beta+\beta^\prime}(\Delta_{0,1};V
\otimes V)$,
\begin{align*}
  &\|\dD_{q-}^{1-\alpha}v[r]\|\le
  c\|v\|_{\beta+\beta^\prime}(q-r)^{\beta+\beta^\prime+\alpha-1}.
\end{align*}
Then, thanks to \eqref{chenbis}, for $r\in (s,q)$ and $U\in
    W(0,1)$,  we obtain
\begin{align}\label{q3}
    &\|D_{q-}^{1-\alpha}\dD_{q-}^{1-\alpha}v[r]\|\le
    c(\|v\|_{\beta+\beta^\prime}+\|u\|_{\beta}\|\omega\|_{\beta^\prime})(q-r)^{\beta+\beta^\prime +2\alpha-2}.
\end{align}
In fact, \eqref{eq21} together with \eqref{vfd2} immediately implies that
\begin{align*}
\|D_{q-}^{1-\alpha}{\mathcal D}_{q-}^{1-\alpha}\cdot v[r]\|_{L_2(L_2(V\otimes V,V_ \delta),V)} \leq c (\|v\|_{\beta+\beta^\prime}+\|u\|_{\beta}\|\omega\|_{\beta^\prime}) (q-r)^{\beta+\beta^\prime +2\alpha-2}.
\end{align*}
On the other hand, by Lemma \ref{lex7} (i),
\begin{equation*}
   \|D_{s+}^{2\alpha-1}\omega_S(\cdot,t)[r]\|_{L(V,V\otimes V)}\le c\bigg(\frac{(t-r)^{\beta^\prime}}{(r-s)^{2\alpha-1}}
    +\int_s^r\frac{(r-q)^{\beta^\prime}}{(r-q)^{2\alpha}}dq\bigg)(\|\omega\|_{\beta^\prime}+\|\omega\|_{\beta^{\prime\prime}}).
\end{equation*}
Combining the previous estimates we can conclude
\begin{align}\label{i1}
\begin{split}
  &  \bigg\|\int_s^q D_{s+}^{2\alpha-1}\omega_S(\cdot,t)[r]D_{q-}^{1-\alpha}{\mathcal D}_{q-}^{1-\alpha}\cdot v[r]dr\bigg\|_{L_2(L_2(V\otimes V,V_\delta),V\otimes V)}\\
  \le &c(\|v\|_{\beta+\beta^\prime}+\|u\|_{\beta}\|\omega\|_{\beta^\prime})(\|\omega\|_{\beta^\prime}+\|\omega\|_{\beta^{\prime\prime}})(t-s)^{\beta^\prime}(q-s)^{\beta+\beta^\prime}.
  \end{split}
\end{align}
Next we deal with
\begin{align*}
I_2(\tilde E):=\int_s^q \hat
D_{s+}^{\alpha}\omega_S(\cdot,t)\tilde
E(u(\cdot)-u(s),\cdot)[r]D_{q-}^{1-\alpha}\omega_{q-}[r]dr.
\end{align*}
First of all, the mapping $\tilde E \in L_2(V\otimes V, V_\delta) \mapsto I_2(\tilde E)$ is in $L_2(L_2(V\otimes V,V_\delta),V\otimes V)$, since for $H>\beta^{\prime\prime} >\beta^\prime$ the integrand of $I_2$ satisfies
\begin{align}
\label{eq23}
\begin{split}
&\bigg(\sum_{i,j,k}\|\hat
D_{s+}^{\alpha}\omega_S(\cdot,t)\tilde E_{ijk}(u(\cdot)-u(s),D^{1-\alpha}_{q-}\omega_{q-})[r]\|^2\bigg)^\frac12\\
\leq  & c\bigg(\sum_{i,j,k}\bigg\|  \frac{\omega_S(r,t)\tilde E_{ijk}(u(r)-u(s),D^{1-\alpha}_{q-}\omega_{q-}[r])}
    {(r-s)^\alpha}\bigg\|^2\\
&+\alpha \bigg\| \int_s^r\frac{(\omega_S(r,t)-\omega_S(\theta,t))\tilde E_{ijk}(u(r)-u(\theta),D^{1-\alpha}_{q-}\omega_{q-}[r])}{(r-\theta)^{1+\alpha}}
    d\theta \bigg\|^2\bigg)^\frac12\\
\leq & c\|u\|_{\beta}\|\omega\|_{\beta^\prime}(\|\omega\|_{\beta^\prime}+\|\omega\|_{\beta^{\prime\prime}}) \bigg(\sum_i\frac{1}{\lambda_i^{2\delta}}\bigg)^\frac12(
(t-r)^{\beta^\prime}(r-s)^{\beta-\alpha}+
(r-s)^{\beta^\prime+\beta-\alpha})(q-r)^{\alpha+\beta^\prime-1}.
\end{split}
\end{align}
Again, by Pettis' Theorem the integrand of $I_2$ is measurable. Therefore (\ref{eq23}) gives
\begin{align*}
 \|I_2(\cdot)\|_{L_2(L_2(V\otimes V,V_\delta),V\otimes V)} & \leq \int_s^q \bigg\| \hat
D_{s+}^{\alpha}\omega_S(\cdot,t) \cdot (u(\cdot)-u(s))[r]D_{q-}^{1-\alpha}\omega_{q-}[r]\bigg\|_{L_2(L_2(V\otimes V,V_\delta),V\otimes V)}dr\\
&\leq c \|u\|_{\beta}\|\omega\|_{\beta^\prime}(\|\omega\|_{\beta^{\prime\prime}}+\|\omega\|_{\beta^\prime})
(q-s)^{\beta+\beta^\prime}(t-s)^{\beta^\prime}.
\end{align*}
Now we estimate
\begin{align}\label{eq34}
I_3(\tilde E):=\int_s^qD_{s+}^{2\alpha-1}\tilde E(u(\cdot)-u(s),\cdot))[r]
D_{q-}^{1-\alpha}\dD_{q-}^{1-\alpha}(\omega_S(t)\otimes\omega)[r] dr.
\end{align}
We can split the previous integral into three integrals due to (\ref{xxxeqbis}). To treat the corresponding first expression let us write down the following estimate for $\alpha<\gamma<1,\,\beta^\prime<\gamma$:
\begin{align*}
  |D_{s+}^\alpha ((S(\tau-\cdot)-{\rm id})(-A)^{-\beta^\prime}e)[r]| &\le  c\bigg(\frac{|(S(\tau-r)-{\rm id})(-A)^{-\beta^\prime}e|}{(r-s)^{\alpha}}+
\int_s^r\frac{|(S(\tau-r)-S(\tau-q))(-A)^{-\beta^\prime}e|}{(r-q)^{1+\alpha}}dq \bigg)\\
  \le & c\bigg(\frac{(\tau-r)^{\beta^\prime}}{(r-s)^\alpha}+\int_s^r\frac{(|(S(r-q)-{\rm id})(-A)^{-\beta^\prime}S(\tau-r)e|}{(r-q)^{1+\alpha}}dq\bigg)\\
  \le &c\bigg(\frac{(\tau-r)^{\beta^\prime}}{(r-s)^\alpha}+\frac{(\tau-r)^{\beta^\prime-\gamma}}{(r-s)^{\alpha-\gamma}}\bigg)|e|.
\end{align*}
By the third statement of Lemma \ref{lex7} (i), for $2\beta^\prime >\gamma -\alpha$ we conclude that
\begin{align*}
&\|\omega_S(\tau,t) \int_s^\tau (S(\tau-r)-{\rm id})\cdot d\omega(r)\|_{L_2(L_2(V,V_\delta),V\otimes V)}\\
&=\|\omega_S(\tau,t)(-A)^{\beta^\prime}\int_s^\tau (S(\tau-r)-{\rm id})(-A)^{-\beta^\prime}\cdot d\omega(r)\|_{L_2(L_2(V,V_\delta),V\otimes V)}\\
&\le c \bigg\|\int_s^\tau (S(\tau-r)-{\rm id})(-A)^{-\beta^\prime}\cdot d\omega(r)\bigg\|_{L_2(L_2(V,V_\delta),V)}\\
   &\le c \int_s^\tau \bigg(\sum_{i,j}\bigg|D_{s+}^\alpha((S(\tau-r)-{\rm id})(-A)^{-\beta^\prime}\frac{e_i}{\lambda_i^\delta})[r] (e_jD_{\tau-}^{1-\alpha}\omega_{\tau-} [r])\bigg|^2\bigg)^\frac12dr\\
  &\le c\int_s^\tau\bigg(\sum_i\frac{1}{\lambda_i^{2\delta}}\bigg(\frac{(\tau-r)^{\beta^\prime}}{(r-s)^\alpha}+\frac{(\tau-r)^{\beta^\prime-\gamma}}{(r-s)^{\alpha-\gamma}}\bigg)^2|D_{\tau-}^{1-\alpha}\omega_{\tau-}[r]|^2\bigg)^\frac12dr\\
  &\le c\bigg(\sum_i\frac{1}{\lambda_i^{2\delta}}\bigg)^\frac12\int_s^\tau\bigg(\frac{(\tau-r)^{\beta^\prime}}{(r-s)^\alpha}+\frac{(\tau-r)^{\beta^\prime-\gamma}}{(r-s)^{\alpha-\gamma}}\bigg)(\tau-r)^{\beta^\prime+\alpha-1}dr\le c(\tau-s)^{2\beta^\prime}.
\end{align*}
For the other terms of the right-hand side of \eqref{xxxeqbis} we have by Hypothesis {\bf H}, item (4), and Lemma \ref{lex7} that
\begin{equation*}
\|(-1)^{-\alpha}\cdot(\omega\otimes_S\omega)(s,\tau)-(\omega_S(s,t)-\omega_S(\tau,t))\cdot(\omega(\tau)-\omega(s))\|_{L_2(L_2(V,V_\delta),V\otimes V)} \le c(\tau-s)^{2\beta^\prime}.
\end{equation*}
Since $\omega^n$ is smooth, the expression $\omega^n_S(t)\otimes \omega^n$ satisfies the Chen equality, see Remark \ref{r1}. Moreover, we have the convergence of $\omega^n_S(t)\otimes \omega^n$ to $\omega_S(t)\otimes\omega$ in $L_2(L_2(V,V_\delta),V\otimes V)$ such that the latter term satisfies the Chen equality too. This Chen equality and  the regularity of $\omega_S(t)\otimes\omega$
yields
\begin{align*}
    &\bigg\|D_{q-}^{1-\alpha}\dD_{q-}^{1-\alpha}(\omega_S(t)\otimes\omega)\bigg\|_{L_2(L_2(V,V_\delta),V\otimes V)}\le
    c(q-r)^{2\beta^\prime+2\alpha-2}
\end{align*}
which allows us to treat the  integral $I_3(\tilde E)$.
Let $(\tilde E_{ijk})_{i,j,k\in\NN}$ and $(E_{ik})_{i,k \in\NN}$ be the orthonormal basis of $L_2(V\otimes V;V_\delta)$ and $L_2(V;V_\delta)$ constructed by $(e_i)_{i\in\NN}$, see Section \ref{regular}, then
\begin{align*}
\|I_{3}(\cdot)(s,q)\|_{L_2(L_2(V\otimes V;V_\delta);V\otimes V)}\le
 & \int_s^q\bigg(\sum_{i,j,k}\|D_{q-}^{1-\alpha}\dD_{q-}^{1-\alpha}
 (\omega_S(t)\otimes\omega)[r]
   D_{s+}^{2\alpha-1}\tilde E_{ijk} (u(\cdot)-u(s),\cdot)[r]\|^2\bigg)^\frac12dr \\
  \le& \int_s^q\bigg(\sum_{i,j}\|D_{q-}^{1-\alpha}\dD_{q-}^{1-\alpha}(E_{ik}\omega_S(t)\otimes\omega)[r]\|^2   |D_{s+}^{2\alpha-1}(u(\cdot)-u(s))[r]|^2\bigg)^\frac12 dr\\
  \le&c|||U|||\bigg(\sum_{i}\frac{1}{\lambda_i^{2\delta}}\bigg)^\frac12\int_s^q(r-s)^{\beta-2\alpha+1}
  (q-r)^{2\alpha+2\beta^\prime-2}dr\\
  \le& c|||U|||(q-s)^{\beta+\beta^\prime}(t-s)^{\beta^\prime}.
\end{align*}
Therefore, collecting the estimates for $I_1$, $I_2$ and $I_3$ we get \eqref{neu1}.

Moreover, the previous estimate for
$(u\otimes(\omega\otimes_S\omega))$ implies \eqref{neu2} similarly to \eqref{eq26}. To get  this relation we want to emphasize that it is necessary to use the Chen equality \eqref{chen3fold} and its generalized form \eqref{3tensorChen} which also holds if $u$ is $\beta$-H{\"o}lder continuous and $\omega\otimes_S\omega$ is $2\beta^\prime$-H{\"o}lder continuous, applying the fractional integration technique to \eqref{eq33}.
\end{proof}
The following estimates are crucial for the existence of a
solution.

\begin{lemma}\label{ex1}
Suppose that Hypothesis ${\bf H}$ holds. Then for any $1\geq \delta\,\geq \beta$ there exists $C>0$ such that for $u_0\in V_\delta$, $T\in [0,1]$ and $U\in W(0,T)$ we have
\begin{align}\label{ex4m}
\|{\mathcal T}_1(U)\|_{\beta}& \le CT^{\delta-\beta}|u_0|_{V_\delta}+CT^{\beta^\prime-\beta}(1+T^{2\beta}|||U|||^2),\\
|{\mathcal T}_1(U)(T)|_{V_\delta}& \le|u_0|_{V_\delta}+
CT^{\beta^\prime}(1+T^{2\beta}|||U|||^2).\nonumber
\end{align}

\end{lemma}

\begin{remark}\label{r13}
We want to stress that the second term in both of the above inequalities stems from the integral term of \eqref{equ1},
while the first term is an estimate of $\|S(\cdot)u_0\|_\beta$ and $|S(T)u_0|_{V_\delta}$, respectively.
\end{remark}

The proof of this result is rather technical for which we have
preferred to present it in the Appendix.

In the following result, we obtain the corresponding estimate for
${\mathcal T}_2(U)(s,t)$, which has been defined as the right-hand side
of equation \eqref{equ2}.

\begin{lemma}\label{ex2}
Suppose that Hypothesis ${\bf H}$ holds. Then for any $1\geq \delta\,\geq \beta$ there exists $C>0$ such that for $u_0\in V_\delta$, $T\in [0,1]$ and $U\in W(0,T)$ we have
 \begin{align*}
   &\|{\mathcal T}_2(U)\|_{\beta+\beta^\prime} \le C(T^{\beta^\prime-\beta}(1+T^{2\beta}|||U|||^2)+T^{\delta-\beta}|u_0|_{V_\delta}).
\end{align*}
\end{lemma}

\begin{proof}
Let us denote ${\mathcal T}_2(U)(s,t)=: B_1(s,t)+B_2(s,t)+B_3(s,t)$,
corresponding to the three different addends of \eqref{equ2}.

For $B_1$ we can consider the following splitting:
\begin{align*}
B_1(s,t)=&\int_s^t (S(\xi-s)-{\rm id})u(s)\otimes_V d\omega(\xi)\\
=&\int_s^t (S(\xi)-S(s))u_0 \otimes_V d\omega(\xi)\\
&+\int_s^t
(S(\xi-s)-{\rm id})\bigg(\int_0^s S(s-r)G(u(r))d\omega (r)
\bigg)\otimes_V d\omega(\xi)
\\=:&B_{11}(s,t)+B_{12}(s,t).
\end{align*}
$B_1$ can be interpreted in Weyl's sense thanks to the regularity of its integrand, which means that
\begin{align*}
B_{11}(s,t)& =(-1)^\alpha \int_s^t
D^\alpha_{s+}((S(\cdot)-S(s))u_0)[\xi]\otimes_VD^{1-\alpha}_{t-}\omega_{t-}[\xi]d\xi.
\end{align*}
For $\alpha<\mu<1+\delta$, $\delta\leq 1$ and $s>0$, applying \eqref{eq1} and
\eqref{eq2} we have
\begin{align*}
    |D^\alpha_{s+}((S(\cdot)-S(s))u_0)[\xi]|&\le \frac{1}{\Gamma(1-\alpha)}\bigg(\frac{|(S(\xi)-S(s))u_0|}{(\xi-s)^{\alpha}}+\alpha\int_s^\xi
    \frac{|(S(\xi)-S(q))u_0|}{(\xi-q)^{1+\alpha}}dq\bigg)\\
    &\le c\bigg((\xi-s)^{\delta-\alpha}
    +\int_s^\xi\frac{(\xi-q)^\mu}{(\xi-q)^{1+\alpha}(q-s)^{\mu-\delta}}dq\bigg)|u_0|_{V_\delta}\\
    &\le c(\xi-s)^{\delta-\alpha}|u_0|_{V_\delta}.
\end{align*}
From the last inequality and \eqref{q1}, for $s>0$, it follows that
\begin{align*}
|B_{11}(s,t)|&  \leq c \|\omega\|_{\beta^\prime}
|u_0|_{V_\delta} \int_s^t
(\xi-s)^{\delta-\alpha}(t-\xi)^{\beta^\prime+\alpha-1} d\xi,
\end{align*}
and Lemma \ref{l1-m} immediately implies
\begin{align*}
\|B_{11}\|_{\beta+\beta^\prime}  \leq C |u_0|_{V_{\delta}}T^{\delta-\beta}.
\end{align*}
Besides, note that
\begin{align*}
& D_{s+}^\alpha \bigg((S(\cdot-s)-{\rm id})\int_0^s
S(s-r)G(u(r))d\omega
(r)\bigg)[\xi] \\
&\leq C \bigg(  \frac{(S(\xi-s)-{\rm id}) (\int_0^s S(s-r)G(u(r))d\omega
(r)) }{(\xi-s)^\alpha} + \int_s^\xi \frac{\int_0^s
(S(\xi-r)-S(q-r))G(u(r))d\omega(r)}{(\xi-q)^{1+\alpha}}dq\bigg).
\end{align*}
For the second expression on the right-hand side, by Lemma
\ref{ex1} (ii)  for $\alpha<\mu<1$ and $\delta^\prime\le\delta$,  and Remark \ref{r13}, we have
\begin{align*}
\int_s^\xi& \frac{|\int_0^s (S(\xi-q)-{\rm id})S(q-r)G(u(r))d\omega(r)|}{(\xi-q)^{1+\alpha}}dq \\
&\leq  \int_s^\xi \frac{|(-A)^{\mu-\delta^\prime}S(q-s)|}{(\xi-q)^{1+\alpha-\mu}}\bigg|\int_0^s S(s-r)G(u(r))d\omega(r)\bigg|_{V_{\delta^\prime}}dq\\
&\leq c \int_s^\xi \frac{|(-A)^{\mu-\delta^\prime}S(q-s)|}{(\xi-q)^{1+\alpha-\mu}}\bigg|\int_0^s S(s-r)G(u(r))d\omega(r)\bigg|_{V_{\delta}}dq\\
& \leq C s^{\beta^\prime}  (1+s^{2\beta}|||U|||^2)
\int_s^\xi
\frac{(q-s)^{\delta^\prime-\mu}}{(\xi-q)^{1+\alpha-\mu}}dq \leq C s^{\beta^\prime} (1+s^{2\beta}|||U|||^2)
(\xi-s)^{\delta^\prime-\alpha},
\end{align*}
then in particular we have
\begin{align*}
\|B_{12}\|_{\beta+\beta^\prime}\leq
CT^{\beta^\prime +\delta^\prime-\beta} (1+T^{2\beta}|||U|||^2)\leq
CT^{\beta^\prime-\beta} (1+T^{2\beta}|||U|||^2).
\end{align*}
Finally, the same estimate follows for $B_2$ and $B_3$. In order to see this, note that $B_2$ and $B_3$ can be considered to be, respectively, the first and second integral on the right-hand side of \eqref{eq24}, and then, evaluating respectively the Hilbert-Schmidt operators $D_{t-}^{1-\alpha}(\omega\otimes_S\omega)(\cdot, t)_{t-}[r]$ and
$D_{t-}^{1-\alpha}
\dD_{t-}^{1-\alpha}(u\otimes(\omega\otimes_S\omega)(t))[r]$ we can obtain \begin{align*}
&\|B_{2}\|_{\beta+\beta^\prime}\leq
CT^{\beta^\prime-\beta} (1+T^{2\beta}|||U|||^2),\\
&\|B_{3}\|_{\beta+\beta^\prime}\leq
C T^{\beta^\prime-\beta} (T^\beta |||U|||+T^{2\beta}|||U|||^2)\leq
CT^{\beta^\prime-\beta} (1+T^{2\beta}|||U|||^2).
\end{align*}
\end{proof}

Now we establish a result related to the contraction property of $\mathcal T$:
\begin{lemma}\label{ex3}
Suppose that Hypothesis ${\bf H}$ holds. Then for any $1\geq \delta\,\geq \beta$ there exists $C>0$ such that for $T\in [0,1]$ and $U^1=(u^1,v^1),\,U^2=(u^2,v^2)\in W(0,T)$ with $u^1(0)=u_0^1,\,u^2(0)=u_0^2\in V_\delta$,
we have that
\begin{align*}
&|||{\mathcal T}(U^1)-{\mathcal T}(U^2)||| \le CT^{\beta^\prime-\beta}(1+T^{2\beta}(|||U^1|||^2+|||U^2|||^2))(|||U^1-U^2|||+|u_0^1-u_0^2|)+T^{\delta-\beta}|u_0^1-u_0^2|_{V_\delta}.
\end{align*}
In addition, for the first component $\mathcal T_1$ of the mapping we get
\begin{align*}
&|{\mathcal T}_1(U^1)(T)-{\mathcal T}_1(U^2)(T)|_{V_\delta} \le CT^{\beta^\prime}(1+T^{2\beta}(|||U^1|||^2+|||U^2|||^2))(|||U^1-U^2|||+|u_0^1-u_0^2|_{V_\delta})+c|u_0^1-u_0^2|_{V_\delta}.
\end{align*}

\end{lemma}

\begin{proof}
Trivially, $\|S(\cdot) u_0^1-S(\cdot) u_0^2\|_{\beta}\leq T^{\delta-\beta}|u_0^1-u_0^2|_{V_\delta}.$

We only give an idea of the proof, since this result follows in a similar manner as the previous Lemmata \ref{ex1} and \ref{ex2} by doing some changes in the integrals representing the equation \eqref{equ1}  as well as in the integrals related to the area equation \eqref{equ2}. In particular, in the fractional derivatives containing $G(u)$ we should replace it by $G(u^1)-G(u^2)$. Denoting $\Delta u=u^1-u^2$, on account of Lemma \ref{l2}, it holds
\begin{align*}
    \|G(u^1(r))-G(u^2(r))\|_{L_2(V,V_\delta)}&\le c_{DG}(|u^1(r)-u_0^1-u^2(r)+u_0^2|+|u_0^1-u_0^2|)\\
    &\le c_{DG}\bigg(\sup_{0\le q<r\le T}\frac{|u^1(r)-u^2(r)-(u^1(q)-u^2(q))|}{|r-q|^{\beta}}T^\beta+|u_0^1-u_0^2|\bigg)\\&=c_{DG}(\|\Delta u\|_\beta T^\beta+|u_0^1-u_0^2|)
\end{align*}
and similar for $DG(u(r))$. Moreover, using the first inequality of Lemma \ref{l3} we get
\begin{align*}
    \|DG(u^1(r))-DG(u^2(r))&-(DG(u^1(q))-DG(u^2(q)))\|_{L_2(V,V_\delta)}\\
    &\le c_{D^2G}\|\Delta u\|_{\beta}|r-q|^\beta+c_{D^3G}(\|\Delta u\|_{\beta}
   T^\beta+|u_0^1-u_0^2|) (\|u^1\|_\beta+\|u^2\|_{\beta}) |r-q|^\beta
    \end{align*}
and, using the last inequality of Lemma \ref{l3},
\begin{align*}
    \|G(u^1(r))&-G(u^1(q))-DG(u^1(q))(u^1(r)-u^1(q))
    \\&-(G(u^2(r))-G(u^2(q))-DG(u^2(q))(u^2(r)-u^2(q)))\|_{L_2(V,V_\delta)}\\
    &\le c_{D^2G}
    (\|u^1\|_\beta+\|u^2\|_\beta)\|\Delta u\|_\beta|r-q|^{2\beta}\\
    &+c_{D^3G}\|u^2\|_\beta|r-q|^\beta(\|\Delta u\|_\beta T^\beta+|u_0^1-u_0^2|)(2\|u^1\|_\beta+\|u^2\|_\beta)|r-q|^\beta.
\end{align*}
We also note that  the expressions $(u^i\otimes(\omega\otimes_S\omega))$ are linear with respect to $U^i$. Hence by Lemma
\ref{l7} we obtain
\begin{align*}
&\|D_{t-}^{1-\alpha}
\dD_{t-}^{1-\alpha} ((u^1\otimes(\omega\otimes_S\omega)(t))-(u^2\otimes(\omega\otimes_S\omega)(t)))[r]\|_{L_2(L_2(V\otimes V,V_\delta),V\otimes V)}  \leq c |||U^1-U^2||| (t-r)^{\beta+2\beta^\prime+2\alpha-2}.
\end{align*}

To obtain the second statement of this result we would need to use the previous estimates and follow similar steps than in the proof of Lemma \ref{ex1} (ii), see Appendix Section.
\end{proof}

\begin{remark}\label{ic}
We want to stress that in the previous result we have compared $\mathcal T(U^1)$ with $\mathcal T(U^2)$ by using the $|||\cdot|||$-seminorm. As we already mentioned, we could add the $V$-norm of the initial condition to the $|||\cdot|||$-seminorm. However, in practice this is not necessary, except when the previous result is used with $U^1$ and $U^2$ having different initial conditions, which will happens only in Theorem \ref{rem10} and Lemma \ref{l81} below. In the latter results in fact we have a sequence of initial conditions which converges and therefore it suffices to consider the seminorm.
\end{remark}

In Theorem \ref{rem10} and the Appendix section we will need to apply the previous lemma when having $(u,v)$ driven by $\omega$, and $(u,u\otimes \omega^n)$ driven by $\omega^n$. This is the reason to explain next what happens in this particular situation, which is not included above since the driving noises are different.

In the following results we indicate the dependence of $\tT$ on
$u_0\in V_\delta,\omega$ and $(\omega\otimes_S\omega)$ by $\tT(U,\omega,(\omega\otimes_S\omega),u_0)$.
\begin{lemma}\label{ex3bis}
Suppose Hypothesis ${\bf H}$ holds.
Then we have for any $K>0$  that
\begin{align*}
&\lim_{n\to\infty}\sup\{|||\tT(U,\omega,(\omega\otimes_S\omega),u_0)-\tT(U,\omega^n,(\omega^n\otimes_S\omega^n),u_0)|||:|||U|||\le K,\,|u_0|_{V_\delta}\le K\}=0.
\end{align*}
\end{lemma}

\begin{proof}
We only sketch the proof.
For the first integral on the right-hand side of \eqref{equ1} we obtain the estimate $2c(1+K^2)\|\omega^n-\omega\|_{\beta^\prime}$
which tends to zero for $n\to\infty$ by our Hypothesis {\bf H}. Note that the second integral of \eqref{equ1} is just zero.

Regarding \eqref{equ2}, the estimate of the second integral is straightforward thanks to Hypothesis {\bf H}, item (4).
Consider finally the first and the third integral of \eqref{equ2} which can be rewritten as
\begin{align*}
    &\int_s^t(S(\xi-s)-{\rm id})u(s)\otimes_Vd(\omega(\xi)-\omega^n(\xi)),
\end{align*}
and
\begin{align*}
   (-1)^{2\alpha-1}\int_s^t D_{s+}^{2\alpha-1} DG(u(\cdot))[r] D_{t-}^{1-\alpha}\dD_{t-}^{1-\alpha} ((u\otimes (\omega \otimes_S \omega)(t))- (u\otimes (\omega^n \otimes_S \omega^n)(t))) [r] dr,
\end{align*}
respectively. Recall that $(u\otimes(\omega\otimes_S\omega))$ is a function depending on $U=(u,v),\omega$ and $(\omega\otimes_S\omega)$. Let us pick one of the terms we have to estimate, for instance, the first term appearing in the expression of $\omega_S(t)\otimes\omega$ (see (\ref{xxxeqbis})):
\begin{align*}
    &\|(\omega^n)_S(\tau,t)\int_s^\tau S(\tau-r) Ed\omega^n-\omega_S(\tau,t)\int_s^\tau S(\tau-r)Ed\omega\| \\
    \le & c(t-\tau)^{\beta^\prime}(\tau-s)^{\beta^\prime}\|E\|_{L_2(V,V_\delta)}(\|\omega-\omega^n\|_{\beta^\prime}(\|\omega\|_{\beta^{\prime\prime}}+\|\omega\|_{\beta^\prime})+\|\omega^n\|_{\beta^\prime}(\|\omega-\omega^n\|_{\beta^{\prime\prime}}+\|\omega-\omega^n\|_{\beta^\prime})),
\end{align*}
which converges to zero. In a similar manner we can estimate the other terms such that we have
\begin{equation*}
\|(u\otimes (\omega \otimes_S \omega)(t))(s,q)-(u\otimes (\omega^n \otimes_S \omega^n)(t))(s,q)\|_{L_2(L_2(V\otimes V,V_\delta),V\otimes V)}\le C_n|||U||| (q-s)^{\beta+\beta^\prime}(t-s)^{\beta^\prime}.
\end{equation*}
The constant $C_n$ depends on $\|\omega^n-\omega\|_{\beta^\prime}$ and $\|(\omega^n\otimes_S\omega^n)-(\omega\otimes_S \omega)\|_{2\beta^\prime}$ and other terms which ensure that $(C_n)_{n\in\NN}$
 converges to zero for $n\to\infty$.

\end{proof}

\begin{theorem}\label{tex1}
Assume $U^1,\,U^2$ are two solutions of the system \eqref{equ1}-\eqref{equ2} in $W(0,1)$ with initial condition $u_0\in V_\delta$ for $1\geq \delta \geq \beta$. Then $U^1=U^2$.
\end{theorem}

\begin{proof}
Suppose  $U^1\not=U^2\in W(0,1)$. Then there is a $\rho$ such that $|||U^1-U^2|||=\rho>0$.
By Lemma \ref{l81} (see Appendix) and Hypothesis {\bf H}, item (4), we can approximate these solutions by sequences $(U_n^1)_{n\in\NN},\,(U_n^2)_{n\in\NN}$
having the initial condition $u_0$, being $U_n^i=(u_n^i,v_n^i)$ where $u_n^i$ is given by \eqref{eq0} and $v_n^i$ has the interpretation of \eqref{ns}
driven by smooth $\omega^n$ and $(\omega^n\otimes_S\omega^n)$, such that for sufficiently large $n$ we have that $|||U^i-U_n^i|||<\rho/2$ for $i=1,2$.
However, we have $u_n^1=u_n^2$ and hence $U_n^1=U_n^2$ which contradicts $|||U^1-U^2|||=\rho>0$.
\end{proof}

For $0<a<b$ we consider the {\em concatenation} of elements of $W(0,a)$ and $W(a,b)$. We have to take into account that the elements of these function spaces consists of a path component and an area component. We have to define the concatenation of the area component in agreement with the Chen equality.
Let $U^1=(u^1,v^1)\in W(0,a)$ such that $u^1(0)\in V_\delta$, for $\delta\in [0,1]$ and $U^2=(u^2,v^2)\in  W(a, b)$ such that $u^2(a)=u^1(a)$, for $0\le a<b\le 1$. Define $U=(u,v)$ as follows:
\begin{align*}
    &u(t)=\left\{\begin{array}{lcl}
    u^1(t)&:& 0\le t \le a\\
    u^2(t)&:& a\le t\le b
    \end{array}\right.\\
    &v(s,t)=\left\{\begin{array}{lcr}
    v^1(s,t)&:& 0\le s\le t \le a\\
    v^2(s,t)&:& a\le s\le t\le b\\
    (u^1(a)-u^1(s))\otimes_V(\omega(t)-\omega(a))+v^1(s,a)+v^2(a,t)&:& s\le a<t.
    \end{array}\right.
\end{align*}

\begin{remark}\label{r10}
Suppose that $U^1,\,U^2$ satisfy the above conditions and that $\omega\in C_{\beta^\prime}([0,b];V)$.
Then the concatenation of  $\,U^1$ and  $\,U^2$ is in $W(0,b)$.
\end{remark}
Indeed by the definition it follows that this concatenation is H{\"o}lder continuous with respect to the H{\"o}lder exponents
used in this article.
To see that the Chen equality holds for say $0\le s<r<a<t\le b$ we note that
\begin{align*}
    v(s,r)&+v(r,t)+(u(r)-u(s))\otimes (\omega(t)-\omega(r))\\
    &=v^1(s,r)+v^1(r,a)+v^2(a,t)+(u^1(a)-u^1(r))\otimes (\omega(t)-\omega(a))+(u^1(r)-u^1(s))\otimes (\omega(t)-\omega(r))\\
    &=v^1(s,a)+v^2(a,t)+(u^1(a)-u^1(r))\otimes (\omega(t)-\omega(a))+(u^1(r)-u^1(s))\otimes (\omega(t)-\omega(a))=v(s,t)
\end{align*}
by the definition of the concatenation.

\begin{theorem}\label{tex2}
Suppose Hypothesis ${\bf H}$ holds, and that $\delta+\beta^\prime>1$, $\delta\leq 1,$
and $u_0\in V_\delta$. Then there exists a unique solution to system \eqref{equ1}-\eqref{equ2} in $W(0,1)$.
\end{theorem}
\begin{proof}
We start presenting a few trivial inequalities. Let $C$ be the common constant such that Lemma \ref{ex1}, \ref{ex2} and \ref{ex3} hold (note that the condition $\delta+\beta^\prime>1$ implies that $\delta \geq \beta$). For the following, we have that for any $\rho_0>0$ there is a $K(\rho_0)\geq 1$ such that for $K\ge K(\rho_0)$, $i\in\mathbb{N}$
\begin{equation}\label{ex7}
\rho_0+\sum_{j=1}^{i}2C(Kj)^{-\beta^\prime}= \rho_0+2CK^{-\beta^\prime}\sum_{j=1}^ij^{-\beta^\prime}\le \rho_0+2CK^{-\beta^\prime}\frac{1}{1-\beta^\prime} i^{1-\beta^\prime}< (Ki)^{1-\beta^\prime},
\end{equation}
and
\begin{align}\begin{split}\label{ex8}
    &4C^2 (Ki)^{-\beta^\prime-\beta}((Ki)^{\beta-\delta}(Ki)^{1-\beta^\prime}+(Ki)^{\beta-\beta^\prime})\le C^\prime (Ki)^{1-2\beta^\prime-\delta}< 1,
    \\
    &C(Ki)^{\beta-\beta^\prime}(1+2(Ki)^{-2\beta}(8C^2(Ki)^{2\beta-2\delta}(Ki)^{2-2\beta^\prime}+8C^2(Ki)^{2\beta-2\beta^\prime}))\\
    &\le C(Ki)^{\beta-\beta^\prime}+C^\prime((Ki)^{\beta-3\beta^\prime-2\delta+2}+(Ki)^{\beta-3\beta^\prime})<\frac12,\\
    &C(Ki)^{-\beta^\prime}+C(Ki)^{-\beta^\prime-2\beta}
    (8C^2(Ki)^{2\beta-2\delta}(Ki)^{2-2\beta^\prime}+8C^2(Ki)^{2\beta-2\beta^\prime})< 2C(Ki)^{-\beta^\prime},
    \end{split}
\end{align}
where $C^\prime$ is an appropriate constant independent of $i$. Note that from \eqref{ex7} we also have that  $\rho_0< K^{1-\beta^\prime}$. Define $|u_0|_{V_\delta}=:\rho_0,\,\Delta T_1=K^{-1}\le 1,\,T_1=T_0+\Delta T_1$ where $T_0=0$. Then,  by Lemma \ref{ex1} (i)  and Lemma \ref{ex2}, we have that
\begin{equation*}
    |||{\mathcal T}(U)|||\le C(\Delta T_1^{\delta-\beta}\rho_0+\Delta T_1^{\beta^\prime-\beta}+\Delta T_1^{\beta^\prime+\beta}|||U|||^2).
\end{equation*}
Hence, to find a ball $B_{W(T_0,T_1)}(0,R_1)$ that will be mapped into itself we calculate the minor root $R_1$ of

\begin{equation}\label{2eq}
    x=C(\Delta T_1^{\delta-\beta}\rho_0+\Delta T_1^{\beta^\prime-\beta}+\Delta T_1^{\beta^\prime+\beta}x^2)
\end{equation}
which is given by
\begin{equation*}
   R_1:= \frac{2C(\Delta T_1^{\delta-\beta}\rho_0+\Delta T_1^{\beta^\prime-\beta})}{1+\sqrt{1-4C^2 \Delta T_1^{\beta^\prime+\beta}(\Delta T_1^{\delta-\beta}\rho_0+\Delta T_1^{\beta^\prime-\beta})}}< 2C(\Delta T_1^{\delta-\beta}\rho_0+\Delta T_1^{\beta^\prime-\beta}),
\end{equation*}
see Sohr \cite{Sohr} Page 349.
This root is well-defined which follows from \eqref{ex7} and the first inequality of \eqref{ex8} for $i=1$, since these conditions in particular imply that

\begin{equation}\label{2eq1}
    1-4C^2 \Delta T_1^{\beta^\prime+\beta}(\Delta T_1^{\delta-\beta}\rho_0+\Delta T_1^{\beta^\prime-\beta})> 0.
\end{equation}
Moreover, we obtain from Lemma \ref{ex3} with $u_0^1=u_0^2$ that $\mathcal{T}$ is a contraction on the ball $B_{W(T_0,T_1)}(0,R_1)$ if

\begin{equation*}
    C\Delta T_1^{\beta^\prime-\beta}(1+2\Delta T_1^{2\beta}(8C^2\Delta T_1^{2\delta -2\beta}\rho_0^2+8C^2\Delta T_1^{2\beta^\prime -2\beta}))<\frac12
\end{equation*}
which follows from \eqref{ex7} and the second inequality of \eqref{ex8} for $i=1$. Then the system \eqref{equ1}-\eqref{equ2} has a  solution $U^1$ in $B_{W(T_0,T_1)}(0,R_1)$
which is unique by Theorem \ref{tex1}.

Furthermore, by Lemma \ref{ex1} (ii)  it is known that
\begin{align*}
    |u(T_1)|_{V_\delta}&\le \rho_0+C(\Delta T_1^{\beta^\prime}(1+\Delta T_1^{2\beta}|||U|||^2))\\
    &\leq \rho_0+C\Delta T_1^{\beta^\prime}+C\Delta T_1^{\beta^\prime+2\beta}R_1^2\\
    &\leq \rho_0+C\Delta T_1^{\beta^\prime}+C\Delta T_1^{\beta^\prime+2\beta}(8C^2\Delta T_1^{2\delta-2\beta} \rho_0^2+8C^2\Delta T_1^{2\beta^\prime-2\beta})\\
    &< \rho_0+2C \Delta T_1^{\beta^\prime}\le \rho_0+ 2 C K^{-\beta^\prime}.
    \end{align*}
Hence, by using again \eqref{ex7}, the right hand side of the previous inequality is bounded by $K^{1-\beta^\prime}$.

Suppose now that we have concatenated a solution on $[0,T_{i-1}]$ and that $|u(T_{i-1})|_{V_\delta}< \rho_0+\sum_{j=1}^{i-1} 2C (Kj)^{-\beta^\prime}$ for $i=2, 3, \cdots$, and $T_{i-1}<1$. For the fact that this concatenation is a solution we refer to Theorem \ref{rem10} below. Set $T_i=T_{i-1}+\Delta T_{i}$, $\Delta T_{i}=(Ki)^{-1}$ if $T_i< 1$, and $T_i=1$ in other case.
By \eqref{ex7} we know that $|u(T_{i-1})|_{V_\delta}<{ (K(i-1))}^{1-\beta^\prime}$.
Because of \eqref{ex8}, the Banach fixed point theorem gives us a  solution to the system \eqref{equ1}-\eqref{equ2} $U^i\in B_{ W(T_{i-1},T_i)}(0,R_{i})$ which is unique, where $R_{i}$ is the minor root of \eqref{2eq} when replacing $\rho_0$ by $(K(i-1))^{1-\beta^\prime}<(Ki)^{1-\beta^\prime}$ and $\Delta T_1$ by $\Delta T_{i}$.  Again, by concatenation
we obtain a solution on $[0,T_i]$. In addition, we obtain
\begin{equation*}
|u(T_{i})|_{V_\delta} \le \rho_0 +\sum_{j=1}^{i-1} 2C (Kj)^{-\beta^\prime}+ 2C (Ki)^{-\beta^\prime}=\rho_0 +\sum_{j=1}^{i} 2C (Kj)^{-\beta^\prime}< (Ki)^{1-\beta^\prime}.
\end{equation*}
Finally,
since $\sum_i i^{-1}=\infty$ there is an $i^\ast\in\NN$ such that $T_{i^\ast}\wedge 1= 1$, which means that there exists a global solution of \eqref{equ1}-\eqref{equ2} in $W(0,1)$ for any $u_0\in V_\delta$.
\end{proof}

Now we prove the assertion form the last theorem allowing us to concatenate solutions to another solution.

\begin{theorem}\label{rem10}
Suppose Hypothesis ${\bf H}$ holds, and assume $\delta+\beta^\prime>1$,
and $u_0\in V_\delta$. Let $U^i$ be the elements from the proof of Theorem \ref{tex2}. Then these elements can be concatenated to a solution $U$ in $W(0,1)$. In particular,
this solution satisfies the Chen equality.
\end{theorem}

\begin{proof}

Recall that $\tT=\tT(U,\omega,(\omega\otimes_S\omega),u_0)$ denotes the right-hand side of \eqref{equ1}-\eqref{equ2}. In addition, suppose that $(u_0^n)$ converges to $u_0$ in $V_\delta$. Let $B_{W(T_0,T_1)}(0,R_1)$ be the ball from the proof of Theorem \ref{tex2} such that
$\tT(U,\omega,(\omega\otimes_S\omega),u_0)$ is a self--mapping and a contraction on this ball with a contraction constant less than 1/2. We can choose an $R_1^\prime>R_1$ that $\tT(U,\omega,(\omega\otimes_S\omega),u_0)$ is still a contraction with constant less than 1/2 with respect to $B_{W(T_0,T_1)}(0,R_1^\prime)$. Since the  constant $C$ in Lemmata \ref{ex1}, \ref{ex2} and \ref{ex3} can be chosen continuously depending on $u_0,\,\omega$ and $(\omega\otimes_S\omega)$, for sufficiently large $n$ the mappings
$\tT(\cdot,\omega^n,(\omega^n\otimes_S\omega^n),u_0^n)$ map $B_{W(T_0,T_1)}(0,R_1^n),\,R_1^n\le R_1^\prime$ into itself and have a contraction constant less than 1/2 on these balls. Let $U^1=(u^1,v^1)$ and $U^1_n=(u^1_n,v^1_n)$ be the fixed points of
$\tT(\cdot,\omega,(\omega\otimes_S\omega),u_0)$, $\tT(\cdot,\omega^n,(\omega^n\otimes_S\omega^n),u_0^n)$. Then
\begin{align*}
    |||U^1-U^1_n|||&\le|||\tT(U^1,\omega,(\omega\otimes_S\omega),u_0)-\tT(U^1,\omega^n,(\omega^n\otimes_S\omega^n),u_0)|||\\
  &+  |||\tT(U^1,\omega^n,(\omega^n\otimes_S\omega^n),u_0)-\tT(U^1,\omega^n,(\omega^n\otimes_S\omega^n),u_0^n)|||\\
    &+|||\tT(U^1,\omega^n,(\omega^n\otimes_S\omega^n),u_0^n)-\tT(U^1_n,\omega^n,(\omega^n\otimes_S\omega^n),u_0^n)|||\\
    &\le |||\tT(U^1,\omega,(\omega\otimes_S\omega),u_0)-\tT(U^1,\omega^n,(\omega^n\otimes_S\omega^n),u_0)|||\\
    &+T_1^{\delta-\beta}|u_0-u_0^n|_{V_\delta}
    +\frac12 |||U^1-U^1_n|||.
\end{align*}
It suffices to take into account that, by Lemma \ref{ex3bis}, the first term on the right-hand side converges to zero; furthermore, the second one trivially goes to zero due to the convergence of $(u_0^n)$ to $u_0$ in $V_\delta$, hence $U^1_n$ converges to $U^1$.

Applying the  second part of Lemma \ref{ex3},
\begin{align}\label{eq30}
\begin{split}
 |u^1_n(T_1)-u^1(T_1)|_{V_\delta}&\le |\tT_1(U^1,\omega,(\omega\otimes_S\omega),u_0)-\tT_1(U^1_n,\omega,(\omega \otimes_S\omega),u_0^n)|_{V_\delta}\\
 &+|\tT_1(U^1_n,\omega,(\omega\otimes_S\omega),u_0^n)-\tT_1(U^1_n,\omega^n,(\omega^n\otimes_S\omega^n),u_0^n)|_{V_\delta}\\
 &\leq c| u_0^n-u_0|_{V_\delta} + CT_1^{\beta^\prime} (1+ T_1^{2\beta^\prime} (|||U^1_n|||^2+|||U^1|||^2)) (|||U^1_n-U^1|||+| u_0^n-u_0|_{V_\delta})\\
 &+|\tT_1(U^1_n,\omega,(\omega\otimes_S\omega),u_0^n)-\tT_1(U^1_n,\omega^n,(\omega^n\otimes_S\omega^n),u_0^n)|_{V_\delta}.
\end{split}
\end{align}
Since $U^1_n$ converges to $U^1$, on account of item (4) in {\bf H}, this inequality implies that $u^1_n(T_1)$ converges to $u^1(T_1)$ in $V_\delta$. Indeed, for the convergence of the third expression on the right-hand side we note that $\{|||U^1_n|||:n\in\NN\}$ is bounded and
that the first integral of \eqref{equ1} converges when we replace $\omega$ by  $\omega-\omega^n$ by Hypothesis {\bf H}, while the second integral is zero because it does not depend on $\omega$ or $\omega\otimes_S\omega$. In addition, since by Lemma \ref{l8} $U^1_n$ satisfies the Chen equality so does $U^1$.

Therefore we can repeat the same calculations on $[T_1,T_2]$ and similarly on any of the {\em finitely} many intervals $[T_i,T_{i+1}]$. Since $u^i_n$ are related to {\em classical} mild solutions to \eqref{eqn1} we can concatenate these elements to a $\beta$-H{\"o}lder continuous solution on $[0,1]$ and the associated elements $U^i_n$ can be concatenated to an element $U_n$ in $W(0,1)$, see Remark \ref{r10}. These concatenations converge to $U\in W(0,1)$ given by the concatenation of $U^i$. Since $U^n$ satisfies \eqref{equ1}-\eqref{equ2} so does its limit $U$ which gives then the solution of\eqref{equ1}-\eqref{equ2}  for non-regular $\omega$.

To sum up, in order to obtain the convergence of $(U_n)_{n\in\NN}$ in $W(0,1)$ we have used the Chen equality, the convergence of $U^i_n$ to $U^i$ on $[T_i,T_{i+1}]$ and the convergence of $\omega^n$ to $\omega$. In addition, we need the first part of Lemma \ref{ex3} and Lemma \ref{ex3bis}.
\end{proof}

\section{Fractional Brownian motion with Hurst index in the interval $(1/3,1/2]$. Examples}\label{fBm}

In this section we consider a fractional Brownian motion with Hurst parameter $H\in (1/3,1/2]$. We aim at checking that such a process satisfies item (4) in Hypothesis {\bf H}.  At the end of this section we also present two examples of possible non-linear operator $G$ satisfying all assumptions described in Hypothesis {\bf  H}.

Given a probability space and $H\in (0,1)$, a continuous centered Gau{\ss}ian process
$\beta^H(t)$, $t\in\mathbb{R}$, with the covariance function
\begin{equation*}
    \mathbb{E}\beta^H(t)\beta^H(s)=\frac{1}{2}(|t|^{2H}+|s|^{2H}-|t-s|^{2H}),\qquad t,\,s\in\mathbb{R}
\end{equation*}
is called a {\it two--sided one-dimensional fractional Brownian
motion (fBm)}, and $H$ is the {\it Hurst parameter}.

Let $Q$ be a positive symmetric operator of trace class on $V$, i.e.,  ${\rm tr}_V\,Q<\infty$, with positive discrete spectrum  $(q_i)_{i\in\NN}$ and eigenelements $(f_i)_{i\in \NN}$. Then a continuous {\it $V$-valued fractional Brownian motion $B^H$} with  covariance operator $Q$ and Hurst parameter $H$ is defined by
\begin{equation*}
   B^H(t)=\sum_{i=1}^{\infty} \sqrt{q_i}f_i \beta_i^H(t),\quad t\in\mathbb{R},
\end{equation*}
where $(\beta_i^H(t))_{i\in{\mathbb N}}$ is a sequence of stochastically independent one-dimensional fBm.

Kolmogorov's theorem ensures that $B^H$ has a continuous version. Thus we can consider the canonical interpretation of an fBm given by  $(C_0(\RR;V),\bB(C_0(\RR;V)),\PP)$, where $C_0(\mathbb{R},V)$ denotes the space of continuous functions on $\mathbb{R}$ with values in $V$ such that are zero at zero, equipped with the compact open topology, $\bB(C_0(\RR,V))$ is the associated Borel-$\sigma$-algebra and ${\mathbb P}$ is the distribution of the fBm $B^H$.

This (canonical) process has for any $\beta^{\prime}<H$ a version, denoted by $\omega$, which is ${\beta^{\prime}}$-H{\"o}lder continuous, see Kunita \cite{Kunita90}, Theorem 1.4.1. Since we want to pick the fractional Brownian motion as the driving path in the abstract theory developed in the previous sections, we restrict ourselves to the cases in which $H\in (1/3,1/2]$ and therefore $\beta^\prime \in (1/3,1/2)$.

Note that $\omega_i(t)=q_i^{-\frac12}(\omega(t),f_i)_V$ is an  iid-sequence of fractional Brownian motions in $\RR$. For the sake of brevity, in what follows we assume that the $f_i$ are the same as the $e_i$.

Let us also denote by $\omega^n$ and $\omega_i^n$ the piecewise linearizations of $\omega$ and $\omega_i$, respectively, with respect to the equidistant partition  $\{t_i^n\}_{i=0,\cdots,n}$ of $[0,1]$.  In this context, for $E\in L_2(V,V_\delta)$ we consider
\begin{align}\label{neu53}
\begin{split}
    E(\omega\otimes_S\omega)(s,t)=\int_s^t\int_s^\xi  S(\xi-r)Ed\omega(r)\otimes_V d\omega(\xi).
    \end{split}
\end{align}
Note that the inner integral is defined according to Lemma \ref{lex7} (ii). In what follows we shall prove that also the outer integral exists. To reach this purpose, notice that by partial integration we get
\begin{align*}
    E(\omega^n\otimes_S\omega^n)(s,t)&=\int_s^t\int_s^\xi S(\xi-r)Ed\omega^n(r)\otimes_V d\omega^n(\xi)\\
    &=\int_s^tE(\omega^n(\xi)-\omega^n (s))\otimes_V  d\omega^n(\xi)+\int_s^tA\int_s^\xi S(\xi-r)E(\omega^n(r)-\omega^n(s))dr\otimes_V d\omega^n(\xi).
\end{align*}

In the next results we will prove that the right-hand side of the above expression forms a Cauchy sequence having a limit, which will be given by \eqref{neu53}.

First of all, we focus on the first term on the right-hand side of the previous expression. To simplify a bit the exposition, in the following we denote the mapping
\begin{equation*}
    L_2(V, V_\delta) \ni E\to\int_s^tE(\omega^n(\xi)-\omega^n (s))\otimes_V  d\omega^n(\xi)\quad \text{by }\; E (\omega^n \otimes \omega^n ) (s,t).
\end{equation*}

Note that if we consider $\omega^n_i$, since this one-dimensional random variable is smooth, the couple $(\omega^n_i, \omega^n _ i \otimes \omega^n_i )$ is well-defined and converges to the couple $(\omega_i , \omega_i \otimes \omega_i) $ in $C_{\beta^\prime} ([0,1]; \RR) \times C_{2\beta^\prime} (\Delta_{0,1}; \RR^2)$, see Unterberger \cite{U} and Deya {\it et al.} \cite{DeNeTi10}.
\begin{theorem}\label{tex3}
The sequence $((\omega^n\otimes\omega^n))_{n\in\NN}$,
which elements can be represented by
\begin{equation*}
q_k^\frac12 q_j^\frac12\int_s^t(\omega_k^n(\xi)-\omega_k^n(s)) d\omega_j^n(\xi)
\end{equation*}
converges on a set of full measure  in $C_{2\beta^\prime} (\Delta_{0,1};L_2(L_2(V,V_\delta),V\otimes V))$ to an element denoted by $\omega\otimes\omega$.
\end{theorem}

\begin{proof}
Consider the orthonormal basis $(E_{ki})_{k,i\in\NN}$ of $L_2(V,V_\delta)$ given  by (\ref{etiqueta}), and $(e_l\otimes_V e_j)_{l,j\in\NN}$, the orthonormal basis of $V\otimes V$.
First note that, for $(s,t)\in \Delta_{0,1}$,
\begin{align*}
\|(\omega^n \otimes \omega^n)(s,t)\|^2_{L_2(L_2(V,V_\delta),V\otimes V)}&= \sum_{i,k}\sum_{l,j}  (E_{ki}( \omega^n \otimes \omega^n )(s,t), e_l\otimes_V e_j )_{V\otimes V}^2 \\
& \leq
  \sum_k\lambda_k^{-2\delta}\sum_{i,j}q_iq_j \bigg(\int_s^t(\omega_i^n(\xi)-\omega_i^n(s))d\omega_j^n(\xi)\bigg)^2,
\end{align*}
and hence we have to study the behavior of the last sum. Let us denote
\begin{align*}
 A_{i,j}^n(s,t):&=   \int_s^t(\omega_i(\xi)-\omega_i(s))d\omega_j(\xi)-\int_s^t(\omega_i^n(\xi)-\omega_i^n(s))d\omega_j^n(\xi).
\end{align*}
By symmetry we can assume $i\leq j$. In fact we assume that $i<j$ since the case $i=j$ is easier, see a comment at the end of the proof. Note that even though the tensors are defined in the set $\{(s,t) \in [0,1]^2: 0 \leq s \leq t \leq 1 \}$, since the L\'evy area $A_{i,j}^n$ is symmetric they can be extended to the set  $[0,1]^2$.

To estimate $A_{i,j}^n(s,t)$ we apply the Lemma 3.7 in \cite{DeNeTi10}, which claims that for $p\geq 1$ there exists $K_{\beta^\prime, p}$ such that
\begin{align}\label{tail}
\|A_{i,j}^n\|_{2\beta^\prime} \leq K_{\beta^\prime, p} (R_{n,p}^{i,j}+ \|\omega_i-\omega_i^n\|_{\beta^\prime} \|\omega_j\|_{\beta^\prime}+\|\omega_j-\omega_j^n\|_{\beta^\prime} \|\omega_i^n\|_{\beta^\prime}),
\end{align}
where
\begin{align*}
R_{n,p}^{i,j}:=\bigg(\int_0^1\int_0^1\frac{|A_{i,j}^n(s,t)|^{2p}}{|t-s|^{4\beta^\prime p+2}}dsdt\bigg)^{1/(2p)}.
\end{align*}
In particular, from the
proof of Lemma 3.7 in \cite{DeNeTi10} we know that
\begin{equation*}
   \EE (R_{n,p}^{i,j})^{2p}\le c n^{-4p(H-\beta^{\prime\prime})}<\infty,
\end{equation*}
for $\beta^\prime<\beta^{\prime\prime}<H$, being $\beta^{\prime \prime}$ close enough to $H$. Now let us take $p$ large enough such that $4p(H-\beta^{\prime\prime})>1$, and thus
\begin{align*}
    \PP(\sum_{i,j}q_i q_j(R_{p,n}^{i,j})^2>o_n^2)\le ({\rm tr}_V Q)^{2(p-1)}\sum_{i,j}q_i q_j \EE(R_{n,p}^{i,j})^{2p}o_n^{-2p}
    \le C n^{-4p(H-\beta^{\prime\prime})}o_n^{-2p}.
\end{align*}
For an appropriate sequence $(o_n)_{n\in\NN}$ with limit zero, the right-hand side has a finite sum. Then by the Borel-Cantelli Lemma, $(\sum_{i,j}q_iq_j(R_{p,n}^{i,j})^2)_{n\in\NN}$ tends to zero almost surely. In a similar manner we obtain the convergence of the last terms in \eqref{tail}. It suffices to take into account that, for $\beta^\prime< \beta^{\prime\prime}<H$,
\begin{align}\label{eq20}
\begin{split}
    \|\omega_i-\omega_i^n\|_{\beta^\prime}\le  G_{\beta^{\prime\prime}}(i,\omega)n^{\beta^\prime-\beta^{\prime\prime}},\;
    \|\omega_i\|_{\beta^\prime}\le G_{\beta^{\prime\prime}}(i,\omega),\;
    \|\omega_i^n\|_{\beta^\prime}\le G_{\beta^{\prime\prime}}(i,\omega)
    \end{split}
    \end{align}
where $G_{\beta^{\prime\prime}}(i,\omega)\ge   \|\omega_i\|_{\beta^{\prime\prime}}$ and  $G_{\beta^{\prime\prime}}(i,\omega) \in L^{p}(\Omega)$  for   $p\in\NN$ are iid random variables, see Kunita \cite{Kunita90} Theorem  1.4.1. We then have
\begin{align*}
    \PP(\sum_{i,j}q_i q_j\|\omega_i^n-\omega_i\|_{\beta^\prime}^2\|\omega_j\|_{\beta^\prime}^2>o_n^2)&\le ({\rm tr}_VQ)^{2(p-1)}\sum_{i,j}q_i q_j (\EE G_{\beta^{\prime\prime}}(i,\omega)^{4p})^\frac12(\EE G_{\beta^{\prime\prime}}(j,\omega)^{4p})^\frac12n^{2p(\beta^\prime-\beta^{\prime\prime})}o_n^{-2p}\\
    &\le C n^{2p(\beta^{\prime}-\beta^{\prime\prime})}o_n^{-2p}.
\end{align*}
For $p$ chosen sufficiently large and an appropriate zero--sequence $(o_n)_{n\in\NN}$ we obtain the almost sure convergence of
$(\sum_{i,j}q_i q_j\|\omega_i^n-\omega_i\|_{\beta^\prime}^2 \|\omega_j\|_{\beta^\prime}^2)_{n\in \NN}$. Similarly we can treat the last term  of \eqref{tail}, that is,
$(\sum_{i,j}q_i q_j\|\omega_i^n-\omega_i\|_{\beta^\prime}^2 \|\omega_j^n\|_{\beta^\prime}^2)_{n\in \NN}$.
Finally, note that
\begin{align*}
 A_{i,i}^n(s,t)&=   \frac{1}{2}(\omega_i(t)-\omega_i(s))^2- \frac{1}{2}(\omega_i^n(t)-\omega_i^n(s))^2
 \end{align*}
 and thanks to \eqref{eq20} we can obtain $\| A_{i,i}^n\|_{2\beta^\prime}\leq G_{\beta^{\prime\prime}}(i,\omega)^2n^{2(\beta^\prime-\beta^{\prime\prime})}$, which completes the proof.
\end{proof}
Note that the limit operator $\omega\otimes\omega$ is represented by L\'evy-areas $\int_s^t(\omega_k(\xi)-\omega_k(s))d\omega_j(\xi)$.

\begin{coro}\label{coro1}
By \eqref{eq20} we obtain that $\sum_{j,k}q_jq_k\|\omega_j\|_{\beta^\prime}^2\|\omega_k\|_{\beta^\prime}^2<\infty$ on a set of full measure.
\end{coro}

\begin{theorem}\label{tex4}
Suppose that
\begin{equation}\label{neu54}
\sum_{i}\lambda_i^{2\gamma-2\delta}<\infty
\end{equation}
where $\gamma+\beta^\prime>1$. Then the mapping
\begin{align*}
  ((s,t),E) & \in \Delta_{0,1}\times L_2(V,V_\delta)\mapsto \int_s^t\bigg(A\int_s^\xi S(\xi-r)E(\omega(r)-\omega(s))dr\bigg)\otimes_Vd\omega(\xi)
\end{align*}
is in $C_{2\beta^\prime}(\Delta_{0,1};L_2(L_2(V,V_\delta),V\otimes V))$.
\end{theorem}
\begin{proof}
For $\omega\in C_{\beta^\prime}([0,1];V)$, thanks to Pazy \cite{Pazy} Theorem 4.3.5 (iii),
\begin{align}\label{neu51}
\begin{split}
  \bigg|A&\int_s^\xi S(\xi-r)E_{ij}(\omega(r)-\omega(s))dr\bigg|
=\frac{q_j^\frac12}{\lambda_i^{\delta}}\bigg|\int_s^\xi AS(\xi-r)e_i(\omega_j(r)-\omega_j(s))dr\bigg|\\
  &\le c\frac{q_j^\frac12}{\lambda_i^{\delta}}\|\omega_j\|_{\beta^\prime}(\xi-s)^{\beta^\prime},
\end{split}
\end{align}
and applying Bensoussan and Frehse \cite{BenFre00} Corollary 2.1 we also have
\begin{align}\label{neu52}
\begin{split}
  \bigg|A&\int_s^\xi S(\xi-r)E_{ij}(\omega(r)-\omega(s))dr-A\int_s^{\xi^\prime} S(\xi^\prime-r)E_{ij}(\omega(r)-\omega(s))dr\bigg|\\
  &=\frac{q_j^\frac12}{\lambda_i^{\delta}}\bigg|\int_s^\xi AS(\xi-r)e_i(\omega_j(r)-\omega_j(s))dr-\int_s^{\xi^\prime} AS(\xi^\prime-r)e_i(\omega_j(r)-\omega_j(s))dr\bigg|\\
  &\le \frac{q_j^\frac12}{\lambda_i^{\delta-\gamma}}\bigg|\int_s^\xi (-A)^{1-\gamma}S(\xi-r)e_i(\omega_j(r)-\omega_j(s))dr-\int_s^{\xi^\prime} (-A)^{1-\gamma}S(\xi^\prime-r)e_i(\omega_j(r)-\omega_j(s))dr\bigg|\\&\le c\frac{q_j^\frac12}{\lambda_i^{\delta-\gamma}}\|\omega_j\|_{\beta^\prime}|\xi-\xi^\prime|^\gamma.
\end{split}
\end{align}
Therefore, for an $\alpha < \gamma$ such that $\beta^\prime >1-\alpha $, which exists thanks to the assumption $\gamma+\beta^\prime >1$, we can define the integral
\begin{align*}
\int_s^t&\bigg(A\int_s^\xi S(\xi-r)e_i(\omega_j(r)-\omega_j(s))dr\bigg) d\omega_k(\xi)=\int_s^tD_{s+}^\alpha\bigg( A \int_s^\cdot S(\cdot-r)e_i(\omega_j(r)-\omega_j(s))dr\bigg)[\xi]D_{t-}^{1-\alpha}\omega_{k,t-}[\xi]d\xi,
\end{align*}
since, as a consequence of \eqref{neu51} and \eqref{neu52}, and because $\gamma >\beta^\prime$,
\begin{align*}
&\bigg|D_{s+}^\alpha\bigg(A \int_s^\cdot S(\cdot-r)e_i(\omega_j(r)-\omega_j(s))dr\bigg)[ \xi]\bigg|\le c\|\omega_j\|_{\beta^\prime}(
  (\xi-s)^{\beta^\prime-\alpha}+(\xi-s)^{\gamma-\alpha}) \le c\|\omega_j\|_{\beta^\prime}
  (\xi-s)^{\beta^\prime-\alpha}.
  \end{align*}

We remind that $|D_{t-}^{1-\alpha}\omega_{k,t-}[\xi]|\le c\|\omega_k\|_{\beta^\prime}(t-\xi)^{\alpha+\beta^\prime-1}$, hence, by Corollary \ref{coro1},
\begin{align*}
  \bigg( & \sum_{i,j,k} \bigg( \int_s^t\bigg(\frac{q_j^\frac12}{\lambda_i^{\delta}}\int_s^\xi AS(\xi-r)e_i(\omega_j(r)-\omega_j(s))dr\bigg)\otimes_V q_k^\frac12d\omega_k(\xi )\bigg)^2\bigg)^\frac12
  \\
  &\le \bigg(\sum_i\lambda^{2\gamma-2\delta}_i\bigg)^\frac12 \bigg(\sum_{j,k}q_jq_k\|\omega_j\|_{\beta^\prime}^2\|\omega_k\|_{\beta^\prime}^2\bigg)^\frac12(t-s)^{2\beta^\prime}<\infty.
\end{align*}
\end{proof}
\begin{remark}
Replacing in the above proof $\omega$ by $\omega^n-\omega$ we obtain the following convergence
\begin{equation*}
\lim_{n\to\infty}\int_s^tA\int_s^\xi S(\xi-r)E(\omega^n(r)-\omega^n(s))dr\otimes_V d\omega^n(\xi)=\int_s^tA\int_s^\xi S(\xi-r)E(\omega(r)-\omega(s))dr\otimes_V d\omega(\xi)
\end{equation*}
in $C_{2\beta^\prime}(\Delta_{0,1};L_2(L_2(V,V_\delta),V\otimes V))$.
Indeed, in the proof of Theorem \ref{tex3} we have shown that
\begin{equation*}
  \lim_{n\to\infty}\sum_{i,j}q_i q_i(\|\omega_i-\omega_i^n\|_{\beta^\prime}^2\|\omega_j\|_{\beta^\prime}^2+\|\omega_j-\omega_j^n\|_{\beta^\prime}^2\|\omega_i^n\|_{\beta^\prime}^2)=0\quad \text{a.s.}\\
\end{equation*}
\end{remark}

Finally, Theorem \ref{tex3} and Theorem \ref{tex4} ensure that $E(\omega^n\otimes_S\omega^n)$ is a Cauchy sequence having a limit, and this limit is given by \eqref{neu53}. \\

\subsection{Examples}

 Let us consider two examples of operators $G$ satisfying the  assumptions of Theorem \ref{tex2}:
\smallskip

We start by considering some  lattice system. Let $V=l_2$  be the space
of square additive sequences with values in $\RR$. In addition, let
$A$ be a negative symmetric operator defined on $D(-A)\subset l_2$
with compact inverse. In particular, we can assume that $-A$ has a
discrete spectrum $0<\lambda_1\le \lambda_2\le\cdots\le
\lambda_i\le \cdots\to\infty$  where the associated
eigenelements $(e_i)_{i\in\NN}$ form a complete orthonormal system
in $l_2$. The
spaces $D((-A)^\nu)$ are then defined by
\begin{equation*}
    \{u=(u_i)_{i\in\NN}\in l_2: \sum_i \lambda_i^{2\nu}
    u_i^2=:|u|_{V_\nu}^2<\infty\}.
\end{equation*}
For $1/3<\beta^\prime<1/2$ we take $\delta$ such that according to Theorem \ref{rem10} $\delta+\beta^\prime>1$. We then consider the space $V_\delta$ for which we assume that for some $\gamma>1/2$ the condition
$\sum_i\lambda_i^{2\gamma-2\delta}<\infty$ holds, and hence \eqref{neu54} is satisfied.

Consider a sequence of functions
$(g_{ij})_{i,j\in \NN}$, with $g_{ij}:V\to \RR$, and define $G(u)$ for $u\in V$ by

\begin{equation*}
    G(u)v=\bigg(\sum_j
    g_{ij}(u)v_j\bigg)_{i\in\NN}\quad\text{for all }v\in V.
\end{equation*}
We assume that
\begin{align*}
\|G(u)\|_{L_2(V,V_\delta)}^2&=\sum_j |G(u) e_j|_{V_\delta}^2=\sum_j \bigg(\sum_i \lambda_i^{2\delta} (G(u) e_j)_i^2\bigg)=\sum_{i,j} \lambda_i^{2\delta} g_{ij}^2(u)\le c
\end{align*}
uniformly with respect to $u\in V$.

In addition, assume that $g_{ij}$ are four times  differentiable and their  derivatives  are uniformly bounded in the following way
\begin{align*}
&|Dg_{ij}(u)(e_k)|\le c_{g,1}^{ijk},\quad  |D^2g_{ij}(u)(e_k,h_1)|\le c_{g,2}^{ijk}|h_1|,\quad  |D^3g_{ij}(u)(e_k,h_1,h_2)|\le c_{g,3}^{ijk}|h_1||h_2|,\\
&|D^4g_{ij}(u)(e_k,h_1,h_2,h_3)|\le  c_{g,4}^{ijk}|h_1||h_2||h_3|\quad\text{for any }u\in V,
\end{align*}
and these bounds  satisfy
\begin{align*}
&\sum_{ijk}\lambda_i^{2\delta}(c_{g,1}^{ijk})^2<\infty,\quad  \sum_{ijk}\lambda_i^{2\delta}(c_{g,2}^{ijk})^2<\infty,\quad  \sum_{ijk}\lambda_i^{2\delta}(c_{g,3}^{ijk})^2<\infty,\sum_{ijk}\lambda_i^{2\delta}(c_{g,4}^{ijk})^2<\infty.
\end{align*}
To see for instance that $DG$ exists, note that by Taylor expansion
\begin{align*}
|g_{ij}(u+h)-g_{ij}(u)-Dg_{ij}(u)(h)|^2\le \frac12|  D^2g_{ij}(u+\eta h)(h,h)|^2
    \le (c_{g,2}^{ijk})^2|h|^2
\end{align*}
where $u,\, h\in V$ and $\eta\in [0,1]$.
In particular, we also note that
\begin{equation*}
    \sum_{j,k}\bigg|DG(u)(e_k,e_j)\bigg|_{V_\delta}^2\le  \sum_{ijk}\lambda_i^{2\delta}(c_{g,1}^{ijk})^2=:c_{DG}^2<\infty.
\end{equation*}
This condition ensures the Lipschitz
continuity of $G$ as well as the Hilbert-Schmidt property of   $DG$. Similarly, we obtain that $DG$ is also Lipschitz with respect  to the Hilbert-Schmidt norm. We also obtain the existence of  the second and third derivative. By the choice of  $\beta^\prime$ and $\delta$, the conditions on $G$ in Hypothesis {\bf H} hold.\\

Now, we consider the second example. Let us assume that $A$ is generated by the Laplacian on $\mathcal O =(0,1)$ with  homogenous Dirichlet boundary condition. $A$ with domain $D(-A)=H^2(\oO)\cap H_0^1(\oO)$ generates a semigroup in $L_2(\oO)$ .
Let $\rho=1/4+\eps$, $\eps>0$, small. Then $V:=D((-A)^\rho)$ consists of the Slobodetski spaces $H^{2\rho}(\mathcal O)$
satisfying the homogeneous boundary conditions, see Da Prato and Zabczyk \cite{DaPrato}, Page 401. In particular, the continuous embedding $V\subset C(\bar\oO)$ holds. In what follows we consider the restriction of the semigroup to $V$. Note that the inequalities \eqref{eq1} and \eqref{eq2} continue being true, and that $(\lambda_i^{-\rho}e_i)_{i\in \NN}$ is an    orthonormal basis of $V$ where  $(\lambda_i)_{i\in\NN}$
is the spectrum of $A$ and $(e_i)_{i\in \NN}$ are the associated eigenelements which are uniformly bounded in $L_\infty(\mathcal O)$.
The asymptotical behavior of the spectrum is given by $\lambda_i\sim i^2$.

\begin{lemma}
For $\mu\in (1,5/4)$
\begin{equation*}
  D((-A)^{\mu})= H^{2\mu}(\mathcal O)\cap H_0^1(\oO).
\end{equation*}
\end{lemma}
\begin{proof}
On $H^{2\mu}(\mathcal O)\cap H_0^1(\oO)$ we know that $A=\Delta_{HDBC}$  which is an isomorphism with range $H^{2\mu-2}(\mathcal O)$,
see Egorov and Shubin \cite{egorov} Page 124. In addition, $(-A)^{\mu-1}$ has the domain $H^{2\mu-2}(\mathcal O)$ if $\mu\in (1,5/4)$, , see Da Prato and Zabczyk \cite{DaPrato}, Page 401.
\end{proof}
Now for $1/3<\beta^\prime<1/2$ we take
\begin{equation*}
  \gamma>1-\beta^\prime,\qquad\delta=\gamma+\rho< 1.
\end{equation*}
That choice of $\gamma$ and $\delta$ ensures the condition (\ref{neu54}) of Theorem \ref{tex4}, since we have previously taken $\rho=1/4+\eps$.

Let $g$ be a four times differentiable  function on $\bar{\mathcal O}\times \RR$  which is zero on $\{0,1\}\times \RR$, such that all the corresponding derivatives ($g$ itself  included) are bounded. Define

\begin{equation*}
    G(u)(v)[x]=\int_{\mathcal O} g(x,u(y))v(y)dy\quad \text{for }u,\,v\in V.
\end{equation*}
Following Kantorowitsch and Akilow \cite{KA} Section XVII.3  it is not hard to prove that  $G$ is  three times continuously differentiable where the derivatives are given by
\begin{align*}
    DG(u)(v,h_1)[x]&=\int_{\mathcal O}D_2g(x,u)v(y)h_1(y)dy,\\
    D^2G(u)(v,h_1,h_2)[x]&=\int_{\mathcal O}D_2^2g(x,u)v(y)h_1(y)h_2(y)dy,\\
     D^3G(u)(v,h_1,h_2,h_3)[x]&=\int_{\mathcal O}D_2^3g(x,u)v(y)h_1(y)h_2(y)h_3(y)dy,
\end{align*}
for $v,\,h_1,\,h_2,\,h_3\in V$. We have that $G(u)(v),\, DG(u)(v,h_1),\, D^2G(u)(v,h_1,h_2),\,D^3G(u)(v,h_1,h_2,h_3) \in H^3(\oO)\cap H^1_0(\oO)\subset D((-A)^{\mu})\subset V_\delta$  where $\mu\in (1,5/4)$ (note that the image of $G$ should be contained in $D((-A)^{\rho+\delta})= D((-A)^{2\rho+\gamma})$, and with the choice we have done, $2\rho+\gamma \in (1,5/4)$. Let us check, for instance, that $D^3G(u)(v,h_1,h_2,h_3) \in D((-A)^\mu)\subset V_\delta$. By the
continuous embedding theorem we have that
\begin{align*}
     \int_{\mathcal O}&\bigg|\int_{\mathcal O}D^2_2D_1^kg(x,u(y)+h_3(y))v(y)h_1(y)h_2(y)-D^2_2D_1^kg(x,u(y))v(y)h_1(y)h_2(y)\\
&-D_2D_2^2 D_1^k g(x,u(y))v(y)h_1(y)h_2(y)h_3(y)dy\bigg|^2dx\le c\bigg(\int_{\mathcal O}|v(y)h_1(y)h_2(y)h_3(y)|dy\bigg)^2\\
&\le c^\prime|v|_{C}^2|h_1|_{C}^2|h_2|_{C}^2|h_3|_{C}^2\le c^{\prime\prime} |v|^2|h_1|^2|h_2|^2|h_3|^2\quad\text{for } k=1,2,3,
\end{align*}
where $c^\prime$ is a uniform bound for $|D_2^4D_1^kg(x,u)|^2|\oO|$.
This gives the differentiability of $D^2G(u)$ in $H^3(\oO)\cap H_0^1(\oO)$ and hence in $D((-A)^\mu)$ too.
Furthermore, $G(u)(v)[x],\cdots,D^3G(u)(v,h_1,h_2,h_3)[x]$ are zero for $x\in\{0,1\}$.

The Hilbert-Schmidt property of $DG(u)$ follows by
\begin{equation*}
     \sum_{i,j}\int_{\mathcal O}\bigg(\int_{\mathcal O}|D_2D_1^k g(x,u(y))\lambda_i^{-\rho}e_{i}(y)\lambda_j^{-\rho}e_{j}(y)dy\bigg)^2dx<c(\sum_i \lambda_i^{-2\rho})^2<\infty,
\end{equation*}
due to the  boundedness of $D_2 D_1^k g$. In the same manner we obtain that the other derivatives are Hilbert--Schmidt operators.

\medskip

These estimates allow us to apply Theorem  \ref{tex2}.

\section{Appendix: A priori estimates}

We start this section by considering a {\it modified
Beta--function}.

\begin{lemma}
\label{l1-m} The integral
\begin{equation*}
B_\nu^\eta(a,b):=\int_a^b(b-q)^\eta(q-a)^\nu
dq=c(b-a)^{\nu+\eta+1},\qquad c=c(\nu,\eta),
\end{equation*}
for $a<b$ and $\nu,\,\eta>-1$. In addition,
\begin{equation*}
\int_a^b B_\nu^\eta(a,r)(r-a)^\xi(b-r)^\mu
dr=c(b-a)^{\nu+\eta+\mu+\xi+2},\qquad c=c(\nu,\eta,\xi,\mu),
\end{equation*}
if $\nu+\eta+\xi>-2\,$ and $\, \nu,\eta,\mu>-1$.
\end{lemma}
The proofs of these equalities are standard from the definition of
the Beta--function.

When $\nu=0$, we denote the corresponding modified
Beta--function just by $B^\eta(a,b)$.

We now give some properties of the fractional derivatives
$D_{s+}^{2\alpha-1}$ and $\hat D_{s+}^{\alpha}$. We shall omit their proofs since they follow straightforwardly.

In the following $\hat V$ represents an abstract Hilbert space that will be determined in the proof of Lemma \ref{ex1} below.

\begin{lemma}
\label{l2-m} Let $[0,T]\ni r\mapsto Q(r)\in
L(V)$ be a mapping satisfying
\begin{equation*}
\|Q(r)-Q(q)\|_{L(V)}\le c_Q(r-q)^{\beta_1}\qquad
\text{for}\quad r>q,\quad -2\alpha+\beta_1>-1.
\end{equation*}
We also suppose that, for $0\leq s\leq r\leq T$,
\begin{equation*}
\sup_{q\in[s,r]}\|Q(q)\|_{L(V)}\le c_Q^\prime.
\end{equation*}
In addition, let $V\ni x\mapsto R(x)\in L_2(\hat V,V_\delta)$ be a continuously
differentiable function bounded by $c_R$.
The first
derivative of $R$
is supposed to be bounded by $c_{DR}$. Then, for $0\le s<r\le T$ and $%
-2\alpha+\beta>-1$, for $u \in C_{\beta}([0,T];V)$ we have
\begin{equation*}
|D_{s+}^{2\alpha-1}(Q(\cdot)R(u(\cdot)))[r]|\le c\bigg(\frac{%
c_Q^\prime c_R}{(r-s)^{2\alpha-1}}+
c_Qc_RB^{-2\alpha+\beta_1}(s,r)+c_Q^\prime c_{DR}
B^{-2\alpha+\beta}(s,r)\|u\|_{\beta}\bigg).
\end{equation*}
\end{lemma}

\begin{lemma}
\label{l3-m} Suppose that $2\beta>\alpha$. Let $Q(\cdot)$ be given in Lemma %
\ref{l2-m} such that $\beta_1> \alpha$ and let $R$ be
mapping from $V $ to $L_2(\hat V,V_\delta)$ such that
\begin{equation*}
\|R(x)-R(y)-DR(y)(x-y)\|_{L_2(\hat V,V_\delta)}\le c_{D^2 R}|x-y|^2 \qquad \text{for
} x,\,y\in V.
\end{equation*}
Then,  for $0\le s<r\le T$ and $u\in C_{\beta}([0,T];V)$ we have that
\begin{align*}
|\hat D_{s+}^\alpha&(Q(\cdot)R(u(\cdot)))[r]|\le c\bigg(%
\frac{c_Q^\prime\sup_{p\in [0,T]}\|R(u(p))\|_{L_2(\hat V,V_\delta)}}{(r-s)^\alpha} \\
&+c_QB^{-\alpha-1+\beta_1}(s,r)\sup_{p\in[0,T]}\|R(u(p))\|_{L_2(\hat V,V_\delta)}+c_Q^%
\prime c_{D^2R} B^{-\alpha-1+2\beta}(s,r)\|u\|_{\beta}^2\bigg).
\end{align*}
\end{lemma}

\begin{remark}
We notice that, in Lemma \ref{l2-m} and Lemma \ref{l3-m},
$c_Q$ and $c_Q^\prime$ denote expressions related with the
operator $Q$ but have not to be necessarily constants, they can depend on parameters.

We also want to point out that, in the two previous lemmata, we could have considered $Q$ such that $r\in [0,T] \mapsto Q(r)\in L(V_\delta,V)$. However, this assumption would give us no significant improvements to the estimates in Lemma \ref{ex1}, thus for the sake of easier presentation, we have assumed $r\in [0,T] \mapsto Q(r)\in L(V)$.

\end{remark}
Next, we prove the Lemma \ref{ex1}.

\begin{proof}
We start assuming that $u_0\in V_\delta$ for $1\geq \delta \geq \beta$.

By $\eqref{eq1}$ and $\eqref{eq2}$, for $0< s<t \leq T$ we have that
\begin{align*}
|(S(t)-S(s))u_0| \leq C s^{\delta-\beta}(t-s)^\beta |u_0|_{V_\delta},
\end{align*}
and then $\|S(\cdot)u_0\|_{\beta} \le C|u_0|_{V_\delta}T^{\delta-\beta}$.

Now, in order to complete the proof of
\eqref{ex4m}, we are going to estimate the $V$-norm
of the following expression:
\begin{align*}
& \int_0^tS(t-r)G(u(r))d\omega(r)
-\int_0^sS(s-r)G(u(r))d\omega(r)
\\
=& \int_s^t S(t-r)G(u(r)) d\omega(r)+ \int_0^s
(S(t-r)-S(s-r))G(u(r)) d\omega(r) \\
=:&A_{11}(s,t)+A_{12}(s,t)+A_{21}(0,s)+A_{22}(0,s)
\end{align*}
where $0< s<t \leq T$, $U=(u,v)\in W(0,T)$ and
\begin{align*}
A_{11}(s,t)&=(-1)^\alpha\int_s^t \hat D_{s+}^\alpha (
S(t-\cdot)G(u(\cdot)))[r]D_{t-}^{1-\alpha}\omega_{t-}[r]dr, \\
A_{12}(s,t)&=-(-1)^{2\alpha-1}\int_s^t D_{s+}^{2\alpha-1}(
S(t-\cdot)DG(u(\cdot)))[r] D_{t-}^{1-\alpha}\mathcal{D}
_{t-}^{1-\alpha}v[r]dr, \\
A_{21}(0,s)&=(-1)^\alpha\int_0^s \hat D_{0+}^\alpha
((S(t-\cdot)-S(s-\cdot))G(u(\cdot)))[r]D_{s-}^{1-\alpha}\omega_{s-}[r]dr, \\
A_{22}(0,s)&=-(-1)^{2\alpha-1}\int_0^s
D_{0+}^{2\alpha-1}((S(t-\cdot)-S(s-\cdot))DG(u(\cdot)))[r]
D_{s-}^{1-\alpha}\mathcal{D} _{s-}^{1-\alpha}v[r]dr.
\end{align*}
To do that, we use Lemma
\ref{l1-m} together with Lemma \ref{l2-m} or Lemma \ref{l3-m},
depending on  whether in the integrand  the fractional
derivative or the compensated fractional derivative appear. As examples, let us show here how to estimate $A_{12}$ and
$A_{21}$ (for the rest of cases, the reader could look at the
different values of parameters which are on the table \ref{table}
below).

For the term $A_{12}$ we make the following
identification in Lemma \ref{l2-m}:
\begin{align*}
&\hat V=V\otimes V,\; R=DG, \;c_{R}=c_{DG},\; \;c_{DR}=c_{D^2G},\;
\beta_1=\beta, \; Q(\cdot)= S(t-\cdot),\; c_Q=(t-r)^{-\beta} ,\;
c_Q^\prime=c.
\end{align*}
Therefore, using \eqref{q3} and applying then Lemma \ref{l1-m},
with $a=s$ and $b=t$, we have that
\begin{align*}
    |A_{12}(s,t)|\le&
    c\int_s^t\left(\frac{1}{(r-s)^{2\alpha-1}}+(t-r)^{-\beta}B^{-2\alpha+\beta}(s,r)\right.\\
    \qquad & \left.+B^{-2\alpha+\beta}(s,r)|||U|||\right)
    |||U||| (t-r)^{2\alpha+\beta+\beta^\prime-2}dr,
\end{align*}
 and hence
\begin{align*}
&|A_{12}(s,t)| \leq C (t-s)^{\beta^\prime+
\beta}(1+(t-s)^{\beta}|||U|||^2).
\end{align*}
Let us take $\alpha^\prime>\alpha$ such that $\alpha^\prime +\beta < \alpha +\beta^\prime$. By the second inequality of Lemma \ref{l0}, for the term $A_{21}$ we make the following
identification in Lemma \ref{l3-m}:
\begin{align*}
&\hat V=V,\;  R=G, \;c_{D^2R}=c_{D^2G},\; \beta_1=\alpha^\prime,\; \sup_{p\in
[0,T]}\|R(u(p))\|_{L_2(V,V_\delta)} \leq c_G, \\
& Q(\cdot)=(S(t-\cdot)-S(s-\cdot)),\; c_Q=\frac{(t-s)^{\beta}}{%
(s-r)^{\alpha^\prime+\beta}} ,\; c_Q^\prime=\frac{(t-s)^{\beta}}{%
(s-r)^{\beta}}.
\end{align*}
Therefore, using \eqref{q1} and applying Lemma \ref{l1-m}
with $a=0$ and $b=s$, we have
\begin{align*}
    |A_{21}(0,s)|&\le c\int_0^s\left(\frac{(t-s)^{\beta}}{(s-r)^{\beta}r^\alpha}
    +\frac{(t-s)^{\beta}B^{\alpha^\prime-\alpha-1}(0,r)}
    {(s-r)^{\alpha^\prime+\beta}}\right.\\
   \qquad & \left.+\frac{(t-s)^{\beta}B^{-\alpha-1+2\beta}(0,r)|||U|||^2}{(s-r)^{\beta}}\right)(s-r)^{\alpha+\beta^\prime-1}dr.
\end{align*}
Hence,
\begin{align*}
&|A_{21}(0,s)| \leq C (t-s)^\beta (s^{\beta^\prime-\beta} +s^{\beta+\beta^\prime}|||U||| ^2)=C (t-s)^\beta s^{\beta^\prime-\beta}(1 +s^{2\beta}|||U||| ^2),
\end{align*}
and thus
\begin{align*}
\sup_{0\leq s<t \leq T} \frac{|A_{21}(0,s)|}{(t-s)^\beta} \leq C T^{\beta^\prime-\beta}(1 +T^{2\beta}|||U||| ^2).
\end{align*}
Consequently, the previous estimates imply
\begin{equation*}
\bigg \|\int_0^\cdot S(\cdot-r) G(u(r))d\omega
\bigg\|_{\beta} \leq C T^{\beta^\prime-\beta}(1+T^{2\beta}|||U|||^2),
\end{equation*}
and thus \eqref{ex4m} is proved.

\begin{table}[h]
\vspace{0.4cm}
\begin{tabular}{|l|l|l|l|l|l|l|l|l|l|}
\hline & $a$ & $b$ & $\beta _{1}$ & $c_{Q}$ & $c_{Q}^{\prime }$ &
sup $R$ & $c_{R}$ & $c_{DR}$ \\ \hline
$A_{11}$ & $s$ & $t$ & $\beta  $ & $(t-r)^{-\beta}$ & $%
{c}$ & $c_{G}$ &  &    \\
\hline
$A_{12}$ & $s$ & $t$ & $\beta $ & $(t-r)^{-\beta  }$ & $%
c$ &  & $c_{DG}$ & $c_{D^{2}G}$\\ \hline
$A_{21}$ & $0$ & $s$ & $\alpha^\prime $ & $\frac{(t-s)^{\beta }}{%
(s-r)^{\alpha^\prime+\beta  }}$ & $\frac{(t-s)^{\beta
}}{(s-r)^{\beta }}$ & $c_{G}$ &  &    \\
\hline
$A_{22}$ & $0$ & $s$ & $\alpha^\prime $ & $\frac{(t-s)^{\beta }}{%
(s-r)^{\alpha^\prime+\beta  }}$ & $\frac{(t-s)^{\beta
}}{(s-r)^{\beta }}$ &  & $c_{DG}$ & $c_{D^{2}G}$  \\   \hline
\end{tabular}
\begin{center}\caption{Values of the different parameters appearing in the estimates of $A_{ij}$, for $i,j=1,2$.} \label{table}
\end{center}
\end{table}

(ii) We obtain clearly that $|S(T) u_0|_{V_\delta}\leq
|u_0|_{V_\delta}$. It is interesting to emphasize here that the constant in the previous estimate
is just 1, which is of importance in the proof of
Theorem \ref{tex2}.

Moreover, since the semigroup and the operator $A$ commute, we can write
\begin{align*}
(-A)^\delta&\int_0^tS(t-r)G(u(r))d\omega(r)=\int_0^t S(t-r)(-A)^\delta G(u(r)) d\omega(r).
\end{align*}
In this point, since $G$ takes values in
$L_2(V,V_\delta)$ we have that $(-A)^\delta G(\cdot)\in L_2(V)$, and $\|G(u(\cdot))\|_{L_2(V,V_\delta)}=\|(-A)^\delta G(u(\cdot))\|_{L_2(V)}$. Therefore, we can use the properties of Lemma \ref{l2} and the calculations in part (i) of this proof to conclude (ii).
\end{proof}

In the following two technical lemmata we collect some properties that are needed for the global existence of a solution of \eqref{equ1}-\eqref{equ2}. {In particular, the first result emphasizes that the Chen equality is preserved under the mild form of the solution.

\begin{lemma}\label{l8}
Let $U=(u,v)\in W(0,1)$ be a solution of \eqref{equ1}-\eqref{equ2} for a smooth $\omega$, then $(u\otimes \omega)$ satisfies the Chen equality.
\end{lemma}

\begin{proof}
A straightforward calculation  shows
\begin{equation*}
    u(\xi)-u(s)=\int_s^\xi S(\xi-r)G(u(r))d\omega(r)+S(\xi-s)u(s)-u(s),
\end{equation*}
and therefore
\begin{align*}
    (u\otimes\omega)(s,\tau)=&\int_s^\tau\bigg(\int_s^\xi S(\xi-r)G(u(r))d\omega(r)+S(\xi-s)u(s)-u(s)\bigg)\otimes_Vd\omega(\xi)\\
   (u\otimes \omega)(\tau,t)=&\int_\tau^t\bigg(\int_\tau^\xi S(\xi-r)G(u(r))d\omega(r)+S(\xi-\tau)u(\tau)-u(\tau)\bigg)\otimes_Vd\omega(\xi)\\
   (u\otimes\omega)(s,t)=&\int_s^t\bigg(\int_s^\xi S(\xi-r)G(u(\tau))d\omega(r)+S(\xi-s)u(s)-u(s)\bigg)\otimes_Vd\omega(\xi)\\
\end{align*}
We note that for $\xi\in (\tau,t)$, we have
\begin{align*}
S(\xi-\tau)u(\tau)=S(\xi-\tau)S(\tau-s) u(s)  + \int_s^\tau
S(\xi-\tau)S(\tau-r)G(u(r))d\omega(r),
\end{align*}
and therefore
\begin{align}\label{chen2}
\begin{split}
&\int_s^\tau S(\xi-s)u(s)\otimes_V d\omega(\xi)
+\int_\tau^t S(\xi-\tau)u(\tau)\otimes_V d\omega(\xi)
\\=&\int_s^t S(\xi-s) u(s)\otimes_V d\omega(\xi)+\int_\tau^t \int_s^\tau
S(\xi-r)G(u(r))d\omega(r)\otimes_V d\omega(\xi).
\end{split}
\end{align}
Moreover, the rectangular term in the Chen equality can
be written as
\begin{align*}
(u(\tau)-u(s))\otimes_V(\omega(t)-\omega(\tau))&=u(\tau)\otimes_V(\omega(t)-\omega(\tau))-u(s)\otimes_V(\omega(t)-\omega(s))
\\&-u(s)\otimes_V(\omega(s)-\omega(\tau)).
\end{align*}
Thus, combining the previous equality with \eqref{chen2}, we obtain the Chen equality for smooth $\omega$.
\end{proof}


\begin{lemma}\label{l81}
Suppose Hypothesis {\bf H} holds. Let $U=(u,v)\in W(0,1)$ be a solution of \eqref{equ1}-\eqref{equ2} with initial condition $u_0\in V_\delta$, and let $U_n=(u_n,u_n\otimes \omega^n)$ where $u_n$ is a solution of \eqref{eq0} and $u_n\otimes \omega^n$ is given by \eqref{eq11} for a sequence $(\omega^n)_{n\in\NN}$ of smooth trajectories given in Hypothesis {\bf H}, item (4), having also initial condition $u_0$.
Then on $W(0,1)$
\begin{equation*}
    \lim_{n\to\infty}|||U-U_n|||=0.
\end{equation*}
\end{lemma}

\begin{proof}
 The proof is quite similar to  Theorem \ref{tex2}. Because we assume {\em a priori} that there exists  a solution $U\in W(0,1)$ we have that
\begin{equation}\label{cnv}
    \rho_0=2\sup_{[0,1]}|u(t)|_{V_\delta}<\infty.
\end{equation}
Denote the solution of the equation \eqref{2eq} by $R$ for an appropriate $\Delta T_1\le 1$. Since the constant $C$
in this formula can be chosen continuously depending on $\omega$ and $(\omega\otimes_S\omega)$ we have  $|||U_n|||\le 2R$ if $n$ is chosen
sufficiently large. In addition, we choose a $\Delta T$ less than or equal to $\Delta T_1$ such that
\begin{equation*}
    C \Delta T^{\beta^\prime-\beta}(1+\Delta T^{2\beta}4R^2)<\frac12.
\end{equation*}
Using the notation we have introduced in front of Lemma \ref{ex3bis} we have
\begin{align}\label{6.6}
\begin{split}
    |||U_n-U|||&\le|||\tT(U,\omega,(\omega\otimes_S\omega),u_0)-\tT(U_n,\omega,(\omega\otimes_S\omega),u_0)|||\\
    &+
    |||\tT(U_n,\omega,(\omega\otimes_S\omega),u_0)-\tT(U_n,\omega^n,(\omega^n\otimes_S\omega^n),u_0)|||.
    \end{split}
\end{align}
Hence  by Lemmata \ref{ex3} and \ref{ex3bis} we have $\lim_{n\to\infty}|||U_n-U|||=0$ in $W(0,\Delta T\wedge 1)$. If $\Delta T<1$ we consider $[\Delta T,2\Delta T\wedge 1]$. First we show that the restriction of $U$ to
$[\Delta T,  2\Delta T\wedge 1]$, $\Delta_{\Delta T,  2\Delta T\wedge 1}$ with initial condition $u(\Delta T)$ solves
\eqref{equ1}-\eqref{equ2} in $W(\Delta T,2\Delta T\wedge 1)$.
Note that by Lemma \ref{ex3bis}
\begin{equation*}
    U=\lim_{n\to\infty}\tT(U,u_0,\omega^n,\omega^n\otimes_S\omega^n).
\end{equation*}
The limit exists also with respect to $W(\Delta T,2\Delta T\wedge 1)$ when we restrict the elements of the right and left hand side to $[\Delta T,2\Delta T\wedge 1],\Delta_{\Delta T,2\Delta T\wedge 1}$. Then we have the presentation
\begin{align*}
    \tT_1(U,u_0,\omega^n,\omega^n\otimes_S\omega^n)&=S(t)u_0+\int_0^t S(t-r)G(u(r))d\omega^n=S(t-\Delta T)\tT_1(U,u_0,\omega^n,\omega^n\otimes_S\omega^n)(\Delta T)\\
    &+\int_{\Delta T}^tS(t-r)G(u(r))d\omega^n.
\end{align*}
Hence writing this expression in the form of \eqref{equ1} we have for $n\to\infty$ by Lemma \ref{ex3bis}
\begin{equation*}
    u=\tT_1(U,u(\Delta T),\omega,\omega\otimes_S\omega)\quad\text{on } [\Delta T,  2\Delta T\wedge 1].
\end{equation*}
and trivially
\begin{equation*}
    v=\tT_2(U,u_0,\omega,\omega\otimes_S\omega)\quad\text{on } \Delta_{\Delta T,  2\Delta T\wedge 1}.
\end{equation*}
Hence the restriction of $U\in W(\Delta T,2\Delta T\wedge 1)$
satisfies the system \eqref{equ1}-\eqref{equ2} on $[\Delta T,  2\Delta T\wedge 1]$, $\Delta_{\Delta T,  2\Delta T\wedge 1}$ with initial condition $u(\Delta T)$. Now we show the convergence of the approximated solutions $U_n$ to $U$ in $W(\Delta T,2\Delta T\wedge 1)$. First note that because of Lemmata \ref{ex1} and \ref{ex2} we have that the restriction of $U$ belongs to the ball $B_{W(\Delta T,2\Delta T\wedge 1)}(0,R)$, and moreover it is not possible that there are solutions outside this ball. Therefore, by using a similar inequality than (\ref{6.6}), for $n$ sufficiently large we obtain that on $W(\Delta T,2\Delta T\wedge 1)$
\begin{align*}
    |||U_n-U|||&\le|||\tT(U,\omega,(\omega\otimes_S\omega),u_0)-\tT(U_n,\omega,(\omega\otimes_S\omega),u_0)|||\\
    &+
    |||\tT(U_n,\omega,(\omega\otimes_S\omega),u_0)-\tT(U_n,\omega^n,(\omega^n\otimes_S\omega^n),u_0)|||
    \\
    &\leq 1/2   |||U_n-U|||+
    |||\tT(U_n,\omega,(\omega\otimes_S\omega),u_0)-\tT(U_n,\omega^n,(\omega^n\otimes_S\omega^n),u_0)|||,
\end{align*}
where the last inequality is true due to the fact that the restriction of $U$ belongs to the ball $B_{W(\Delta T,2\Delta T\wedge 1)}(0,R)$ and $|u(\Delta T)-u_n(\Delta T)|_{V_\delta}$ is sufficiently small (similar to \eqref{eq30}). Now Lemma \ref{ex3bis} implies that $\lim_{n\to\infty}|||U_n-U|||=0$ in $W(\Delta T,2\Delta T\wedge 1)$. In a similar manner
we obtain the convergence on  $W((i-1)\Delta T,i\Delta T\wedge 1)$ and then by the concatenation argument we obtain the convergence in $W(0,1)$.
\end{proof}


\begin{thebibliography}{99}

\bibitem{AnhGre99}V.V Anh  and W. Grecksch,
A parabolic stochastic differential equation with fractional Brownian motion input.
\newblock {\em Statistics \& Probability Letters}, 41 (1999), 337--346.

\bibitem{BenFre00} A. Bensoussan and J. Frehse.
\newblock Local Solutions for Stochastic Navier Stokes Equations.
\newblock  {\em Mathematical Modelling and Numerical Analysis}, 34 (2000), no. 2, 241--273.


\bibitem{CarKloSchm03} T. ~Caraballo, P. ~Kloeden  and B. ~Schmalfu{\ss}.
\newblock Exponentially stable stationary solutions for stochastic
evolution equations and their perturbation.
\newblock  {\em Appl. Math. Optimization}, 50 (2004), no. 3, 183--207.

\bibitem{caruana} M. Caruana and P. Friz.
\newblock Partial differential equations driven by rough paths.
\newblock  {\em J. Differential Equations}, 247 (2009), no. 1, 140Ð173.

\bibitem{caruana2} M. Caruana, P. Friz and H. Oberhauser.
\newblock A (rough) pathwise approach to a class of non-linear stochastic partial differential equations.
\newblock  {\em Ann. Inst. H. PoincarŽ Anal. Non LinŽaire}, 28 (2011), no. 1, 27Ð46.


\bibitem{Chu02} I.D. Chueshov.
\newblock {\em Monotone Random Systems. Theory and Application}.
\newblock Lecture Notes in Mathematics, 1779,  Springer, Berlin-Heidelberg-New York,  2002.

\bibitem{coutin} L. Coutin, P. Friz and N. Victoir.
\newblock Good rough path sequences and applications to anticipating stochastic calculus.
\newblock  {\em Ann. Probab.}, 35 (2007), no. 3, 1172--1193.

\bibitem {DaPrato}G. Da Prato, J. Zabczyk, \textquotedblleft Stochastic
equations in infinite dimensions", Cambridge University Press, Cambridge, 1992.

\bibitem{DGT} A. Deya, M. Gubinelli, S. Tindel, Non-linear rough heat equations, \newblock {\em Probab. Theory Relat. Fields}, 153 (2012), 97-147.

\bibitem{DeNeTi10} A. Deya, A. Neuenkirch, S. Tindel,
A Milstein-type scheme without L\'evy area terms for SDES driven by fractional Brownian motion,
\newblock {\em Ann. Inst. H. Poincar\'e Probab. Statist.}, 48 (2012), no. 2, 518-550.


\bibitem{DLS}
J.~Duan, K.~Lu, and B.~Schmalfu{\ss}.
\newblock Invariant manifolds for stochastic partial
 differential equations.
\newblock {\em Ann. Prob.}, 31 (2003), no. 4, 2109--2135.

\bibitem{DLS2}
J.~Duan, K.~Lu, and B.~Schmalfu{\ss}.
\newblock Smooth stable and unstable manifolds for stochastic evolutionary equations.
\newblock {\em  J. Dynam. Differential Equations},  16 (2004), no. 4, 949--972.

\bibitem{egorov} Y.V. Egorov and M.A. Shubin.
{\em Foundations of the Classical Theory of Partial Differential Equations}, Encyclopaedia of Mathematical Sciences, 30. Springer, 1998.

\bibitem{FL} F. ~Flandoli and H. ~Lisei.
\newblock Stationary conjugation of flows for parabolic SPDEs with
multiplicative noise and some applications. \newblock {\em
Stochastic Anal. Appl.},  22 (2004), no. 6, 1385--1420.

\bibitem{FO11} P. ~Friz and H. ~Oberhauser.
\newblock On the splitting-up method for rough (partial) differential equations.
\newblock {\em
J. Differential Equations},  251 (2011), no. 2, 316--338.


\bibitem{FV10} P. Friz and N. Victoir. \newblock {\em Multidimensional Stochastic Processes as Rough Paths. Theory and
Applications}. \newblock Cambridge Studies of Advanced Mathematics
Vol. 120. Cambridge University Press, 2010.


\bibitem{GLS09} M.J. Garrido-Atienza, K. Lu and
B.~Schmalfuss, Random dynamical systems for stochastic partial
differential equations driven by a fractional Brownian motion,
\newblock {\em Discrete and continuous dynamical systems, series B}, 14 (2010), no. 2, 473-493.

\bibitem{GLS12note} M. J. Garrido-Atienza, K. Lu and B. Schmalfu{\ss}, Compensated Fractional Derivatives and Stochastic Evolution Equations, {\it Comptes Rendus Math\'ematique,}  350 (2012), no. 23--24, 1037--1042.



\bibitem{GuLeTin} M.Gubinelli, A. Lejay and
S. Tindel.
\newblock Young integrals and SPDEs.
\newblock {\em Potential Anal.}, 25 (2006), no. 4, 307--326.

\bibitem{GuTin} M.Gubinelli and
S. Tindel, Rough Evolution Equations,
\newblock {\em The Annals of Probability}, 38 (2010), no. 1, 1--75.



\bibitem{HinZah09}
M. Hinz and M. Z{\"a}hle, Gradient type noises II--Systems of stochastic partial
differential equations,
\newblock {\em Journal of Functional Analysis}, 256 (2009), 3192–-3235.



\bibitem{HuNu09} Y. Hu, D. Nualart.
\newblock Rough path analysis via fractional calculus. {\em Trans. Amer. Math. Soc.}, 361 (2009), no. 5, 2689--2718.

\bibitem{KadRin97}
R.V. Kadison and J.R. Ringrose.
\newblock {\em Fundamentals of the Theory of Operator Algebras: Elementary theory}.
\newblock Graduate Studies in mathematics, AMS, 1997.

\bibitem{KA}
L.W. Kantorowitsch and  G. P. Akilow.
\newblock {\em Funktionalanalysis in normierten {R}\"aumen},
\newblock Verlag Harri Deutsch, 1978.

\bibitem{Kunita90}
H.~Kunita.
\newblock {\em Stochastic Flows and Stochastic Differential Equations}.
\newblock Cambridge University Press, 1990.


\bibitem{Lejay}
A. Lejay.
\newblock {\em An introduction to rough paths.} S\'eminaire de Probabilit\'es XXXVII, 1Ð59, Lecture Notes in Math., 1832, Springer, Berlin, 2003.

\bibitem{Lyons} T. Lyons and Z. Qian, {\em  System control and rough paths. Oxford Mathematical Monographs}, Oxford Science Publications. Oxford University Press, Oxford, (2002).

\bibitem{MasNua03}
B. Maslowski and D. Nualart.
\newblock Evolution equations driven by a fractional {B}rownian motion.
\newblock {\em J. Funct. Anal.}, 202(1):277--305, 2003.

\bibitem{MZZ} S. ~ Mohammed, T. ~ Zhang, and H. ~Zhao.
\newblock The stable manifold theorem for semilinear SPDEs,
to appear, Memoirs of AMS.


\bibitem{NuaRas02}
D. Nualart and A. R{\u{a}}{\c{s}}canu.
\newblock Differential equations driven by fractional {B}rownian motion.
\newblock {\em Collect. Math.}, 53(1):55--81, 2002.

\bibitem{Pazy}
A. Pazy.
\newblock {\em Semigroups of Linear Operators and Applications to Partial Differential Equations}.
\newblock Springer Applied Mathematical Series. Springer-Verlag, Berlin, 1983.

\bibitem{Samko}
S.G. Samko, A.A. Kilbas and O.I. Marichev.
\newblock {\em Fractional integrals and derivatives: Theory and applications}.
\newblock Gordon and Breach Science Publishers (Switzerland and Philadelphia, Pa., USA), 1993.


\bibitem{Sohr}
H. Sohr.
\newblock {\em The Navier-Stokes equations. An elementary functional analytic approach}.
\newblock Birkh\"auser Advances Texts, Birkh\"auser Varlag, Basel-Boston-Berlin, 2001.

\bibitem{Tmann}
J. Teichmann.
\newblock Another Approach to some Rough and Stochastic Partial Differential Equations.
Stochastics and Dynamics, 11 (2011), no 2--3, 535--550.

\bibitem{U}
J. Unterberger.
\newblock Stochastic calculus for fractional Brownian motion with Hurst exponent $H>1/4$: a rough path method by analytic extension.
\newblock {\em Ann. Prob. Probab.}, 37(2):565--614, 2009.

\bibitem{Zah98}
M.~Z{\"a}hle.
\newblock Integration with respect to fractal functions and stochastic
  calculus. {I}.
\newblock {\em Probab. Theory Related Fields}, 111(3):333--374, 1998.


\end{thebibliography}
\end{document}